\documentclass[a4paper,12pt,draft]{article}
\usepackage{amsmath,amssymb,enumerate,color}
\usepackage{amsthm,color}
\usepackage{txfonts}
\usepackage{bigints}
\usepackage{comment}
\usepackage{tabularx}
\usepackage{longtable}
\usepackage{mathtools}
\usepackage{tikz}
\usepackage{ulem}
\usepackage{here}
\usetikzlibrary{positioning, intersections, calc, arrows.meta,math, patterns}

\newcommand{\R}{\mathbb{R}}

\newcommand{\N}{\mathbb{N}}

\newtheorem{thm}{Theorem}[section]
\newtheorem{lemm}[thm]{Lemma}
\newtheorem{prop}[thm]{Proposition}
\newtheorem{cor}[thm]{Corollary}
\theoremstyle{remark}
\newtheorem{rem}{Remark}[section]
\theoremstyle{definition}

\newtheorem{defn}{Definition}[section]

\numberwithin{equation}{section}

\makeatletter
\def\@cite#1#2{[{{\bfseries #1}\if@tempswa , #2\fi}]}
\makeatother                  
\begin{document}
\begin{center}
\Large{{\bf
On solitary wave solutions with two-frequency parameters
to the three-component system of quadratic nonlinear Schr\"odinger equations}}
\end{center}

\vspace{5pt}

\begin{center}
Hiroyuki Hirayama%
\footnote{
Faculty of Education, University of Miyazaki, 1-1, Gakuenkibanadai-nishi, Miyazaki, 889-2192 Japan,  
E-mail:\ {\tt h.hirayama@miyazaki-u.ac.jp}}
and 
Masahiro Ikeda%
\footnote{
Graduate School of Information Science and Technology, The
University of Osaka, 1-5, Yamadaoka, Suita-shi, Osaka
565-0871, Japan, 
E-mail:\ {\tt ikeda@ist.osaka-u.ac.jp}}
\end{center}

\newenvironment{summary}{\vspace{.5\baselineskip}\begin{list}{}{%
     \setlength{\baselineskip}{0.85\baselineskip}
     \setlength{\topsep}{0pt}
     \setlength{\leftmargin}{12mm}
     \setlength{\rightmargin}{12mm}
     \setlength{\listparindent}{0mm}
     \setlength{\itemindent}{\listparindent}
     \setlength{\parsep}{0pt}
     \item\relax}}{\end{list}\vspace{.5\baselineskip}}
{\footnotesize {\bf Abstract.}}
In the present paper, 
we consider the Cauchy problem of 
a system of three nonlinear Schr\"odinger equations 
with quadratic nonlinearity. 
We first prove the existence of ground states in the form of solitary wave solutions with two frequency parameters. 
We then show that the conditions on the frequency parameters for the existence of ground states depend on the resonance structure of the system.

Next, we give the two results for global solutions. 
The first is an improvement of global well-posedness 
for initial data below the ground state threshold. 
The second is an improvement of global well-posedness for oscillating initial data. 
We also prove the orbital stability of the ground state sets 
with small speed parameter.


\section{Introduction}

We consider the Cauchy problem of the 
system of nonlinear Schr\"odinger equations:
\begin{equation}\label{dnls}
\begin{cases}
\left(i\partial_{t}+\dfrac{1}{\sigma_1} \Delta \right)u_1=-\overline{u_2}u_3,\ \ t\in \R,\ x\in \R^d,\\
\left(i\partial_{t}+\dfrac{1}{\sigma_2} \Delta \right)u_2=-\overline{u_1}u_3,\ \ t\in \R,\ x\in \R^d,\\
\left(i\partial_{t}+\dfrac{1}{\sigma_3} \Delta \right)u_3=-u_1u_2,\ \ t\in \R,\ x\in \R^d,\\
    (u_1,u_2,u_3)|_{t=0}=(u_{1,0},u_{2,0},u_{3,0}),\ \ x\in \R^d,
\end{cases}
\end{equation}
where $\partial_t:=\partial/\partial t$,\ 
$\partial_k:=\partial/\partial x_k$,\ 
$\nabla :=(\partial_1,\cdots,\partial_d)$,\ 
$\Delta :=\sum_{k=1}^d\partial_k^2$, 
$\sigma_1,\sigma_2,\sigma_3 \in \R\setminus \{0\}$ are constants, 
and the unknown functions $u_1,u_2,u_3$ are 
complex-valued. 
The system (\ref{dnls}) appears in the laser--plasma interaction 
(see \cite{CC04, CCO09}). 
We put
\[
F_1(U):=-\overline{u_2}u_3,\ \ 
F_2(U):=-\overline{u_1}u_3,\ \ 
F_3(U):=-u_1u_2
\]
for $U=(u_1,u_2,u_3)$. 
Then, (\ref{dnls}) can be written as
\begin{equation}\label{dnls2}
\begin{cases}
i\partial_tu_j+\dfrac{1}{\sigma_j}\Delta u_j=F_j(U),\ t>0,\ x\in \R^d,\ j=1,2,3,\\
U|_{t=0}=U_0,
\end{cases}
\end{equation}
where $U_0=(u_{1,0},u_{2,0},u_{3,0})$. 
The system (\ref{dnls2}) has conserved quantities 
$Q_1$, $Q_2$, $E$, and $P$ defined by
\[
\begin{split}
    Q_1(U)&:=\frac{1}{2}\left(\|u_1\|_{L^2(\R^d)}^2+\|u_3\|_{L^2(\R^d)}^2\right),\ \ 
    Q_2(U):=\frac{1}{2}\left(\|u_2\|_{L^2(\R^d)}^2+\|u_3\|_{L^2(\R^d)}^2\right),\\
    E(U)&:=L(U)+N(U),\ \ L(U):=\frac{1}{2}\sum_{j=1}^3\frac{1}{\sigma_j}\|\nabla u_j\|_{L^2(\R^d)}^2,\ \ N(U):=-{\rm Re}\left(u_1u_2,u_3\right)_{L^2(\R^d)},\\
    P(U)&:=(P_1(U),\cdots,P_d(U)),\ \ 
    P_k(U):=\frac{1}{2}\sum_{j=1}^3{\rm Re}(i\partial_ku_j,u_j)_{L^2(\R^d)}. 
\end{split}
\]
Namely, if $U$ is a smooth solution to (\ref{dnls2}) on $[0,T)$, 
then it holds that
\[
\frac{d}{dt}Q_1(U(t))=0,\ \frac{d}{dt}Q_2(U(t))=0, \frac{d}{dt}E(U(t))=0, 
\frac{d}{dt}P(U(t))=0
\]
for any $t\in (0,T)$. In particular, the energy space of (\ref{dnls2}) is
\[
\mathcal{H}^1(\R^d):=H^1(\R^d)\times H^1(\R^d) \times H^1(\R^d),
\]
where $H^1(\R^d)$ denotes the $L^2$-based Sobolev space. 

The system (\ref{dnls2}) is invariant under the scaling transformation
\[
U_{\lambda}(t,x):=\lambda^{-2}U(\lambda^{-2}t,\lambda^{-1}x),\ \ \lambda>0 
\]
and the scaling critical Sobolev exponent is $s_c=\frac{d}{2}-2$. 
We note that (\ref{dnls2}) is $L^2$-critical when $d=4$, 
and $H^1$-critical when $d=6$.
\subsection{Previous works}
We first introduce the results by Noguera and Pastor. 
In \cite{NP21, NP22},  
they considered the nonlinear Schr\"odinger system
\begin{equation}\label{nlssg}
i\alpha_j\partial_tu_j+\gamma_j\Delta u_j-\beta_ju_j=-f_j(u_1,\cdots,u_l),\ t>0,\ x\in \R^d,\ j=1,\cdots,l
\end{equation}
with general quadratic nonlinearities $f_j$. 
The system (\ref{dnls2}) is a special case of (\ref{nlssg})
with 
\[
l=3,\ 
\alpha_j=1,\ \gamma_j=\frac{1}{\sigma_j},\ \beta_j=0,\ f_j=-F_j. 
\]
Here, we focus on the results for the system (\ref{dnls2}). 
For $1\le d\le 5$, 
Noguera and Pastor proved the existence of standing wave solutions to (\ref{dnls2}) of the form of
\[
U(t,x)=(e^{i\omega t}\varphi_1,e^{i\omega t}\varphi_2,e^{2i\omega t}\varphi_3)
\]
with $\omega >0$, where $\Phi =(\varphi_1,\varphi_2,\varphi_3)$ is 
a ground state solution to the elliptic system
\begin{equation}\label{ellip1}
-\frac{1}{\sigma_j}\Delta\varphi_j+\omega_j\varphi_j=-F_j(\Phi),\ x\in \R^d,\ j=1,2,3 
\end{equation}
with $(\omega_1,\omega_2,\omega_3)=(\omega,\omega,2\omega)$. 
They also obtained the following well-posedness results in the energy space. 
\begin{thm}[\cite{NP21, NP22}]\label{Theorem_NP}
Let $\Phi$ be a ground state solution to {\rm (\ref{ellip1})} with $\omega =1$. 
Put $Q:=Q_1+Q_2$. 
\begin{itemize}
    \item[{\rm (i)}] If $1\le d\le 6$, then 
    {\rm (\ref{dnls2})} is locally well-posed in $\mathcal{H}^1(\R^d)$. 
    \item[{\rm (ii)}] If $1\le d\le 3$, then 
    {\rm (\ref{dnls2})} is globally well-posed in $\mathcal{H}^1(\R^d)$. 
    \item[{\rm (iii)}] If $d=4$, then 
    {\rm (\ref{dnls2})} is globally well-posed in $\mathcal{H}^1(\R^d)$ 
    for initial data $U_0$ satisfying
    \[
    Q(U_0)<Q(\Phi).
    \]
    \item[{\rm (iv)}] If $d=5$, then 
    {\rm (\ref{dnls2})} is globally well-posed in $\mathcal{H}^1(\R^d)$ 
    for initial data $U_0$ satisfying
    \[
    Q(U_0)E(U_0)<Q(\Phi)E(\Phi)
    \]
    and
    \[
    Q(U_0)L(U_0)<Q(\Phi)L(\Phi).
    \]
\end{itemize}
\end{thm}
\begin{rem}
\begin{itemize}
\item[{\rm (i)}] It is also proved in \cite{NP21} 
that $Q=Q_1+Q_2$ and $E$ are conserved quantities 
for the solution to {\rm (\ref{dnls2})} constructed 
in {\rm Theorem~\ref{Theorem_NP}\ (i)}. 
The conservation laws for $Q_1$ and $Q_2$ are obtained in a similar way 
to the conservation law for $Q$. 
\item[{\rm (ii)}] The local well-posedness in $L^2(\R^d)\times L^2(\R^d)\times L^2(\R^d)$ is also established when $1\le d\le 4$ in \cite{NP21}. 
The solution can be extended globally in time when $1\le d\le 3$ by 
using the conservation law for $Q$. 
\item[{\rm (iii)}] For $d=6$, existence of the ground states with $\omega =0$ is also obtained in \cite{NP22}. 
\item[{\rm (iv)}] There are also the results on blow-up of solution and stability/instability of ground state in \cite{NP21}, \cite{NP22}, and scattering of solution in \cite{NP21_2}, \cite{NP21_3}. 
\end{itemize}
\end{rem}
The elliptic system {\rm (\ref{ellip1})} with $\sigma_1=\sigma_2=\sigma_3=1$ 
was also studied in \cite{FLYY26} as a limiting system of
\begin{equation}\label{ellip1_power}
-\Delta\phi_j+\lambda_j\phi_j=|\phi_j|^{p-2}\phi_j-\alpha F_j(\phi_1,\phi_2,\phi_3),\ x\in \R^d,\ j=1,2,3.  
\end{equation}
More precisely, when $2+4/d<p<2^*$ and
$\lambda_j=\alpha^{-\frac{4}{4-d}}\omega_j$, by setting
\[
\varphi_j(x)=\alpha^{\frac{d}{4-d}}\phi_j(\alpha^{\frac{2}{4-d}}x)
\]
and letting $\alpha \rightarrow +0$, 
the system (\ref{ellip1_power}) formally converges to the system {\rm (\ref{ellip1})}, 
where $2^*=\infty$ if $d=1,2$ and $2^*=2d/(d-2)$ if $d\ge 3$. 
In \cite{FLYY26}, the existence of ground states for the mass constraint minimizing problem 
\[
E(\Phi)=\inf\{E(U)|\ U\in \mathcal{H}^1(\R^d),\ Q_1(U)=a_1^2,\ Q_2(U)=a_2^2\}
\]
was established for any $a_1,a_2>0$ when $1\le d\le 3$ and $\sigma_1=\sigma_2=\sigma_3=1$. 
Such a ground state satisfies the system {\rm (\ref{ellip1})} for some $\omega_1,\omega_2$. 
In this paper, we will give the existence of ground states under the Nehari constraint (see, Theorem~\ref{ex_gs_1} below). 
For the ground state of the elliptic system (\ref{ellip1_power}), 
there are many previous results (\cite{Ardila18, KM26, KO22, Pomponio10, Wang17}). 

The two-component system
\begin{equation}\label{two_NLSS}
\begin{cases}
    i\partial_tu+\Delta u=-2\overline{u}v,\ \ t>0,\ x\in \R^d,\\
    i\partial_tv+\kappa \Delta v=-u^2,\ \ t>0,\ x\in \R^d
\end{cases}
\end{equation}
is also one of the typical examples of (\ref{nlssg}). 
For (\ref{two_NLSS}), there are many results 
for the local/global well-posedness, 
existence of standing wave solutions, 
blow-up of solution, and scattering of solution (\cite{Hapre, HIN21, HM21, HOT13, IKN19, NP21_4}). 

When the constants $\sigma_1,\sigma_2,\sigma_3$ satisfy
\begin{equation}\label{massres}
\sigma_1+\sigma_2=\sigma_3, 
\end{equation}
which is called the ``mass resonance'' condition, 
the system {\rm (\ref{dnls2})} is invariant under the 
Galilean transformation
\[
u_{c,j}(t,x)= e^{-i\frac{\sigma_j}{4}|c|^2t}e^{i\frac{\sigma_j}{2}c\cdot x}u_j(t,x-ct),\ \ j=1,2,3.
\]
Then, the traveling wave solution to {\rm (\ref{dnls2})} 
can be obtained as the Galilean transform of the standing wave solution. 
On the other hand, if the mass resonance condition {\rm (\ref{massres})} 
is not satisfied, then the existence of traveling wave solutions is not clear. 
For the two-component system (\ref{two_NLSS}), 
Fukaya, Hayashi, and Inui (\cite{FHI24}) proved the existence of 
traveling wave solutions of the form of
\[
(u(t,x),v(t,x))=(e^{i\omega t}\phi (x-ct), e^{2i\omega t}\psi (x-ct))
\]
with $(\omega,c)\in \R\times \R^d$ satisfying the following conditions:
\[
\begin{dcases}
\omega >\max\left\{\dfrac{|c|^2}{4},\dfrac{|c|^2}{8\kappa}\right\}&{\rm if}\ 1\le d\le 5\ {\rm and}\ \kappa >0\\
\omega =\dfrac{|c|^2}{8\kappa}&{\rm if}\ 3\le d\le 5,\ c\ne 0,\ {\rm and}\ 0<\kappa <\dfrac{1}{2},\\
\omega=\dfrac{|c|^2}{4}&{\rm if}\ d=5,\ c\ne 0,\ {\rm and}\ \kappa >\dfrac{1}{2}.
\end{dcases}
\]
We note that the system (\ref{two_NLSS}) 
is a reduction of the system (\ref{dnls2}) such as
\[
\sigma_1=\sigma_2=1,\ \sigma_3=\frac{1}{\kappa},\ u_1=u_2=\sqrt{2}u,\ u_3=2v,
\]
and then, the mass resonance condition {\rm (\ref{massres})} 
can be written $\kappa =\frac{1}{2}$. 
For $d=4$, they also obtained the global well-posedness of {\rm (\ref{two_NLSS})} 
for oscillating initial data $(u,v)|_{t=0}=(e^{\frac{i}{2}c\cdot x}u_0,e^{\frac{i}{2\kappa}c\cdot x}v_0)$ with large enough $|c|$ (which depends on $u_0$, $v_0$) 
under the smallness condition 
$\|u_0\|_{L^2}^2< A_0$ if $0<\kappa <\frac{1}{2}$ 
and $\|v_0\|_{L^2}^2< B_0$ if $\kappa >\frac{1}{2}$ 
for some universal constants $A_0$, $B_0>0$. 
Analogous properties for the system (\ref{dnls2}) were considered by Li in \cite{Lipre}, 
and the existence of traveling wave solutions to (\ref{dnls2}) of the form of
\begin{equation}\label{tws_li}
U(t,x)=(e^{i\omega t}\varphi_1(x-ct),e^{i\omega t}\varphi_2(x-ct),e^{2i\omega t}\varphi_3(x-ct))
\end{equation}
with $(\omega,c)\in \R\times \R^d$ is obtained under the following conditions:
\[
\begin{dcases}
\omega >\max\left\{\dfrac{\sigma_1|c|^2}{4},\dfrac{\sigma_2|c|^2}{4},\dfrac{\sigma_3|c|^2}{8}\right\}&{\rm if}\ 1\le d\le 5\ {\rm and}\ \sigma_1,\sigma_2,\sigma_3 >0\\
\omega =\max\left\{\dfrac{\sigma_1|c|^2}{4},\dfrac{\sigma_2|c|^2}{4},\dfrac{\sigma_3|c|^2}{8}\right\}&{\rm if}\ 3\le d\le 5,\ c\ne 0,\ {\rm and}\ 0<\sigma_3^*\le \sigma_2^*<\sigma_1^*,\\
\omega=\max\left\{\dfrac{\sigma_1|c|^2}{4},\dfrac{\sigma_2|c|^2}{4},\dfrac{\sigma_3|c|^2}{8}\right\}&{\rm if}\ d=5,\ c\ne 0,\ {\rm and}\ 0<\sigma_3^*<\sigma_2^*=\sigma_1^*,
\end{dcases}
\]
where $\sigma_1^*,\sigma_2^*,\sigma_3^*$ are the rearrangement 
of $\sigma_1,\sigma_2,\frac{\sigma_3}{2}$ such that $\sigma_1^*\ge \sigma_2^*\ge \sigma_3^*$. 
We note that $\sigma_1=\sigma_2=\frac{\sigma_3}{2}$ does not occur 
if the mass resonance condition (\ref{massres}) is not satisfied. 
\begin{rem}
    The functions $\varphi_1,\varphi_2,\varphi_3$ in {\rm (\ref{tws_li})} satisfy the elliptic system
    \begin{equation}\label{ellip_sys_tra}
    -\frac{1}{\sigma_j}\Delta\varphi_j+\omega_j\varphi_j+i(c\cdot \nabla)\varphi_j=-F_j(\Phi),\ x\in \R^d,\ j=1,2,3
    \end{equation}
    with $(\omega_1,\omega_2,\omega_3)=(\omega,\omega,2\omega)$. 
    Li also shows that there is no nontrivial solution to this system 
    when $1\le d\le 5$, $\sigma_1=\sigma_2=\frac{\sigma_3}{2}$, 
    and $\omega=\frac{\sigma_1|c|^2}{4}$. 
\end{rem}
Global well-posedness for oscillating initial data was also obtained as follows. 
\begin{thm}[\cite{Lipre}]\label{Theorem_Li}
Let $d=4$, $\sigma_1,\sigma_2,\sigma_3>0$, $c\in \R^4\setminus \{0\}$, $V_0=(v_{1,0},v_{2,0},v_{3,0})\in \mathcal{H}^1(\R^4)$, and put
\[
    U_0(x)=(e^{i\frac{\sigma_1}{2}c\cdot x}v_{1,0},e^{i\frac{\sigma_2}{2}c\cdot x}v_{2,0},e^{i\frac{\sigma_3}{2}c\cdot x}v_{3,0}). 
    \]
\begin{itemize}
    \item[{\rm (i)}] Assume $\sigma_1+\sigma_2<\sigma_3$. 
    There exists $A_0>0$ such that if $\max\{\|v_{1,0}\|_{L^2}^2,\|v_{2,0}\|_{L^2}^2\}<A_0$ 
    and $c\in \R^4$ is large enough (which depends on $V_0$), 
    then the local solution to (\ref{dnls2}) can be extended globally in time.
    \item[{\rm (ii)}] Assume $\sigma_1+\sigma_2>\sigma_3$ and $\sigma_1<\sigma_2$. 
    There exists $B_0>0$ such that if $\max\{\|v_{1,0}\|_{L^2}^2,\|v_{3,0}\|_{L^2}^2\}<B_0$ 
    and $c\in \R^4$ is large enough (which depends on $V_0$), 
    then the local solution to (\ref{dnls2}) can be extended globally in time. 
    \item[{\rm (iii)}] Assume $\sigma_1+\sigma_2>\sigma_3$ and $\sigma_1>\sigma_2$. 
    There exists $C_0>0$ such that if $\max\{\|v_{2,0}\|_{L^2}^2,\|v_{3,0}\|_{L^2}^2\}<C_0$ 
    and $c\in \R^4$ is large enough (which depends on $V_0$), 
    then the local solution to (\ref{dnls2}) can be extended globally in time. 
    \item[{\rm (iv)}] Assume $\sigma_1+\sigma_2>\sigma_3$ and $\sigma_1=\sigma_2$. 
    There exists $D_0>0$ such that if $\|v_{3,0}\|_{L^2}^2<D_0$ 
    and $c\in \R^4$ is large enough (which depends on $V_0$), 
    then the local solution to (\ref{dnls2}) can be extended globally in time. 
\end{itemize}
\end{thm}
Theorem~\ref{Theorem_Li} says that the global well-posedness 
for oscillating initial data can be obtained 
even if one of $v_{1,0}$,$v_{2,0}$ and $v_{3,0}$ is not small. 
In particular, when $\sigma_1=\sigma_2$, 
the smallness assumption is required only for $v_{3,0}$. 

There are also results for
the three component system with derivative nonlinearities. 
In \cite{HI24}, the authors considered the system
\begin{equation}\label{dNLS_sys}
\begin{cases}
    i\partial_tu_1+\alpha\Delta u_1=-(\nabla \cdot u_3)u_2,\ \ t>0,\ x\in \R^d,\\
    i\partial_tu_2+\beta\Delta u_2=-(\nabla \cdot \overline{u_3})u_1,\ \ t>0,\ x\in \R^d,\\
    i\partial_tu_3+\gamma \Delta u_3=\nabla (u_1\cdot \overline{u_2}),\ \ t>0,\ x\in \R^d
\end{cases}
\end{equation}
and obtained the results for
the existence of ground states and stability of ground state sets. 
\subsection{Main results}
For two-frequency parameters $\omega_1,\omega_2>0$ and 
a velocity parameter $c\in \R^d$, 
we define the action functional $S_{\omega_1,\omega_2,c}$ by
\[
S_{\omega_1,\omega_2,c}(U):=E(U)+\omega_1Q_1(U)+\omega_2Q_2(U)+c\cdot P(U). 
\]
A key point of the present paper is to 
treat the two-frequency parameters $\omega_1$ and $\omega_2$ independently.
Put
\[
\omega_3:=\omega_1+\omega_2. 
\]
A direct calculation shows that
\[
\begin{split}
&\langle D_jS_{\omega_1,\omega_2,c}(\Phi),\theta_j\rangle\\
&={\rm Re}\left(-\frac{1}{\sigma_j}\Delta \varphi_j+\omega_j\varphi_j+i(c\cdot \nabla)\varphi_j+F_j(\Phi),\theta_j\right)_{L^2(\R^d)}\\
&={\rm Re}\left(-\frac{1}{\sigma_j}\Delta (e^{-i\frac{\sigma_j}{2}c\cdot x}\varphi_j)+\left(\omega_j-\frac{\sigma_j}{4}|c|^2\right)e^{-i\frac{\sigma_j}{2}c\cdot x}\varphi_j+e^{-i\frac{\sigma_j}{2}c\cdot x}F_j(\Phi),e^{-i\frac{\sigma_j}{2}c\cdot x}\theta_j\right)_{L^2(\R^d)},\ \ \ j=1,2,3
\end{split}
\]
for smooth $\Phi=(\varphi_1,\varphi_2,\varphi_3)$ and smooth $\Theta =(\theta_1,\theta_2,\theta_3)$, where 
$D_jS_{\omega_1,\omega_2,c}$ denotes the derivative of $S_{\omega_1,\omega_2,c}$ with respect to the $j$-th component defined by
\[
\begin{split}
\langle D_jS_{\omega_1,\omega_2,c}(\Phi),\theta\rangle&:=
\lim_{\epsilon \rightarrow 0}\frac{S_{\omega_1,\omega_2,c}(\Phi +\epsilon E_j(\theta))-S_{\omega_1,\omega_2,c}(\Phi)}{\epsilon},\\
E_1(\theta)&=(\theta, 0,0),\ \ 
E_2(\theta)=(0,\theta, 0),\ \ 
E_3(\theta)=(0,0,\theta). 
\end{split}
\]
We define the function space
\[
X_{\omega_1,\omega_2,c}
=X_1\times X_2\times X_3
\]
by
\[
X_j=X_j(\omega_1,\omega_2,c)
:=
\begin{dcases}
    e^{i\frac{\sigma_j}{2}c\cdot (\cdot)}H^1(\R^d)&{\rm if}\ \omega_j-\frac{\sigma_j}{4}|c|^2>0,\\
    e^{i\frac{\sigma_j}{2}c\cdot (\cdot)}\dot{H}^1(\R^d)&{\rm if}\ \omega_j-\frac{\sigma_j}{4}|c|^2=0. 
\end{dcases}
\]
The norm on $X_{\omega_1,\omega_2,c}$ will be defined below (see, Remark~\ref{rel_X_tilX}). 
In particular, $X_j=H^1(\R^d)$ holds if $\omega_j-\frac{\sigma_j}{4}|c|^2>0$ 
(see, Proposition~\ref{norm_X_H1_equi} and Remark~\ref{norm_eq_X_H_rem}). 
\begin{defn}
    We say that $\Phi=(\varphi_1,\varphi_2,\varphi_3)\in X_{\omega_1,\omega_2,c}$ 
    is a weak solution to (\ref{ellip_sys_tra}) if 
    \[
    \frac{1}{\sigma_j}\left(\nabla (e^{-i\frac{\sigma_j}{2}c\cdot x}\varphi_j),\nabla (e^{-i\frac{\sigma_j}{2}c\cdot x}\theta_j)\right)_{L^2(\R^d)}
    +\left(\left(\omega_j -\frac{\sigma_j}{4}|c|^2\right)\varphi_j+F_j(\Phi),\theta_j\right)_{L^2(\R^d)}
    =0
    \]
    holds for any $\Theta =(\theta_1,\theta_2,\theta_3)\in X_{\omega_1,\omega_2,c}$. 
\end{defn}
\begin{rem}
When $X_{\omega_1,\omega_2,c}=\mathcal{H}^1(\R^d)$, 
$\Phi =(\varphi_1,\varphi_2,\varphi_3)\in \mathcal{H}^1(\R^d)$ is a weak solution to (\ref{ellip_sys_tra})
if and only if
\[
\frac{1}{\sigma_j}\left(\nabla \varphi_j,\nabla\theta_j\right)_{L^2(\R^d)}+\left(\omega_j\varphi_j+i(c\cdot \nabla )\varphi_j+F_j(\Phi),\theta_j\right)_{L^2(\R^d)}=0
\]
holds for any $\Theta =(\theta_1,\theta_2,\theta_3)\in \mathcal{H}^1(\R^d)$. 
\end{rem}
Because the smooth solution $\Phi =(\varphi_1,\varphi_2,\varphi_3)$ to the elliptic system (\ref{ellip_sys_tra}) satisfies
\begin{equation}\label{S_derive_0}
D_1S_{\omega_1,\omega_2,c}(\Phi)
=D_2S_{\omega_1,\omega_2,c}(\Phi)
=D_3S_{\omega_1,\omega_2,c}(\Phi)
=0,
\end{equation}
we write $S_{\omega_1,\omega_2,c}'(\Phi)=0$ 
when $\Phi$ is a weak solution to (\ref{ellip_sys_tra}). 
For a weak solution $\Phi$ to (\ref{ellip_sys_tra}), we put
\begin{equation}\label{2req_para_sol}
U(t,x)=(e^{i\omega_1t}\varphi_1(x-ct),e^{i\omega_2t}\varphi_2(x-ct) ,e^{i(\omega_1+\omega_2)t}\varphi_3(x-ct)).
\end{equation}
Then $U$ satisfies the Schr\"odinger system (\ref{dnls2}). 
It is a solitary wave solution with two-frequency parameters to (\ref{dnls2}). 

For $c\in \R^d$, we define 
the operator $\Lambda_c$ as
\begin{equation}\label{op_lambdac_def}
(\Lambda_c\Psi)(x):=(e^{i\frac{\sigma_1}{2}c\cdot x}\psi_1(x),e^{i\frac{\sigma_2}{2}c\cdot x}\psi_2(x),e^{i\frac{\sigma_3}{2}c\cdot x}\psi_3(x)). 
\end{equation}
We introduce the transformation $\Phi=\Lambda_c\Psi$. 
Namely, $\varphi_j(x)=e^{i\frac{\sigma_j}{2}c\cdot x}\psi_j(x)$\ $(j=1,2,3)$. 
If $\Phi=(\varphi_1,\varphi_2,\varphi_3)$ is a smooth solution to (\ref{ellip_sys_tra}), 
then $\Psi =(\psi_1,\psi_2,\psi_3)$ satisfies
\begin{equation}\label{ellip_sys2}
    -\frac{1}{\sigma_j}\Delta \psi_j+\left(\omega_j-\frac{\sigma_j}{4}|c|^2\right)\psi_j
    =-G_{j,c}(\Psi),\ x\in \R^d,\ j=1,2,3, 
\end{equation}
where $G_{j,c}(\Psi)=e^{-i\frac{\sigma_j}{2}c\cdot x}F_j(\Lambda_c\Psi)$. Namely, 
\[
G_{1,c}(\Psi)=-e^{-i\frac{\sigma_1+\sigma_2-\sigma_3}{2}c\cdot x}\overline{\psi_2}\psi_3,\ \ 
G_{2,c}(\Psi)=-e^{-i\frac{\sigma_1+\sigma_2-\sigma_3}{2}c\cdot x}\overline{\psi_1}\psi_3,\ \ 
G_{3,c}(\Psi)=-e^{i\frac{\sigma_1+\sigma_2-\sigma_3}{2}c\cdot x}\psi_1\psi_2. 
\]
Moreover, we define the functional $\widetilde{S}_{\omega_1,\omega_2,c}$ by
\[
\widetilde{S}_{\omega_1,\omega_2,c}(\Psi):=S_{\omega_1,\omega_2,c}(\Phi)
=S_{\omega_1,\omega_2,c}(\Lambda_c \Psi). 
\]
Then, we have
\begin{equation}\label{S_phi_psi}
\begin{split}
\widetilde{S}_{\omega_1,\omega_2,c}(\Psi)
&=\frac{1}{2}\sum_{j=1}^3\frac{1}{\sigma_j}\|\nabla \psi_j\|_{L^2(\R^d)}^2
+\frac{1}{2}\sum_{j=1}^3\left(\omega_j-\frac{\sigma_j}{4}|c|^2\right)\|\psi_j\|_{L^2(\R^d)}^2+N_c(\Psi),
\end{split}
\end{equation}
where
\[
N_c(\Psi)=-{\rm Re}\left(e^{i\frac{\sigma_1+\sigma_2-\sigma_3}{2}c\cdot x}\psi_1\psi_2,\psi_3\right)_{L^2(\R^d)}.
\]
We note that if
\[
\sigma_j>0\ {\rm and}\ 
\omega_j-\frac{\sigma_j}{4}|c|^2\ge 0,\ j=1,2,3
\]
holds, 
then the summation of quadratic terms with respect to $\psi_j$ and $\nabla \psi_j$\ $(j=1,2,3)$  
in the right-hand side of (\ref{S_phi_psi}) is non-negative definite. 
In particular, when $\omega_j=\frac{\sigma_j}{4}|c|^2$, 
then the second term of the left-hand-side of (\ref{ellip_sys2}) and 
the term of $L^2$-norm of $\psi_j$ in the right-hand side of (\ref{S_phi_psi}) 
disappear. 
In such a case, $\psi_j\in H^1(\R^d)$ is not guaranteed 
because we cannot control the $L^2$-norm of $\psi_j$ 
by using $\widetilde{S}_{\omega_1,\omega_2,c}(\Psi)$. 
This situation is called the ``zero mass'' case. 
Handling the zero mass cases is difficult 
because we have to use the Sobolev embedding for $\dot{H}^1(\R^d)$ 
instead of for $H^1(\R^d)$. 
We will also treat the zero mass cases. 
However, not all cases are covered (see, Theorem~\ref{ex_gs_1} below).

By the relation (\ref{S_phi_psi}), 
we consider the action functional $\widetilde{S}_{\omega_1,\omega_2,c}$ 
on the class
\[
\widetilde{X}_{\omega_1,\omega_2,c}
=\widetilde{X}_1\times \widetilde{X}_2\times \widetilde{X}_3
\]
equipped with the norm
\[
\|\Psi\|_{\widetilde{X}_{\omega_1,\omega_2,c}}^2
:=\sum_{j=1}^3\frac{1}{\sigma_j}\|\nabla \psi_j\|_{L^2(\R^d)}^2+\sum_{j=1}^3\left(\omega_j-\frac{\sigma_j}{4}|c|^2\right)\|\psi_j\|_{L^2(\R^d)}^2,
\]
where $\widetilde{X}_j=\widetilde{X}_j(\omega_1,\omega_2,c)$ is defined by
\[
\widetilde{X}_j:=
\begin{dcases}
    H^1(\R^d)&{\rm if}\ \omega_j-\frac{\sigma_j}{4}|c|^2>0,\\
    \dot{H}^1(\R^d)&{\rm if}\ \omega_j-\frac{\sigma_j}{4}|c|^2=0. 
\end{dcases}
\]
\begin{rem}\label{rel_X_tilX}
We note that $X_{\omega_1,\omega_2,c}=\Lambda_c\widetilde{X}_{\omega_1,\omega_2,c}$ holds
and $\Psi \in \widetilde{X}_{\omega_1,\omega_2,c}$ is equivalent to 
$\Phi =\Lambda_c\Psi\in X_{\omega_1,\omega_2,c}$. 
We define the norm $\|\cdot\|_{X_{\omega_1,\omega_2,c}}$ by
$\|\Phi\|_{X_{\omega_1,\omega_2,c}}:=\|\Lambda_{-c}\Phi\|_{\widetilde{X}_{\omega_1,\omega_2,c}}$.
\end{rem}
\begin{prop}\label{norm_X_H1_equi}
    If $\omega_j>\frac{\sigma_j}{4}|c|^2$\ $(j=1,2,3)$, then it holds 
    \[
    \|\Phi\|_{X_{\omega_1,\omega_2,c}}^2\ge 
    \max\left\{\sum_{j=1}^3\frac{1}{\omega_j\sigma_j}\left(\omega_j-\frac{\sigma_j}{4}|c|^2\right)\|\nabla \varphi_j\|_{L^2}^2,\ \ 
    \sum_{j=1}^3\left(\omega_j-\frac{\sigma_j}{4}|c|^2\right)\|\varphi_j\|_{L^2}^2\right\}
    \]
    for any $\Phi=(\varphi_1,\varphi_2,\varphi_3)\in X_{\omega_1,\omega_2,c}$. 
    In particular, there exists $C=C(\omega_1,\omega_2)>0$ such that 
    $\|\Phi\|_{X_{\omega_1,\omega_2,c}}^2\ge C\min_{1\le j\le 3}(\omega_j-\frac{\sigma_j}{4}|c|^2)\|\Phi\|_{\mathcal{H}^1}^2$. 
\end{prop}
\begin{proof}
    Because $\nabla (e^{-i\frac{\sigma_j}{2}c\cdot x}\varphi_j)
    =e^{-i\frac{\sigma_j}{2}c\cdot x}(\nabla \varphi_j-i\frac{\sigma_j}{2}c\varphi_j)$, 
    we have
    \begin{equation}\label{norm_X_form}
    \begin{split}
    \|\Phi\|_{X_{\omega_1,\omega_2,c}}^2
    &=\sum_{j=1}^3\frac{1}{\sigma_j}\left\|\nabla \varphi_j-i\frac{\sigma_j}{2}c\varphi_j\right\|_{L^2}^2
    +\sum_{j=1}^3\left(\omega_j-\frac{\sigma_j}{4}|c|^2\right)\|\varphi_j\|_{L^2}^2
    \end{split}
    \end{equation}
    by the definitions of $\|\cdot\|_{\widetilde{X}_{\omega_1,\omega_2,c}}$ and $\|\cdot\|_{X_{\omega_1,\omega_2,c}}$. 
    Therefore, it suffices to show that
    \[
    \|\Phi\|_{X_{\omega_1,\omega_2,c}}^2\ge \sum_{j=1}^3\frac{1}{\omega_j\sigma_j}\left(\omega_j-\frac{\sigma_j}{4}|c|^2\right)\|\nabla \varphi_j\|_{L^2}^2. 
    \]
    It follows from
    \[
    \begin{split}
    \left\|\nabla \varphi_j-i\frac{\sigma_j}{2}c\varphi_j\right\|_{L^2}^2
    &\ge \|\nabla \varphi_j\|_{L^2}^2-\sigma_j|c|\|\nabla \varphi_j\|_{L^2}\|\varphi_j\|_{L^2}+\frac{\sigma_j^2}{4}|c|^2\|\varphi_j\|_{L^2}^2\\
    &\ge \frac{1}{\omega_j}\left(\omega_j-\frac{\sigma_j}{4}|c|^2\right)\|\nabla \varphi_j\|_{L^2}^2
    -\sigma_j\left(\omega_j-\frac{\sigma_j}{4}|c|^2\right)\|\varphi_j\|_{L^2}^2, 
    \end{split}
    \]
    where we used the Young inequality in the second inequality. 
\end{proof}
\begin{rem}\label{norm_eq_X_H_rem}
    It is clear that $\|\Phi\|_{X_{\omega_1,\omega_2,c}}^2\le C\|\Phi\|_{\mathcal{H}^1}^2$ 
    holds for some $C>0$ by (\ref{norm_X_form}). 
    Therefore, the norms $\|\cdot\|_{X_{\omega_1,\omega_2,c}}$ and $\|\cdot \|_{\mathcal{H}^1}$ 
    are equivalent when $\omega_j-\frac{\sigma_j}{4}|c|^2>0$\ $(j=1,2,3)$. 
\end{rem} 
\begin{rem}\label{X_norm_rep_LQP}
We can see that
\[
\|\Phi\|_{X_{\omega_1,\omega_2,c}}^2
=2L(\Phi)+2\omega_1Q_1(\Phi)+2\omega_2Q_2(\Phi)+2c\cdot P(\Phi)
\]
by (\ref{norm_X_form}) 
and the definitions of $L$, $Q_1$, $Q_2$, and $P$. 
\end{rem}
\begin{defn}
    We say that $\Psi=(\psi_1,\psi_2,\psi_3)\in \widetilde{X}_{\omega_1,\omega_2,c}$ 
    is a weak solution to (\ref{ellip_sys2}) if 
    \[
    \frac{1}{\sigma_j}\left(\nabla \psi_j,\nabla \theta_j\right)_{L^2(\R^d)}
    +\left(\left(\omega_j -\frac{\sigma_j|c|^2}{4}\right)\psi_j+G_{j,c}(\Psi),\theta_j\right)_{L^2(\R^d)}
    =0
    \]
    holds for any $\Theta=(\theta_1,\theta_2,\theta_3) \in \widetilde{X}_{\omega_1,\omega_2,c}$, 
    and then we write $\widetilde{S}_{\omega_1,\omega_2,c}'(\Psi)=0$. 
\end{defn}
\begin{rem}
    If $\Phi=\Lambda_c\Psi$, 
    then $\widetilde{S}_{\omega_1,\omega_2,c}'(\Psi)=0$ 
    is equivalent to $S_{\omega_1,\omega_2,c}'(\Phi)=0$.
\end{rem}

Next, we define the sets of weak solutions to (\ref{ellip_sys_tra}) 
and (\ref{ellip_sys2}) by
\[
\begin{split}
\mathcal{E}_{\omega_1,\omega_2,c}
&:=\left\{\Phi\in X_{\omega_1,\omega_2,c}
\left|
\ \Phi \ne (0,0,0),\ 
S_{\omega_1,\omega_2,c}'(\Phi)=0
\right.\right\},\\
\widetilde{\mathcal{E}}_{\omega_1,\omega_2,c}
&:=\left\{\Psi\in \widetilde{X}_{\omega_1,\omega_2,c}
\left|
\ \Psi \ne (0,0,0),\ 
\widetilde{S}_{\omega_1,\omega_2,c}'(\Psi)=0
\right.\right\}
\end{split}
\]
respectively. We also put 
\[
\begin{split}
\mathcal{G}_{\omega_1,\omega_2,c}
&:=
\left\{\Phi\in \mathcal{E}_{\omega_1,\omega_2,c}
\left|\ S_{\omega_1,\omega_2,c}(\Phi)\le S_{\omega_1,\omega_2,c}(\Theta)\ 
{\rm for\ any}\ \Theta \in \mathcal{E}_{\omega_1,\omega_2,c}
\right.\right\},\\
\widetilde{\mathcal{G}}_{\omega_1,\omega_2,c}
&:=
\left\{\Psi\in \widetilde{\mathcal{E}}_{\omega_1,\omega_2,c}
\left|\ \widetilde{S}_{\omega_1,\omega_2,c}(\Psi)\le \widetilde{S}_{\omega_1,\omega_2,c}(\Theta)\ 
{\rm for\ any}\ \Theta \in \widetilde{\mathcal{E}}_{\omega_1,\omega_2,c}
\right.\right\}. 
\end{split}
\]
We call an element $\Phi \in \mathcal{G}_{\omega_1,\omega_2,c}$ 
and an element $\Psi \in \widetilde{\mathcal{G}}_{\omega_1,\omega_2,c}$
a ``ground state'' solution to (\ref{ellip_sys_tra}) 
and (\ref{ellip_sys2}) respectively. 
The weak solutions to (\ref{ellip_sys_tra}) and (\ref{ellip_sys2}) 
can be characterized by Nehari functionals 
$K_{\omega_1,\omega_2,c}$ and $\widetilde{K}_{\omega_1,\omega_2,c}$ respectively
defined by
\begin{equation}\label{def_nehari}
\begin{split}
   K_{\omega_1,\omega_2,c}(\Phi)
   &:=\partial_{\lambda}S_{\omega_1,\omega_2,c}(\lambda \Phi)|_{\lambda=1}
   =2L(\Phi)+3N(\Phi)+2\omega_1Q_1(\Phi)+2\omega_2Q_2(\Phi)+2c\cdot P(\Phi),\\
   \widetilde{K}_{\omega_1,\omega_2,c}(\Psi)
   &:=\partial_{\lambda}\widetilde{S}_{\omega_1,\omega_2,c}(\lambda \Psi)|_{\lambda=1}
   =\|\Psi\|_{\widetilde{X}_{\omega_1,\omega_2,c}}^2+3N_c(\Psi).
\end{split}
\end{equation}
We note that $\widetilde{K}_{\omega_1,\omega_2,c}(\Psi)=K_{\omega_1,\omega_2,c}(\Lambda_c\Psi)$ and
\begin{equation}\label{S_K_N_rel}
\widetilde{S}_{\omega_1,\omega_2,c}(\Psi)=\frac{1}{2}\|\Psi\|_{\widetilde{X}_{\omega_1,\omega_2,c}}^2+N_c(\Psi)
=\frac{1}{3}\widetilde{K}_{\omega_1,\omega_2,c}(\Psi)+\frac{1}{6}\|\Psi\|_{\widetilde{X}_{\omega_1,\omega_2,c}}^2
\end{equation}
hold. 
\begin{rem}\label{nehari_weak_sol}
By the definition of the Nehari functionals, 
$\Phi\in \mathcal{E}_{\omega_1,\omega_2,c}$ 
and $\Psi\in \widetilde{\mathcal{E}}_{\omega_1,\omega_2,c}$ 
satisfy 
$K_{\omega_1,\omega_2,c}(\Phi)=0$ and $\widetilde{K}_{\omega_1,\omega_2,c}(\Psi)=0$ 
respectively. 
\end{rem}
To state our results, we consider the minimizing problems
\[
\begin{split}
    \mathfrak{m}_{\omega_1,\omega_2,c}
    &:=\inf\left\{S_{\omega_1,\omega_2,c}(\Phi)\left|\ \Phi \in X_{\omega_1,\omega_2,c},\ 
    \Phi\ne (0,0,0),\ K_{\omega_1,\omega_2,c}(\Phi)=0
    \right.\right\},\\
    \widetilde{\mathfrak{m}}_{\omega_1,\omega_2,c}
    &:=\inf\left\{\widetilde{S}_{\omega_1,\omega_2,c}(\Psi)\left|\ \Psi\in \widetilde{X}_{\omega_1,\omega_2,c},\ 
    \Psi\ne (0,0,0),\ \widetilde{K}_{\omega_1,\omega_2,c}(\Psi)=0
    \right.\right\},
\end{split}
\]
and define the sets of minimizers by
\[
\begin{split}
    \mathcal{M}_{\omega_1,\omega_2,c}
    &:=\left\{\Phi \in X_{\omega_1,\omega_2,c}\left|\ 
    \Phi \ne (0,0,0),\ S_{\omega_1,\omega_2,c}(\Phi)=\mathfrak{m}_{\omega_1,\omega_2,c},\ 
    K_{\omega_1,\omega_2,c}(\Phi)=0\right.\right\},\\
    \widetilde{\mathcal{M}}_{\omega_1,\omega_2,c}
    &:=\left\{\Psi \in \widetilde{X}_{\omega_1,\omega_2,c}\left|\ 
    \Psi \ne (0,0,0),\ \widetilde{S}_{\omega_1,\omega_2,c}(\Psi)=\widetilde{\mathfrak{m}}_{\omega_1,\omega_2,c},\ 
    \widetilde{K}_{\omega_1,\omega_2,c}(\Psi)=0\right.\right\}. 
\end{split}
\]
\begin{rem}\label{rel_m_tilm}
We note that $\mathfrak{m}_{\omega_1,\omega_2,c}=\widetilde{\mathfrak{m}}_{\omega_1,\omega_2,c}$ 
holds and $\Phi \in \mathcal{M}_{\omega_1,\omega_2,c}$ is 
equivalent to $\Psi \in \widetilde{\mathcal{M}}_{\omega_1,\omega_2,c}$ 
if $\Phi =\Lambda_c\Psi$. 
In particular, when $c=0$, we have $\mathcal{M}_{\omega_1,\omega_2,0}=\widetilde{\mathcal{M}}_{\omega_1,\omega_2,0}$. 
We also obtain
\[
\mathfrak{m}_{\omega_1,\omega_2,c}
=\widetilde{\mathfrak{m}}_{\omega_1,\omega_2,c}
=\frac{1}{6}\|\Lambda_{-c}\Phi\|_{\widetilde{X}_{\omega_1,\omega_2,c}}^2
=\frac{1}{6}\|\Phi\|_{X_{\omega_1,\omega_2,c}}^2
\]
for any $\Phi\in \mathcal{M}_{\omega_1,\omega_2,c}$ by (\ref{S_K_N_rel}). 
\end{rem}

We now state the main results of the present paper. 
For $c\in \R^d$, we put
    \[
    \Omega_c:=\left\{(\omega_1,\omega_2)\ \left|\ \omega_1>\frac{\sigma_1}{4}|c|^2,\ 
    \omega_2>\frac{\sigma_2}{4}|c|^2,\ \omega_3\ (=\omega_1+\omega_2)>\frac{\sigma_3}{4}|c|^2\right.\right\}
    \]
    and denote the boundary of $\Omega_c$ by $\partial \Omega_c$. For the shape of $\Omega_c$, see Figure~\ref{dom_omega} below. 
\begin{thm}[existence of ground states]\label{ex_gs_1}
    Let $\sigma_1,\sigma_2, \sigma_3>0$, and $c\in \R^d$.     
    Assume one of the following conditions holds:
    \begin{itemize}
        \item[{\rm (A)}] $1\le d\le 5$, $(\omega_1,\omega_2)\in \Omega_c$. 
        \item[{\rm (B)}] $3\le d\le 5$, $(\omega_1,\omega_2)\in \partial \Omega_c$, 
         and $(\omega_i,\omega_j)\ne (\frac{\sigma_i}{4}|c|^2,\frac{\sigma_j}{4}|c|^2)$ for any $1\le i<j\le 3$. 
         \item[{\rm (C)}] $d=5$, $c\ne 0$, $(\omega_1,\omega_2)\in \partial \Omega_c$, $(\omega_i,\omega_j)=(\frac{\sigma_i}{4}|c|^2,\frac{\sigma_j}{4}|c|^2)$ for some $1\le i<j\le 3$, 
         and 
         \[
         (\omega_1,\omega_2,\omega_3)\ne \left(\frac{\sigma_1}{4}|c|^2,\frac{\sigma_2}{4}|c|^2,\frac{\sigma_3}{4}|c|^2\right).
         \]
    \end{itemize}
    Then $\widetilde{\mathcal{G}}_{\omega_1,\omega_2,c}=\widetilde{\mathcal{M}}_{\omega_1,\omega_2,c}\ne \emptyset$\  
    $($and also $\mathcal{G}_{\omega_1,\omega_2,c}=\mathcal{M}_{\omega_1,\omega_2,c}\ne \emptyset )$. 
    Namely, there exists at least one ground-state solution to {\rm (\ref{ellip_sys2})}\ $($and also {\rm (\ref{ellip_sys_tra})}$)$. 
\end{thm}
\begin{rem}
\begin{itemize}
    \item[{\rm (i)}] The cases {\rm (B)} and {\rm (C)} in Theorem~\ref{ex_gs_1} 
correspond to zero-mass cases. 
\item[{\rm (ii)}] If the mass resonance condition {\rm (\ref{massres})} does not hold, then the simultaneous equality
\[
(\omega_1,\omega_2,\omega_3)= 
\left(\frac{\sigma_1}{4}|c|^2,\frac{\sigma_2}{4}|c|^2,\frac{\sigma_3}{4}|c|^2\right)
\]
is incompatible with $\omega_3=\omega_1+\omega_2$.
\item[{\rm (iii)}] We can also obtain that there is no nontrivial solution to the system 
{\rm (\ref{ellip_sys2})} when $1\le d\le 5$, $\sigma_1+\sigma_2=\sigma_3$, 
    and 
    \[
    (\omega_1,\omega_2,\omega_3)= \left(\frac{\sigma_1}{4}|c|^2,\frac{\sigma_2}{4}|c|^2,\frac{\sigma_3}{4}|c|^2\right). 
    \]
    For the proof, see \cite{Lipre}. 
\end{itemize}
\end{rem}
\begin{rem}
    Theorem~\ref{ex_gs_1} with $\omega_1=\omega_2$ contains the 
    results for the existence of standing wave solutions in \cite{NP21, NP22} 
    and the existence of traveling wave solutions in \cite{Lipre} as above. 
    By treating the frequency parameters $\omega_1$ and $\omega_2$ independently, 
    we can find more solitary wave solutions. 
\end{rem}
\begin{rem}
We can treat $c=0$ for the case {\rm (A)}. 
Furthermore, when $3\le d\le 5$, we can treat $c=0$ also for the case {\rm (B)}. 
Namely, we have $\mathcal{G}_{\omega_1,\omega_2,0}=\mathcal{M}_{\omega_1,\omega_2,0}\ne \emptyset$ if $\max\{\omega_1,\omega_2\}>0$ and $\min\{\omega_1,\omega_2\}\ge 0$ are satisfied. 
Moreover, 
\[
U(t,x)=(e^{i\omega_1t}\varphi_1,e^{i\omega_2t}\varphi_2,e^{i(\omega_1+\omega_2)t}\varphi_3)
\]
for $\Phi=(\varphi_1,\varphi_2,\varphi_3)\in \mathcal{M}_{\omega_1,\omega_2,0}$ 
becomes a standing wave solution with two-frequency parameters to {\rm (\ref{dnls2})}. 
In particular, the case $\min\{\omega_1,\omega_2\}=0$ 
is not treated in \cite{NP21, NP22}. 
A ground state $\Phi \in \mathcal{M}_{\omega_1,\omega_2,0}$ with $\min\{\omega_1,\omega_2\}=0$ 
plays an important role to improve the results for 
global well-posedness (see, Theorem~\ref{GWP_1} below). 
\end{rem}
\begin{rem}\label{pos_Q_1_Q_2_rem}
    For $\Psi =(\psi_1,\psi_2,\psi_3)\in \widetilde{\mathcal{M}}_{\omega_1,\omega_2,c}$, 
    we obtain $\psi_1\ne 0$, $\psi_2\ne 0$, and $\psi_3\ne 0$. 
    Indeed, if $\psi_1=0$, then $\psi_2$ and $\psi_3$ satisfy
    \[
    -\frac{1}{\sigma_j}\Delta \psi_j+\left(\omega_j-\frac{\sigma_j}{4}|c|^2\right)\psi_j=0,\ \ j=2,3
    \]
    because $\Psi$ is a solution to {\rm (\ref{ellip_sys2})}. 
    This implies
    \[
    \frac{1}{\sigma_j}\|\nabla \psi_j\|_{L^2}^2+\left(\omega_j-\frac{\sigma_j}{4}|c|^2\right)\|\psi_j\|_{L^2}^2=0,\ j=2,3.
    \]
    Namely, $\psi_2=\psi_3=0$. 
    This contradicts $\Psi \ne 0$. 
    Therefore, $\psi_1\ne 0$. 
    For the similar reason, we have $\psi_2\ne 0$ and $\psi_3\ne 0$. 
    We also obtain $\varphi_1\ne 0$, $\varphi_2\ne 0$, and $\varphi_3\ne 0$ 
    for $\Phi =(\varphi_1,\varphi_2,\varphi_3)\in \mathcal{M}_{\omega_1,\omega_2,c}$. 
    In particular, $Q_1(\Phi)>0$ and $Q_2(\Phi)>0$ hold. 
\end{rem}
For the case $\sigma_1+\sigma_2\ge \sigma_3$, the condition
\[
\omega_3-\frac{\sigma_3}{4}|c|^2>0
\]
follows from the condition
\begin{equation}\label{omega_range}
\omega_1> \frac{\sigma_1}{4}|c|^2\ \ {\rm and}\ \ 
\omega_2> \frac{\sigma_2}{4}|c|^2.
\end{equation}
Indeed, it holds
\[
\omega_3-\frac{\sigma_3}{4}|c|^2
=\left(\omega_1-\frac{\sigma_1}{4}|c|^2\right)
+\left(\omega_2-\frac{\sigma_2}{4}|c|^2\right)
+\frac{\sigma_1+\sigma_2-\sigma_3}{4}|c|^2. 
\]
Therefore, we can obtain the existence of solitary wave solution with two-frequency parameters of the form (\ref{2req_para_sol}) under the condition (\ref{omega_range}).
On the other hand, for the case $\sigma_1+\sigma_2<\sigma_3$, 
we need to impose an additional condition
\[
    \omega_1+\omega_2 > \frac{\sigma_3}{4}|c|^2. 
\]
Namely, the shape of the set $\Omega_c$ 
depends on the sign of $\mu:=\sigma_1+\sigma_2-\sigma_3$. 
(See, Figure~\ref{dom_omega} below.)
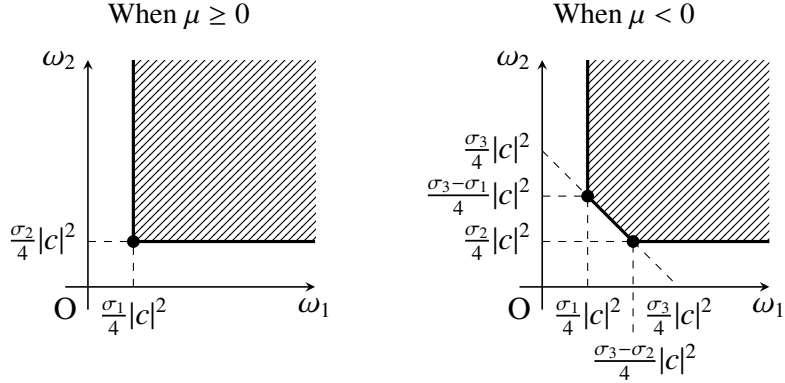
\begin{figure}[H]
\begin{center}
\begin{tikzpicture}[scale=0.6]
\draw[->,>=stealth,semithick](-0.5,0)--(5,0)node[below]{$\omega_1$};
\draw[->,>=stealth,semithick](0,-0.5)--(0,5)node[left]{$\omega_2$};
\draw(0,0)node[below left]{O};
\draw[dashed](1,0)--(1,2);
\draw[black, very thick](1,1)--(1,5);
\draw[dashed](0,1)--(2,1);
\draw[black, very thick](1,1)--(5,1);
\draw(0,1)node[left]{$\frac{\sigma_2}{4}|c|^2$};
\draw(1,0)node[below]{$\frac{\sigma_1}{4}|c|^2$};
\draw(2.25,-1.8)node[below]{};
\draw(2,6.5)node[below]{\small ${\rm When}\ \mu\ge 0$};
\fill (1,1) circle (4pt) coordinate (A);
\path [pattern=north east lines, pattern color=black] (1,1) rectangle (5, 5);

\draw[->,>=stealth,semithick](-0.5+10,0)--(5+10,0)node[below]{$\omega_1$};
\draw[->,>=stealth,semithick](0+10,-0.5)--(0+10,5)node[left]{$\omega_2$};
\draw(0+10,0)node[below left]{O};
\draw[black,very thick,samples=100,domain=1:2]plot(\x+10,3-\x);
\draw[black,dashed,samples=100,domain=0:1]plot(\x+10,3-\x);
\draw[black,dashed,samples=100,domain=2:3]plot(\x+10,3-\x);
\draw[dashed](1+10,0)--(1+10,2);
\draw[dashed](2+10,-1)--(2+10,1);
\draw[black, very thick](1+10,2)--(1+10,5);
\draw[dashed](0+10,1)--(2+10,1);
\draw[dashed](0+10,2)--(1+10,2);
\draw[black, very thick](2+10,1)--(5+10,1);
\draw(0+10,3)node[left]{$\frac{\sigma_3}{4}|c|^2$};
\draw(0+10,2)node[left]{$\frac{\sigma_3-\sigma_1}{4}|c|^2$};
\draw(0+10,1)node[left]{$\frac{\sigma_2}{4}|c|^2$};
\draw(1+10,0)node[below]{$\frac{\sigma_1}{4}|c|^2$};
\draw(2.25+10,-1)node[below]{$\frac{\sigma_3-\sigma_2}{4}|c|^2$};
\draw(3+10,0)node[below]{$\frac{\sigma_3}{4}|c|^2$};
\fill (1+10,2) circle (4pt) coordinate (A);
\fill (2+10,1) circle (4pt) coordinate (B);
\draw(1.8+10,6.5)node[below]{\small ${\rm When}\ \mu <0$};
\fill [pattern=north east lines] plot [smooth,samples=100,domain=1:2] (\x+10,{3-\x}) -- (5+10,1) -- (5+10,5) -- (1+10,5);
\end{tikzpicture}
\caption{The region $\Omega_c$ and boundary $\partial \Omega_c$ in the $(\omega_1,\omega_2)$-plane}\label{dom_omega}
\end{center}
\end{figure}
\begin{rem}
When $\mu \ge 0$, 
the two half lines from the point $(\frac{\sigma_1}{4}|c|^2,\frac{\sigma_2}{4}|c|^2)$
in Figure~\ref{dom_omega} correspond to zero-mass cases. 
When $\mu <0$, 
the two half lines from the points 
$(\frac{\sigma_1}{4}|c|^2,\frac{\sigma_3-\sigma_1}{4}|c|^2)$, 
$(\frac{\sigma_3-\sigma_2}{4}|c|^2,\frac{\sigma_2}{4}|c|^2)$, 
and the line segment connecting these points in Figure~\ref{dom_omega}
correspond to zero-mass cases. 
See also Figure~\ref{dom_omega_zeromass} below. 
In particular, the point $(\frac{\sigma_1}{4}|c|^2,\frac{\sigma_2}{4}|c|^2)$ 
when $\mu>0$ and the points $(\frac{\sigma_1}{4}|c|^2,\frac{\sigma_3-\sigma_1}{4}|c|^2)$, 
$(\frac{\sigma_3-\sigma_2}{4}|c|^2,\frac{\sigma_2}{4}|c|^2)$ when $\mu <0$
correspond to the case {\rm (C)} in Theorem~\ref{ex_gs_1}. 
\end{rem}
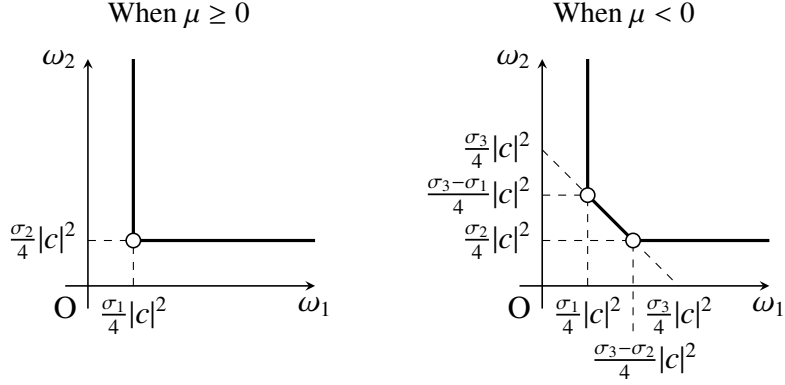
\begin{figure}[H]
\begin{center}
\begin{tikzpicture}[scale=0.6]
\draw[->,>=stealth,semithick](-0.5,0)--(5,0)node[below]{$\omega_1$};
\draw[->,>=stealth,semithick](0,-0.5)--(0,5)node[left]{$\omega_2$};
\draw(0,0)node[below left]{O};
\draw[dashed](1,0)--(1,2);
\draw[black, very thick](1,1)--(1,5);
\draw[dashed](0,1)--(2,1);
\draw[black, very thick](1,1)--(5,1);
\draw(0,1)node[left]{$\frac{\sigma_2}{4}|c|^2$};
\draw(1,0)node[below]{$\frac{\sigma_1}{4}|c|^2$};
\draw(2.25,-1.8)node[below]{};
\draw(2,6.5)node[below]{\small ${\rm When}\ \mu\ge 0$};
\fill (1,1) circle (5pt) coordinate (A);
\fill[white] (1,1) circle (4pt) coordinate (A);

\draw[->,>=stealth,semithick](-0.5+10,0)--(5+10,0)node[below]{$\omega_1$};
\draw[->,>=stealth,semithick](0+10,-0.5)--(0+10,5)node[left]{$\omega_2$};
\draw(0+10,0)node[below left]{O};
\draw[black,very thick,samples=100,domain=1:2]plot(\x+10,3-\x);
\draw[black,dashed,samples=100,domain=0:1]plot(\x+10,3-\x);
\draw[black,dashed,samples=100,domain=2:3]plot(\x+10,3-\x);
\draw[dashed](1+10,0)--(1+10,2);
\draw[dashed](2+10,-1)--(2+10,1);
\draw[black, very thick](1+10,2)--(1+10,5);
\draw[dashed](0+10,1)--(2+10,1);
\draw[dashed](0+10,2)--(1+10,2);
\draw[black, very thick](2+10,1)--(5+10,1);
\draw(0+10,3)node[left]{$\frac{\sigma_3}{4}|c|^2$};
\draw(0+10,2)node[left]{$\frac{\sigma_3-\sigma_1}{4}|c|^2$};
\draw(0+10,1)node[left]{$\frac{\sigma_2}{4}|c|^2$};
\draw(1+10,0)node[below]{$\frac{\sigma_1}{4}|c|^2$};
\draw(2.25+10,-1)node[below]{$\frac{\sigma_3-\sigma_2}{4}|c|^2$};
\draw(3+10,0)node[below]{$\frac{\sigma_3}{4}|c|^2$};
\fill (1+10,2) circle (5pt) coordinate (A);
\fill (2+10,1) circle (5pt) coordinate (B);
\fill[white] (1+10,2) circle (4pt) coordinate (A);
\fill[white] (2+10,1) circle (4pt) coordinate (B);
\draw(1.8+10,6.5)node[below]{\small ${\rm When}\ \mu <0$};
\end{tikzpicture}
\caption{The bold lines correspond to the case {\rm (B)} in the $(\omega_1,\omega_2)$-plane}\label{dom_omega_zeromass}
\end{center}
\end{figure}
\begin{rem}
In (\ref{dnls2}), we put $f_j:=e^{-it\frac{\Delta}{\sigma_j}}u_j$ $(j=1,2,3)$. 
Then, the Fourier transform of the third equation of (\ref{dnls2}) can be written
\[
\partial_t\widehat{f_3}(t,\xi)
=i\int_{\R^d}e^{-itR(\xi,\xi_1)}\widehat{f_1}(t,\xi_1)\widehat{f_2}(t,\xi-\xi_1)d\xi_1, 
\]
where
\[
R(\xi,\xi_1)=\frac{|\xi_1|^2}{\sigma_1}+\frac{|\xi-\xi_1|^2}{\sigma_2}-\frac{|\xi|^2}{\sigma_3}. 
\]
Therefore, the oscillation disappears when $R(\xi,\xi_1)=0$. 
We say ``resonance occurs'' if $R(\xi,\xi_1)=0$ holds for some $(\xi,\xi_1)\ne (0,0)$. 
The sign of $\mu$ $(=\sigma_1+\sigma_2-\sigma_3)$ determines whether or not resonance occurs. 
    In particular, resonance occurs when $\mu \ge 0$, 
    and $\mu =0$ is equivalent to the mass resonance condition (\ref{massres}). 
    Such resonance structure affects the well-posedness for the Cauchy problem. 
    Indeed, for the system (\ref{dNLS_sys}), 
    it is known that 
    the condition for the regularity of initial data 
    which admits the well-posedness depends on the sign of $\mu$ 
    (\cite{H14}, \cite{HK19}, \cite{HKO21}). 
    Because Theorem~\ref{ex_gs_1} is proved by using the variational method, 
    we can see that the resonance structure of {\rm (\ref{dnls2})} affects 
    the variational structure of the stationary problem {\rm (\ref{ellip_sys_tra})}.  
\end{rem}
\begin{rem}
    We note that the condition {\rm (A)} in 
    Theorem~\ref{ex_gs_1} is not a zero mass case. 
    Then, we have
    \[
    X_{\omega_1,\omega_2,c}=\widetilde{X}_{\omega_1,\omega_2,c}=\mathcal{H}^1(\R^d). 
    \]
    On the other hand, the conditions {\rm (B)} and {\rm (C)} in Theorem~\ref{ex_gs_1} are zero-mass cases. 
    Under the condition {\rm (B)}, we have
    \[
        \widetilde{X}_{\omega_1,\omega_2,c}
        =
        \begin{cases}
            H^1(\R^d)\times H^1(\R^d)\times \dot{H}^1(\R^d)&{\rm if}\ \omega_1>\frac{\sigma_1}{4}|c|^2,\ \omega_2>\frac{\sigma_2}{4}|c|^2,\ \omega_3=\frac{\sigma_3}{4}|c|^2,\\
            \dot{H}^1(\R^d)\times H^1(\R^d)\times H^1(\R^d)&{\rm if}\ \omega_1=\frac{\sigma_1}{4}|c|^2,\ \omega_2>\frac{\sigma_2}{4}|c|^2,\ \omega_3>\frac{\sigma_3}{4}|c|^2,\\
            H^1(\R^d)\times \dot{H}^1(\R^d)\times H^1(\R^d)&{\rm if}\ \omega_1>\frac{\sigma_1}{4}|c|^2,\ \omega_2=\frac{\sigma_2}{4}|c|^2,\ \omega_3>\frac{\sigma_3}{4}|c|^2.
        \end{cases}
    \]
    Under the condition {\rm (C)}, we have
    \[
        \widetilde{X}_{\omega_1,\omega_2,c}
        =
        \begin{cases}
            \dot{H}^1(\R^d)\times \dot{H}^1(\R^d)\times H^1(\R^d)&{\rm if}\ (\omega_1,\omega_2)= (\frac{\sigma_1}{4}|c|^2,\frac{\sigma_2}{4}|c|^2),\ \omega_3>\frac{\sigma_3}{4}|c|^2,\\
            H^1(\R^d)\times \dot{H}^1(\R^d)\times \dot{H}^1(\R^d)&{\rm if}\ (\omega_2,\omega_3)= (\frac{\sigma_2}{4}|c|^2,\frac{\sigma_3}{4}|c|^2),\ \omega_1>\frac{\sigma_1}{4}|c|^2,\\
            \dot{H}^1(\R^d)\times H^1(\R^d)\times \dot{H}^1(\R^d)&{\rm if}\ (\omega_1,\omega_3)= (\frac{\sigma_1}{4}|c|^2,\frac{\sigma_3}{4}|c|^2),\ \omega_2>\frac{\sigma_2}{4}|c|^2.
        \end{cases}
    \]
\end{rem}
We also obtain the following global well-posedness results. 
\begin{thm}[global well-posedness I]\label{GWP_1}
Let $\sigma_1,\sigma_2,\sigma_3>0$.
\begin{itemize} 
    \item[{\rm (i)}] Assume $d=4$ and 
    $U_0\in \mathcal{H}^1(\R^4)$ satisfies at least one of 
    the 
    following two conditions {\rm (a)} and {\rm (b)}$:$
    \[
    \begin{split}
    &{\rm (a)}\ \ Q_1(U_0)<\sup_{\Phi \in \mathcal{M}_{1,0,0}}Q_1(\Phi).\\
    &{\rm (b)}\ \ Q_2(U_0)<\sup_{\Phi \in \mathcal{M}_{0,1,0}}Q_2(\Phi).
    \end{split}\]
    Then, the local solution to {\rm (\ref{dnls2})} 
    constructed in {\rm Theorem~\ref{Theorem_NP}\ (i)}\ 
    can be extended globally in time.
    \item[{\rm (ii)}] Assume $d=5$ and 
    $U_0\in \mathcal{H}^1(\R^5)$ satisfies at least one of 
    the 
    following two conditions {\rm (c)} and {\rm (d)}$:$
    \[
    \begin{split}
    &{\rm (c)}\ \ Q_1(U_0)E(U_0)<\sup_{\Phi \in \mathcal{M}_{1,0,0}}Q_1(\Phi)E(\Phi)\ \ \ 
    {\rm and}\ \ \ Q_1(U_0)L(U_0)<\sup_{\Phi \in \mathcal{M}_{1,0,0}}Q_1(\Phi)L(\Phi).\\
    &{\rm (d)}\ \ Q_2(U_0)E(U_0)<\sup_{\Phi \in \mathcal{M}_{0,1,0}}Q_2(\Phi)E(\Phi)\ \ \ 
    {\rm and}\ \ \ Q_2(U_0)L(U_0)<\sup_{\Phi \in \mathcal{M}_{0,1,0}}Q_2(\Phi)L(\Phi).
    \end{split}
    \]
    Then, the local solution to {\rm (\ref{dnls2})}
    constructed in {\rm Theorem~\ref{Theorem_NP}\ (i)}\ 
    can be extended globally in time. 
\end{itemize}
\end{thm}
\begin{rem}
As mentioned in Remark~\ref{pos_Q_1_Q_2_rem}, 
$Q_1(\Phi)>0$ and $Q_2(\Phi)>0$ 
hold for $\Phi \in \mathcal{M}_{\omega_1,\omega_2,0}$. 
Therefore, there exists $U_0\in \mathcal{H}^1(\R^4)$ 
satisfying either condition {\rm (a)} or {\rm (b)} in Theorem~\ref{GWP_1}. 
When $d=5$, we can also obtain $E(\Phi)=\frac{1}{5}L(\Phi)>0$ 
for $\Phi \in \mathcal{M}_{\omega_1,\omega_2,0}$ by Proposition~\ref{Pohoz_id} below. 
Therefore, there exists $U_0\in \mathcal{H}^1(\R^5)$ 
satisfying the either condition {\rm (c)} or {\rm (d)} in Theorem~\ref{GWP_1}. 
\end{rem}
\begin{rem}\label{GWP1_rem_refine}
{\rm Theorem}~\ref{GWP_1} says that 
    the smallness of $u_{1,0}$ is not needed for the cases {\rm (b)} and {\rm (d)} 
    to obtain the global well-posedness. 
    Similarly, the smallness of $u_{2,0}$ is not needed for the cases {\rm (a)} and {\rm (c)}. 
    These are the refinement of the results in \cite{NP21, NP22} 
    (Theorem~\ref{Theorem_NP} {\rm (iii)}, {\rm (iv)} above). 
\end{rem}
\begin{thm}[global well-posedness II]\label{GWP_2}
    Let $d=4$, $\sigma_1,\sigma_2,\sigma_3>0$, and $V_0=(v_{1,0},v_{2,0},v_{3,0})\in \mathcal{H}^1(\R^4)$. 
    \begin{itemize}
        \item[{\rm (i)}] Assume $\sigma_1+\sigma_2>\sigma_3$. 
        There exists $A_0>0$ such that if $\|v_{3,0}\|_{L^2}^2<A_0$ 
        and $|c|$ is large enough depending on $V_0$, 
        then the local solution to {\rm (\ref{dnls2})} 
        with $U_0=\Lambda_cV_0$ 
        constructed in {\rm Theorem~\ref{Theorem_NP}\ (i)}\ 
        can be extended globally in time.
        \item[{\rm (ii)}] Assume $\sigma_1+\sigma_2<\sigma_3$. 
        There exist $B_0>0$ and $C_0>0$ such that if either $\|v_{1,0}\|_{L^2}^2<B_0$ or $\|v_{2,0}\|_{L^2}^2<C_0$ holds
        and $c$ is large enough depending on $V_0$, 
        then the local solution to {\rm (\ref{dnls2})} 
        with $U_0=\Lambda_cV_0$ 
        constructed in {\rm Theorem~\ref{Theorem_NP}\ (i)}\ 
        can be extended globally in time.
    \end{itemize} 
\end{thm}
\begin{rem}
    Theorem~\ref{GWP_2} says that 
    the smallness assumption is needed 
    for only one of $v_{1,0}$, $v_{2,0}$, and $v_{3,0}$ 
    to obtain the global well-posedness for oscillating initial data. 
    This is the refinement of the results in \cite{Lipre} 
    (Theorem~\ref{Theorem_Li} {\rm (i)}, {\rm (ii)}, {\rm (iii)} above). 
\end{rem}
Finally, we mention the stability of the ground state sets. 
\begin{defn}
    We say that the set $\mathcal{M}\subset \mathcal{M}_{\omega_1,\omega_2,c}$ 
    is orbitally stable if the following property holds:

    For any $\epsilon >0$, there exists $\delta>0$ such that
    if $U_0\in \mathcal{H}^1(\R^d)$ satisfies
    \[
    \inf_{\Phi \in \mathcal{M}}\|U_0-\Phi\|_{X_{\omega_1,\omega_2,c}}<\delta,
    \]
    then the solution $U(t)$ to {\rm (\ref{dnls2})} exists globally in time and satisfies
    \[
    \inf_{\Phi\in \mathcal{M}}\|U(t)-\Phi\|_{X_{\omega_1,\omega_2, c}}<\epsilon
    \]
    for any $t>0$. 
\end{defn}
\begin{rem}
    When $(\omega_1,\omega_2)\in \Omega_c$, 
    we can replace $X_{\omega_1,\omega_2,c}$ by $\mathcal{H}^1(\R^d)$ 
    (see, Proposition~\ref{norm_X_H1_equi} and Remark~\ref{norm_eq_X_H_rem}). 
\end{rem}
\begin{thm}[orbital stability]\label{stab_thm}
    Let $\sigma_1,\sigma_2,\sigma_3>0$. 
    \begin{enumerate}
    \item[{\rm (i)}] Assume $1\le d\le 3$. For any $\omega_1,\omega_2>0$, 
    there exists $c_0=c_0(\omega_1,\omega_2)>0$  with 
    $c_0<\min\left\{\sqrt{\frac{4\omega_1}{\sigma_1}},\sqrt{\frac{4\omega_2}{\sigma_2}},\sqrt{\frac{4(\omega_1+\omega_2)}{\sigma_3}}\right\}$
    such that if $|c|\le c_0$, then $\mathcal{M}_{\omega_1,\omega_2,c}$ is orbitally stable. 
    \item[{\rm (ii)}] Assume $d=3$. For any $\omega_1,\omega_2>0$, the sets 
    $\mathcal{M}_{\omega_1,0,0}$ and $\mathcal{M}_{0,\omega_2,0}$ are orbitally stable. 
    \end{enumerate}
\end{thm}
\begin{rem}
    In the proof of Theorem~\ref{stab_thm}, 
    we will prove the set $\mathcal{M}_{\omega_1,\omega_2,c}^*(\eta)$ defined by 
    \[
    \mathcal{M}_{\omega_1,\omega_2,c}^*(\eta)
    :=\{\Phi \in \mathcal{M}_{\omega_1,\omega_2,c}|\ (8-2d)(\omega_1Q_1(\Phi)+\omega_2Q_2(\Phi))+(5-d)c\cdot P(\Phi)\ge \eta\}
    \]
    is orbitally stable for any $\eta>0$. 
    The proof of Theorem~\ref{stab_thm} will be completed by showing that
    $\mathcal{M}_{\omega_1,\omega_2,c}$
    coincides with the set $\mathcal{M}_{\omega_1,\omega_2,c}^*(\eta)$ for some $\eta >0$ if $c$ is small enough 
    when $1\le d\le 3$. 
    Next, we mention about $d=4$ and $5$. 
    Clearly, $\mathcal{M}_{\omega_1,\omega_2,c}^*(\eta)=\emptyset$ for any $\eta >0$ when $d=5$. 
    It also holds that $\mathcal{M}_{\omega_1,\omega_2,c}^*(\eta)=\emptyset$ when $d=4$ 
    because $c\cdot P(\Phi)\le 0$ holds for any $(\omega_1,\omega_2)\in \Omega_c$ (see, Remark~\ref{cp_neg_rem} below). 
      
\end{rem}
We give a notation table (Table~\ref{tab1} below) for the reader's convenience. 
\begin{table}[H]
\caption{Notation table}\label{tab1}
\vspace{.3cm}
\renewcommand{\arraystretch}{1.1}
 \begin{tabularx}{\linewidth}{|c|X|}
\hline
$\mathcal{H}^1(\R^d)$& Sobolev space $H^1(\R^d)\times H^1(\R^d)\times H^1(\R^d)$\\
$Q_1(U)$ & Charge $\frac{1}{2}\|u_1\|_{L^2(\R^d)}^2+\frac{1}{2}\|u_3\|_{L^2(\R^d)}^2$\\
$Q_2(U)$ & Charge $\frac{1}{2}\|u_2\|_{L^2(\R^d)}^2+\frac{1}{2}\|u_3\|_{L^2(\R^d)}^2$\\
$L(U)$ & Kinetic energy $\frac{1}{2\sigma_1}\|\nabla u_1\|_{L^2(\R^d)}^2+\frac{1}{2\sigma_2}\|\nabla u_2\|_{L^2(\R^d)}^2
    +\frac{1}{2\sigma_3}\|\nabla u_3\|_{L^2(\R^d)}^2$\\
$N(U)$ & Potential energy $-{\rm Re}\left(u_1u_2,u_3\right)_{L^2(\R^d)}$\\
$P(U)$ & Momentum vector $(P_1(U),\cdots,P_d(U))$\\
$P_k(U)$ & $k$-th component of the momentum 
$\frac{1}{2}\sum_{j=1}^3{\rm Re}(i\partial_ku_j,u_j)_{L^2(\R^d)}$\\
$\omega_1$, $\omega_2$ & Independent frequency parameters; $\omega_3:=\omega_1+\omega_2$\\
$S_{\omega_1,\omega_2,c}(U)$ & Action $E(U)+\omega_1 Q_1(U)+\omega_2Q_2(U)+c\cdot P(U)$\\
$K_{\omega_1,\omega_2,c}(U)$ & Nehari functional $\partial_{\lambda}S_{\omega_1,\omega_2,c}(\lambda U)\big|_{\lambda=1}$\\
$X_{\omega_1,\omega_2,c}$ & Function space $X_1\times X_2\times X_3$\\ 
$X_j$ & Set of all $f$ such that $\nabla (e^{-i\frac{\sigma_j}{2}c\cdot(\cdot )}f)\in L^2(\R^d)$ and  
$(\omega_j-\frac{\sigma_j}{4}|c|^2)^{\frac{1}{2}}f\in L^2(\R^d)$\\
$\mathcal{E}_{\omega_1,\omega_2,c}$ & Set of the critical points of action $\{\Phi \in X_{\omega_1,\omega_2,c}\ |\ \Phi \ne 0,\ 
    S_{\omega_1,\omega_2,c}'(\Phi)=0\}$\\
$\mathcal{G}_{\omega_1,\omega_2,c}$ & Set of ground states $\{\Phi \in \mathcal{E}_{\omega_1,\omega_2,c}\ |\ \ 
    S_{\omega_1,\omega_2,c}(\Phi)\le S_{\omega_1,\omega_2,c}(\Theta)\ \text{for\ any}\ 
    \Theta \in \mathcal{E}_{\omega_1,\omega_2,c}\}$\\
$\mathfrak{m}_{\omega_1,\omega_2,c}$& Least action level $\inf \{S_{\omega_1,\omega_2,c}(\Phi)\ |\ \Phi \in X_{\omega_1,\omega_2,c},\ \Phi \ne 0,\ 
K_{\omega_1,\omega_2,c}(\Phi )=0\}$\\
$\mathcal{M}_{\omega_1,\omega_2,c}$ & Set of minimizers $\{\Phi \in X_{\omega_1,\omega_2,c}\ |\ \Phi \ne 0,\ 
S_{\omega_1,\omega_2,c}(\Phi)=\mathfrak{m}_{\omega_1,\omega_2,c},\ 
K_{\omega_1,\omega_2,c}(\Phi)=0\}$\\
\hline
 \end{tabularx}
 \end{table}
\section{Properties of weak solutions to the stationary problem}
In this section, we prove the properties of 
weak solutions to {\rm (\ref{ellip_sys_tra})} and {\rm (\ref{ellip_sys2})}. 
We first give the regularity property. 
\begin{prop}\label{regu_prop_ws}
    Let $1\le d\le 5$, $\sigma_1,\sigma_2,\sigma_3>0$, and $c\in \R^d$. 
    Assume one of the conditions {\rm (A)} or {\rm (B)} in Theorem~\ref{ex_gs_1} holds. 
     Then, any weak solution $\Psi =(\psi_1,\psi_2,\psi_3) \in \widetilde{X}_{\omega_1,\omega_2,c}$ to {\rm (\ref{ellip_sys2})}
    satisfies $\psi_j\in \bigcap_{m=0}^{\infty}H^m(\R^d)$ if $\omega_j-\frac{\sigma_j}{4}|c|^2>0$ and $\nabla\psi_j\in \bigcap_{m=0}^{\infty}H^m(\R^d)$ if $\omega_j-\frac{\sigma_j}{4}|c|^2=0$ for $j=1,2,3$.  
\end{prop}
\begin{proof}
    We put $\gamma_j:=\omega_j-\frac{\sigma_j}{4}|c|^2\ (j=1, 2, 3)$ and $\mu:=\sigma_1+\sigma_2-\sigma_3$. 
    We only consider the case {\rm (B)} because 
    the case {\rm (A)}, where all $\gamma_j>0$, follows by the standard elliptic bootstrap in $H^1(\mathbb{R}^d)$.
    
    We assume $\gamma_1=0$ and $\gamma_2,\gamma_3>0$. 
    Then $\psi_1\in \dot{H}^1(\R^d)$ and $\psi_2,\psi_3\in H^1(\R^d)$. 
    Let $p_0\in [1,\frac{d}{d-2}]$. 
    By the first equation of {\rm (\ref{ellip_sys2})}, 
    the H\"older inequality, and the Sobolev embedding $H^1(\R^d)\hookrightarrow L^{2p_0}(\R^d)$, we have
    \[
    \begin{split}
    \|\psi_1\|_{\dot{W}^{2,p_0}}
    &=\left\|\left(-\frac{1}{\sigma_1}\Delta\right)^{-1}(e^{-i\frac{\mu}{2}c\cdot x}\overline{\psi_2}\psi_3)\right\|_{\dot{W}^{2,p_0}}\\
    &\lesssim \|\psi_2\psi_3\|_{L^{p_0}}
    \le \|\psi_2\|_{L^{2p_0}}\|\psi_3\|_{L^{2p_0}}\lesssim \|\psi_2\|_{H^1}\|\psi_3\|_{H^1}. 
    \end{split}
    \]
    Therefore, we get $\psi_1\in \dot{W}^{2,p_0}(\R^d)$. 
    Additionally assume $p_0\in [\frac{2d}{d+4},\frac{d}{2})$. Let $a_0,b_0\ge 2$ satisfy
    \[
    2=\frac{d}{p_0}-\frac{d}{a_0},\ \ \frac{1}{2}=\frac{1}{a_0}+\frac{1}{b_0}
    \]
    Then, by the second equation of {\rm (\ref{ellip_sys2})}, 
    the H\"older inequality, and the Sobolev embedding $\dot{W}^{2,p_0}(\R^d)\hookrightarrow L^{a_0}(\R^d)$, we have
    \[
    \begin{split}
    \|\psi_2\|_{H^2}
    &=\left\|\left(-\frac{1}{\sigma_2}\Delta +\gamma_2\right)^{-1}(e^{-i\frac{\mu}{2}c\cdot x}\overline{\psi_1}\psi_3)\right\|_{H^2}\\
    &\le \|\psi_1\psi_3\|_{L^{2}}
    \le \|\psi_1\|_{L^{a_0}}\|\psi_3\|_{L^{b_0}}\lesssim \|\psi_1\|_{\dot{W}^{2,p_0}}\|\psi_3\|_{L^{b_0}}. 
    \end{split}
    \]
    If we choose $p_0$ as $p_0\ge \frac{d}{3}$, then it holds
    \[
    \|\psi_3\|_{L^{b_0}}\lesssim \|\psi_3\|_{H^1}
    \]
    by the Sobolev embedding $H^1(\R^d)\hookrightarrow L^{b_0}(\R^d)$. 
    Note that $[\frac{d}{3},\infty)\cap [1,\frac{d}{d-2}]\cap [\frac{2d}{d+4},\frac{d}{2})\ne \emptyset$ when $3\le d\le 5$. 
    Therefore, we obtain $\psi_2\in H^2(\R^d)$. 
    By the same argument, we have $\psi_3\in H^2(\R^d)$. 
    Therefore, by the H\"older inequality and the Sobolev embedding $H^2(\R^d)\hookrightarrow L^{4}(\R^d)$, we have
    \[
    \|\psi_1\|_{\dot{H}^2}
    \lesssim \|\psi_2\psi_3\|_{L^2}\le \|\psi_2\|_{L^4}\|\psi_3\|_{L^4}\lesssim \|\psi_2\|_{H^2}\|\psi_3\|_{H^2},
    \]
    and it implies $\psi_1\in \dot{H}^2(\R^d)$. Namely, $\nabla \psi_1\in H^1(\R^d)$. 
    
    For $m\ge 2$, we also have
    \begin{equation}\label{p2_m1_est}
    \begin{split}
    \|\psi_2\|_{\dot{H}^{m+1}}
    &=\left\|\left(-\frac{1}{\sigma_2}\Delta +\gamma_2\right)^{-1}(e^{-i\frac{\mu}{2}c\cdot x}\overline{\psi_1}\psi_3)\right\|_{\dot{H}^{m+1}}\\
    &\lesssim \|\psi_1\psi_3\|_{H^{m-1}}
    \lesssim \|\psi_1\psi_3\|_{L^2}+\|(|\nabla|^{m-1}\psi_1)\psi_3\|_{L^{2}}
    +\|\psi_1\langle \nabla \rangle^{m-1}\psi_3\|_{L^{2}}. 
    \end{split}
    \end{equation}
    By the H\"older inequality and the Sobolev embeddings $\dot{H}^1(\R^d)\hookrightarrow L^{\frac{2d}{d-2}}(\R^d)$, 
    $H^2(\R^d)\hookrightarrow L^{d}(\R^d)$, we obtain
    \[
    \|\psi_1\psi_3\|_{L^2}\le \|\psi_1\|_{L^{\frac{2d}{d-2}}}\|\psi_3\|_{L^{d}}
    \lesssim \|\psi_1\|_{\dot{H}^1}\|\psi_3\|_{H^2}
    \]
    and
    \[
    \|(|\nabla|^{m-1}\psi_1)\psi_3\|_{L^2}\le \||\nabla|^{m-1}\psi_1\|_{L^{\frac{2d}{d-2}}}\|\psi_3\|_{L^{d}}
    \lesssim \|\psi_1\|_{\dot{H}^m}\|\psi_3\|_{H^2}
    \]
    holds for $3\le d\le 5$. 
    For the third term of the right-hand side of (\ref{p2_m1_est}), 
    we split into the cases $3\le d\le 4$ and $d=5$. 
    We first assume the case $3\le d\le 4$. 
    Then, by the H\"older inequality and the Sobolev embeddings $\dot{H}^1(\R^d)\hookrightarrow L^{\frac{2d}{d-2}}(\R^d)$, 
    $H^1(\R^d)\hookrightarrow L^{d}(\R^d)$, we obtain
    \[
    \|\psi_1\langle \nabla \rangle^{m-1}\psi_3\|_{L^2}\le \|\psi_1\|_{L^{\frac{2d}{d-2}}}\|\langle \nabla \rangle^{m-1}\psi_3\|_{L^{d}}
    \lesssim \|\psi_1\|_{\dot{H}^1}\|\psi_3\|_{H^m}. 
    \]
    Next, we assume the case $d=5$. 
    Then, by the H\"older inequality and the Sobolev embeddings $\dot{H}^2(\R^5)\hookrightarrow L^{10}(\R^5)$, 
    $H^1(\R^5)\hookrightarrow L^{\frac{5}{2}}(\R^5)$, we obtain
    \[
    \|\psi_1\langle \nabla \rangle^{m-1}\psi_3\|_{L^2}\le \|\psi_1\|_{L^{10}}\|\langle \nabla \rangle^{m-1}\psi_3\|_{L^{\frac{5}{2}}}
    \lesssim \|\psi_1\|_{\dot{H}^2}\|\psi_3\|_{H^m}. 
    \]
    As a result, we get 
    \[
    \|\psi_2\|_{\dot{H}^{m+1}}\lesssim \|\nabla \psi_1\|_{H^1}\|\psi_3\|_{H^m}.
    \]
    By a similar argument, we also obtain
    \[
    \|\psi_1\|_{\dot{H}^{m+1}}\lesssim \|\psi_2\|_{H^m}\|\psi_3\|_{H^m},\ \ \ \ 
    \|\psi_3\|_{\dot{H}^{m+1}}\lesssim \|\nabla \psi_1\|_{H^1}\|\psi_2\|_{H^m}.
    \]
    Therefore, by induction, 
    we have $\nabla \psi_1,\psi_2,\psi_3\in \bigcap_{m=1}^{\infty}H^m(\R^d)$. 
\end{proof}
\begin{rem}\label{reg_sol_rem}
\begin{itemize}
    \item[{\rm (i)}] Proposition~\ref{regu_prop_ws} says that, 
    if $\omega_j-\frac{\sigma_j}{4}|c|^2>0$, 
    then we have $\psi_j\in L^{\infty}(\R^d)$ 
    because the Sobolev embedding $H^m(\R^d)\hookrightarrow L^{\infty}(\R^d)$ 
     holds for $m>\frac{d}{2}$. We also have 
    $\varphi_j=e^{i\frac{\sigma_j}{2}c\cdot x}\psi_j\in L^{\infty}(\R^d)$ since $\|\psi_j\|_{L^{\infty}}=\|\varphi_j\|_{L^{\infty}}$. 
    We will use this fact to prove Lemma~\ref{lap_decay_est} below. 
    \item[{\rm (ii)}] When $c=0$, the same properties as in Proposition~\ref{regu_prop_ws} 
    hold for any weak solution $\Phi \in X_{\omega_1,\omega_2,0}$ to {\rm (\ref{ellip_sys_tra})} 
    because $\Phi =\Lambda_0\Psi =\Psi$. 
    \item[{\rm (iii)}] The same properties in Proposition~\ref{regu_prop_ws} also hold 
    for the case {\rm $(C)'$} in {\rm Proposition~\ref{pos_pr_Neh}} below. 
    We will give the proof in Appendix. 
    \end{itemize}
\end{rem}

Next, we give the decay property. 
\begin{prop}\label{ws_decay}
Let $\sigma_1,\sigma_2,\sigma_3>0$, and $c\in \R^d$.
\begin{itemize} 
\item[{\rm (i)}] Assume $1\le d\le 5$. 
    If $(\omega_1,\omega_2)\in \Omega_c$, then there exists $p>0$ such that any weak solution $\Psi =(\psi_1,\psi_2,\psi_3) \in \widetilde{X}_{\omega_1,\omega_2,c}$ to {\rm (\ref{ellip_sys2})}
    satisfies
    \begin{equation}\label{sum_decay_est}
    \sum_{j=1}^3\int_{\R^d}e^{2p|x|}(|\psi_j(x)|^2+|\nabla \psi_j(x)|^2)dx<\infty. 
    \end{equation}
    \item[{\rm (ii)}] Assume $3\le d\le 5$ and $1\le k<l\le 3$. 
     If $(\omega_1,\omega_2)\in \partial\Omega_c$ 
     satisfies
     $\omega_{k}-\frac{\sigma_{k}}{4}|c|^2>0$ and $\omega_{l}-\frac{\sigma_{l}}{4}|c|^2>0$, then there exists $p>0$ such that any weak solution $\Psi =(\psi_1,\psi_2,\psi_3) \in \widetilde{X}_{\omega_1,\omega_2,c}$ to {\rm (\ref{ellip_sys2})}
    satisfies
    \begin{equation}\label{decay_est_kl}
    \int_{\R^d}e^{2p|x|}(|\psi_k(x)|^2+|\nabla \psi_k(x)|^2)dx<\infty,\ \ \ \ 
    \int_{\R^d}e^{2p|x|}(|\psi_l(x)|^2+|\nabla \psi_l(x)|^2)dx<\infty. 
    \end{equation}
\end{itemize}
\end{prop}
\begin{proof}We first prove {\rm (ii)}.  
    Put $\gamma_j:= \omega_j-\frac{\sigma_j}{4}|c|^2$\ $(j=1,2,3)$. 
    Then, {\rm (\ref{ellip_sys2})} can be written
    \begin{equation}\label{sp_gamma}
    -\frac{1}{\sigma_j}\Delta \psi_j+\gamma_j\psi_j=-G_{j,c}(\Psi),\ x\in \R^d,\ \ j=1,2,3. 
    \end{equation}
    We only consider the case $k=1$, $l=2$. Namely, we assume $\gamma_1,\gamma_2>0$, and $\gamma_3\ge 0$. 
    Let $\Psi=(\psi_1,\psi_2,\psi_3)\in \widetilde{X}_{\omega_1,\omega_2,c}$ be a solution to {\rm (\ref{sp_gamma})}. For $p>0$, which will be chosen later and any $\delta >0$, we put $h_{p,\delta}(x):=\exp \left(\frac{p|x|}{1+\delta |x|}\right)$. 
    Note that $h_{p,\delta}\in W^{1,\infty}(\R^d)$, which implies $h_{p,\delta}^2\psi_1, h_{p,\delta}^2\psi_2 \in H^1(\R^d)$, and 
    \[
    |\nabla h_{p,\delta}(x)|\le ph_{p,\delta}(x)
    \]
    hold. 
    Then, for $j=1,2$, we have
    \begin{equation}\label{sp_hp_psi}
    \begin{split}
    0&=\left(-\frac{1}{\sigma_j}\Delta \psi_j+\gamma_j\psi_j+G_{j,c}(\Psi),h_{p,\delta}^2\psi_j\right)_{L^2}\\
    &=\frac{1}{\sigma_j}\left(\|h_{p,\delta}\nabla \psi_j\|_{L^2}^2+(\nabla \psi_j,(\nabla (h_{p,\delta})^2)\psi_j)_{L^2}\right)+\gamma_j\|h_{p,\delta}\psi_j\|_{L^2}^2+(G_{j,c}(\Psi),h_{p,\delta}^2\psi_j)_{L^2}. 
    \end{split}
    \end{equation}
    By the Cauchy-Schwarz inequality and the Young inequality, we obtain
    \[
    \begin{split}
    |(\nabla \psi_j,(\nabla (h_{p,\delta})^2)\psi_j)_{L^2}|
    &\le 2p\int_{\R^d}h_{p,\delta}(x)^2|\nabla \psi_j(x)||\psi_j(x)|dx
    \le 2p\|h_{p,\delta}\nabla \psi_j\|_{L^2}\|h_{p,\delta}\psi_j\|_{L^2}\\
    &\le \frac{1}{2}\|h_{p,\delta}\nabla \psi_j\|_{L^2}^2+2p^2\|h_{p,\delta}\psi_j\|_{L^2}^2. 
    \end{split}
    \]
    On the other hand, we have
    \begin{equation}\label{nonl_int_est}
    \left|(G_{j,c}(\Psi),h_{p,\delta}^2\psi_j)_{L^2}\right|
    \le \int_{\R^d}h_{p,\delta}(x)^2|\psi_1(x)||\psi_2(x)||\psi_3(x)|dx. 
    \end{equation}
Since $\psi_3\in H^1(\R^d)$ if $\gamma_3>0$ and $\psi_3\in \dot H^1(\R^d)$ if $\gamma_3=0$, in either case $\psi_3\in \dot H^1(\R^d)$. Thus $\psi_3$ belongs to $L^{\frac{2d}{d-2}}(\R^d)$ by the Sobolev embedding $\dot{H}^1(\R^d)\hookrightarrow L^{\frac{2d}{d-2}}(\R^d)$. 
    Due to the density of $C_c(\R^d)$ in $L^{\frac{2d}{d-2}}(\R^d)$, 
    for any $\epsilon>0$, there exists $\eta_{\epsilon}\in C_c(\R^d)$ 
    with ${\rm supp}\ \eta_{\epsilon}\subset B_R(\R^d):=\{x\in \R^d\ |\ |x|\le R\}$ for some $R>0$ such that
    \[
    \|\psi_3-\eta_{\epsilon}\|_{L^{\frac{2d}{d-2}}}<\epsilon
    \]
    holds. We split the integral in the right-hand side of {\rm (\ref{nonl_int_est})} into
    \begin{equation}\label{int_split}
    \begin{split}
       &\int_{\R^d}h_{p,\delta}(x)^2|\psi_1(x)||\psi_2(x)||\psi_3(x)|dx\\
        &\le \int_{|x|\le R}h_{p,\delta}(x)^2|\psi_1(x)||\psi_2(x)||\eta_{\epsilon}(x)|dx
        +\int_{\R^d}h_{p,\delta}(x)^2|\psi_1(x)||\psi_2(x)||\psi_3(x)-\eta_{\epsilon}(x)|dx\\
        &=:I+II. 
    \end{split}
    \end{equation}
    
    \noindent \emph{Estimate of $I$.}
    
    By the H\"older inequality, the Sobolev embeddings $H^1(\R^d)\hookrightarrow L^{\frac{4d}{d+2}}(\R^d)$ and $\dot{H}^1(\R^d)\hookrightarrow L^{\frac{2d}{d-2}}(\R^d)$, 
    and choosing $\epsilon>0$ as $\epsilon <\|\psi_3\|_{L^{\frac{2d}{d-2}}}$, 
    we obtain
    \begin{equation}\label{split_I_est}
    I\le e^{2pR}\|\psi_1\|_{L^{\frac{4d}{d+2}}}\|\psi_2\|_{L^{\frac{4d}{d+2}}}\|\eta_{\epsilon}\|_{L^{\frac{2d}{d-2}}}
    \le Ce^{2pR}\|\psi_1\|_{H^1}\|\psi_2\|_{H^1}\|\psi_3\|_{\dot{H}^1}.
    \end{equation}

    \noindent \emph{Estimate of $II$.}

    By the H\"older inequality and the Gagliardo--Nirenberg inequality, 
    we get
    \[
    \begin{split}
    II&\le \|h_{p,\delta}\psi_1\|_{L^{\frac{4d}{d+2}}}\|h_{p,\delta}\psi_2\|_{L^{\frac{4d}{d+2}}}\|\psi_3-\eta_{\epsilon}\|_{L^{\frac{2d}{d-2}}}\\
    &\le \epsilon \|\nabla (h_{p,\delta}\psi_1)\|_{L^{2}}^{\frac{d-2}{4}}\|\nabla (h_{p,\delta}\psi_2)\|_{L^{2}}^{\frac{d-2}{4}}
    \|h_{p,\delta}\psi_1\|_{L^{2}}^{\frac{6-d}{4}}\|h_{p,\delta}\psi_2\|_{L^{2}}^{\frac{6-d}{4}}. 
    \end{split}
    \]
    Because
    \[
    |\nabla (h_{p,\delta}(x)\psi_j(x))|
    \le ph_{p,\delta}(x)|\psi_j(x)|+h_{p,\delta}(x)|\nabla \psi_j(x)|
    \]
    holds, we obtain
    \[
    \begin{split}
    II&\le \epsilon \left(p^{\frac{d-2}{4}}\|h_{p,\delta}\psi_1\|_{L^2}^{\frac{d-2}{4}}+\|h_{p,\delta}\nabla \psi_1\|_{L^2}^{\frac{d-2}{4}}\right)\left(p^{\frac{d-2}{4}}\|h_{p,\delta}\psi_2\|_{L^2}^{\frac{d-2}{4}}+\|h_{p,\delta}\nabla \psi_2\|_{L^2}^{\frac{d-2}{4}}\right)\\
    &\hspace{55ex}\times\|h_{p,\delta}\psi_1\|_{L^{2}}^{\frac{6-d}{4}}\|h_{p,\delta}\psi_2\|_{L^{2}}^{\frac{6-d}{4}}. 
    \end{split}
    \]
    Furthermore, by using the Young inequality, there exists $C>0$ such that
    \begin{equation}\label{split_II_est}
    II\le \frac{1}{8\sigma_1}\|h_{p,\delta}\nabla \psi_1\|_{L^2}^2+\frac{1}{8\sigma_2}\|h_{p,\delta}\nabla \psi_2\|_{L^2}^2
    +C\epsilon \left(\|h_{p,\delta}\psi_1\|_{L^2}^2+\|h_{p,\delta}\psi_2\|_{L^2}^2\right)
    \end{equation}
    holds. \\

    By {\rm (\ref{sp_hp_psi})}, {\rm (\ref{nonl_int_est})}, {\rm (\ref{int_split})}, {\rm (\ref{split_I_est})}, and {\rm (\ref{split_II_est})}, we have
    \[
    \begin{split}
        &\sum_{j=1}^2\left(\frac{1}{\sigma_j}\|h_{p,\delta}\nabla \psi_j\|_{L^2}^2+\gamma_j\|h_{p,\delta}\psi_j\|_{L^2}^2\right)\\
        &\le Ce^{2pR}\|\psi_1\|_{H^1}\|\psi_2\|_{H^1}\|\psi_3\|_{\dot{H}^1}+\sum_{j=1}^2\frac{1}{\sigma_j}\left(\frac{1}{2}\|h_{p,\delta}\nabla \psi_j\|_{L^2}^2+2p^2\|h_{p,\delta}\psi_j\|_{L^2}^2\right)\\
        &\ \ \ \ \ \ \ \ +\frac{1}{4\sigma_1}\|h_{p,\delta}\nabla \psi_1\|_{L^2}^2+\frac{1}{4\sigma_2}\|h_{p,\delta}\nabla \psi_2\|_{L^2}^2
    +2C\epsilon \left(\|h_{p,\delta}\psi_1\|_{L^2}^2+\|h_{p,\delta}\psi_2\|_{L^2}^2\right)\\
    &=Ce^{2pR}\|\psi_1\|_{H^1}\|\psi_2\|_{H^1}\|\psi_3\|_{\dot{H}^1}+\sum_{j=1}^2\left(\frac{3}{4\sigma_j}\|h_{p,\delta}\nabla \psi_j\|_{L^2}^2+2\left(\frac{p^2}{\sigma_j}+C\epsilon\right)\|h_{p,\delta}\psi_j\|_{L^2}^2\right). 
    \end{split}
    \]
    Putting 
    \[
    p=\min_{1\le j\le 2}\sqrt{\frac{\sigma_j\gamma_j}{4}}
    \]
    and choosing $\epsilon >0$ as $4C\epsilon <\min \{\gamma_1,\gamma_2\}$, it holds that
    \[
    \sum_{j=1}^2\left(\frac{1}{4\sigma_j}\|h_{p,\delta}\nabla \psi_j\|_{L^2}^2+\frac{\gamma_j}{2}\|h_{p,\delta}\psi_j\|_{L^2}^2\right)
    \le Ce^{2pR}\|\psi_1\|_{H^1}\|\psi_2\|_{H^1}\|\psi_3\|_{\dot{H}^1}<\infty. 
    \]
    Letting $\delta \rightarrow +0$ and applying the Lebesgue's monotone convergence theorem, 
    we obtain (\ref{decay_est_kl}). 

    Next, we prove {\rm (i)}. In the proof of {\rm (ii)}, we use the density of $C_c(\R^d)$ in $L^3(\R^d)$ instead of $L^{\frac{2d}{d-2}}(\R^d)$. 
    Therefore, in the estimates for $I$ and $II$ as above, 
    we use the H\"older inequality $L^3(\R^d)\times L^3(\R^d)\times L^3(\R^d)\rightarrow L^1(\R^d)$ 
    instead of $L^{\frac{4d}{d+2}}(\R^d)\times L^{\frac{4d}{d+2}}(\R^d)\times L^{\frac{2d}{d-2}}(\R^d)\rightarrow L^1(\R^d)$. 
    Otherwise, it is the same as the proof of {\rm (ii)}. 
    
\end{proof}
\begin{cor}\label{ws_decay_2}
Let $\sigma_1,\sigma_2,\sigma_3>0$, and $c\in \R^d$. 
\begin{itemize}
\item[{\rm (i)}]
    Assume $1\le d\le 5$. If $(\omega_1,\omega_2)\in \Omega_c$ then there exists $p>0$ such that any weak solution $\Phi =(\varphi_1,\varphi_2,\varphi_3) \in X_{\omega_1,\omega_2,c}$ to {\rm (\ref{ellip_sys_tra})}
    satisfies
    \[
    \sum_{j=1}^3\int_{\R^d}e^{2p|x|}(|\varphi_j(x)|^2+|\nabla \varphi_j(x)|^2)dx<\infty. 
    \]
\item[{\rm (ii)}] Assume $3\le d\le 5$ and $1\le k<l\le 3$. 
     If $(\omega_1,\omega_2)\in \partial\Omega_c$ 
     satisfies
     $\omega_{k}-\frac{\sigma_{k}}{4}|c|^2>0$ and $\omega_{l}-\frac{\sigma_{l}}{4}|c|^2>0$, then there exists $p>0$ such that any weak solution $\Phi =(\varphi_1,\varphi_2,\varphi_3) \in X_{\omega_1,\omega_2,c}$ to {\rm (\ref{ellip_sys_tra})}
    satisfies
    \[
    \int_{\R^d}e^{2p|x|}(|\varphi_k(x)|^2+|\nabla \varphi_k(x)|^2)dx<\infty,\ \ \ \ 
    \int_{\R^d}e^{2p|x|}(|\varphi_l(x)|^2+|\nabla \varphi_l(x)|^2)dx<\infty. 
    \]
    \end{itemize}
\end{cor}
\begin{proof}
    Because $\varphi_j(x)=e^{i\frac{\sigma_j}{2}c\cdot x}\psi_j(x)$, we have
    \[
    \nabla \varphi_j(x)=e^{i\frac{\sigma_j}{2}c\cdot x}\left(i\frac{\sigma_j}{2}c\psi_j(x)+\nabla \psi_j(x)\right). 
    \]
    This and Proposition~\ref{ws_decay} yield the conclusion. 
\end{proof}
Pohozaev's identity plays an important role 
in the proof of Theorems~\ref{GWP_1} and ~\ref{stab_thm}. 
\begin{prop}[Pohozaev's identity]\label{Pohoz_id}
Let $\sigma_1,\sigma_2,\sigma_3>0$. 
\begin{itemize}
 \item[{\rm (i)}] Assume $1\le d\le 5$ and $c\in \R^d$. If $(\omega_1,\omega_2)\in \Omega_c$,  
then any weak solution $\Phi\in X_{\omega_1,\omega_2,c}$ to {\rm (\ref{ellip_sys_tra})}
    satisfies
    \[
    2L(\Phi)+\frac{d}{2}N(\Phi)+c\cdot P(\Phi)=0.
    \]
    \item[{\rm (ii)}] Assume $3\le d\le 5$ and $c=0$. 
    If $(\omega_1,\omega_2)\in \partial \Omega_0 \setminus \{(0,0)\}$, 
    namely $\max\{\omega_1,\omega_2\}>0$ and $\min \{\omega_1,\omega_2\}=0$, 
    then any weak solution $\Phi\in X_{\omega_1,\omega_2,0}$ to {\rm (\ref{ellip_sys_tra})}
    satisfies
    \[
    2L(\Phi)+\frac{d}{2}N(\Phi)=0.
    \]
\end{itemize}
\end{prop}
\begin{rem}\label{cp_neg_rem}
As we will see when $d=4$ and $(\omega_1,\omega_2)\in \Omega_c$, $E(\Phi)\ (=L(\Phi)+N(\Phi))\ge 0$ holds 
for any weak solution $\Phi \in X_{\omega_1,\omega_2,c}$ 
to {\rm (\ref{ellip_sys_tra})}.  
This and the Pohozaev's identity imply $c\cdot P(\Phi)=-2E(\Phi)\le 0$. 

     In general for $d=4$, it is known that 
    if $U_0\in \mathcal{H}^1(\R^4)$ satisfies $|x|U_0\in L^2(\R^4)^3$ and  
    the solution $U(t)$ to {\rm (\ref{dnls2})} exists globally in time, 
    then $E(U_0)\ge 0$ holds (see, {\rm Proposition~5.10} in \cite{NP21}). 
    Note that a weak solution $\Phi \in X_{\omega_1,\omega_2,c}$ 
to {\rm (\ref{ellip_sys_tra})} with $(\omega_1,\omega_2)\in \Omega_c$ satisfies 
    $\Phi \in \mathcal{H}^1(\R^4)$ and $|x|\Phi\in L^2(\R^4)^3$ when $(\omega_1,\omega_2)\in \Omega_c$ 
    by Corollary~\ref{ws_decay_2}. 
    In particular, the global solution to {\rm (\ref{dnls2})} with $U_0=\Phi$ exists, which 
    is given by {\rm (\ref{2req_para_sol})}. 
    Therefore, $E(\Phi)\ge 0$ holds. 
\end{rem}
To prove Proposition~\ref{Pohoz_id}, 
we give the following lemma. 
\begin{lemm}\label{lap_decay_est}
Assume that either condition {\rm (i)} or {\rm (ii)} in Proposition~\ref{Pohoz_id} holds. Then any weak solution $\Phi=(\varphi_1,\varphi_2,\varphi_3)\in X_{\omega_1,\omega_2,c}$ to {\rm (\ref{ellip_sys_tra})} satisfies $(x\cdot \nabla)\Phi\in X_{\omega_1,\omega_2,c}$. 
\end{lemm}
\begin{proof}
We note that
\[
\partial_m((x\cdot \nabla)\varphi_j)
=\partial_m\varphi_j+\sum_{n=1}^dx_n\partial_n\partial_m\varphi_j
\]
and
\[
\begin{split}
\|x_n\partial_n\partial_m\varphi_j\|_{L^2}^2
&\sim \|\partial_{\xi_n}(\xi_n\xi_m\widehat{\varphi_j}(\xi))\|_{L^2_{\xi}}^2
\lesssim \|\xi_m\widehat{\varphi_j}(\xi)\|_{L^2_{\xi}}^2+\|\xi_n^2\partial_{\xi_n}\widehat{\varphi_j}(\xi)\|_{L^2_{\xi}}^2
+\|\xi_m^2\partial_{\xi_n}\widehat{\varphi_j}(\xi)\|_{L^2_{\xi}}^2\\
&\lesssim \|\xi\widehat{\varphi_j}(\xi)\|_{L^2_{\xi}}^2+\||\xi|^2\partial_{\xi_n}\widehat{\varphi_j}(\xi)\|_{L^2_{\xi}}^2
\lesssim \|\xi\widehat{\varphi_j}(\xi)\|_{L^2_{\xi}}^2+\|\partial_{\xi_n}(|\xi|^2\widehat{\varphi_j}(\xi))\|_{L^2_{\xi}}^2\\
&\lesssim \|\nabla \varphi_j\|_{L^2}^2+\||x|\Delta \varphi_j\|_{L^2}^2
\end{split}
\]
by the Plancherel theorem. 

We first consider the case {\rm (i)}. 
Then $X_{\omega_1,\omega_2,c}=\mathcal{H}^1(\R^d)$ 
    and any weak solution $\Phi =(\varphi_1,\varphi_2,\varphi_3)\in X_{\omega_1,\omega_2,c}$
    to {\rm (\ref{ellip_sys_tra})}
    satisfies 
    \[
    \||x|\Delta \varphi_j\|_{L^2}\le \sigma_j\left(\omega_j\||x|\varphi_j\|_{L^2}
    +|c|\||x|\nabla \varphi_j\|_{L^2}+\||x|F_j(\Phi)\|_{L^2}\right). 
    \]
    Because $\||x|\varphi_k\varphi_l\|_{L^2}\le \||x|\varphi_k\|_{L^2}\|\varphi_l\|_{L^{\infty}}$ for $1\le k,l\le 3$, 
    it holds
    $|x|\Delta \varphi_j\in L^2(\R^d)$ for $j=1,2,3$ by 
    Corollary~\ref{ws_decay_2} (see, also {\rm Remark~\ref{reg_sol_rem}}). 
    Therefore, we have $\partial_m(x\cdot \nabla)\varphi_j \in L^2(\R^d)$ for $j=1,2,3$. 
    
    Next, we consider the case {\rm (ii)}. 
    By the symmetry, we only have to consider the case $\omega_1>0$ and $\omega_2=0$. 
    Then, $\omega_3=\omega_1+\omega_2=\omega_1>0$. 
    Therefore, it holds $\varphi_3\in L^{\infty}(\R^d)$ (see, {\rm Remark~\ref{reg_sol_rem}}), and we obtain
    \[
    \||x|\Delta \varphi_2\|_{L^2(\R^d)}=\||x|\varphi_1\varphi_3\|_{L^2(\R^d)}
    \le \||x|\varphi_1\|_{L^2(\R^d)}\|\varphi_3\|_{L^{\infty}(\R^d)}
    \]
    by the second equation of (\ref{ellip_sys_tra}). 
    On the other hand, by the first equation of (\ref{ellip_sys_tra}) and the H\"older inequality, we have
    \[
    \||x|\Delta \varphi_1\|_{L^2}
    \lesssim \||x|\varphi_1\|_{L^2}+\||x|\varphi_2\varphi_3\|_{L^2}
    \le \||x|\varphi_1\|_{L^2}+\|\varphi_2\|_{L^p}\||x|\varphi_3\|_{L^q}
    \]
    for $p,q\ge 2$ with $\frac{1}{2}=\frac{1}{p}+\frac{1}{q}$. 
    When $3\le d\le 4$, we choose $p=\frac{2d}{d-2}$, $q=d$. 
    Then, we have
    \[
    \|\varphi_2\|_{L^p}\||x|\varphi_3\|_{L^q}=\|\varphi_2\|_{L^{\frac{2d}{d-2}}}\||x|\varphi_3\|_{L^d}
    \lesssim \|\varphi_2\|_{\dot{H}^1}\||x|\varphi_3\|_{H^1}
    \]
    by the Sobolev embeddings $\dot{H}^1(\R^d)\hookrightarrow L^{\frac{2d}{d-2}}(\R^d)$ and 
    $H^1(\R^d)\hookrightarrow L^{d}(\R^d)$. 
    When $d=5$, we choose $p=10$, $q=\frac{10}{3}$. 
    Then, we have
    \[
    \|\varphi_2\|_{L^p}\||x|\varphi_3\|_{L^q}=\|\varphi_2\|_{L^{10}}\||x|\varphi_3\|_{L^{\frac{10}{3}}}
    \lesssim \|\varphi_2\|_{\dot{H}^2}\||x|\varphi_3\|_{H^1}
    \]
    by the Sobolev embeddings $\dot{H}^2(\R^5)\hookrightarrow L^{10}(\R^5)$ and 
   $H^1(\R^5)\hookrightarrow L^{\frac{10}{3}}(\R^5)$. 
    As a result, we obtain $|x|\Delta\varphi_j\in L^2(\R^d)$ for $j=1,2$ by Proposition~\ref{regu_prop_ws} and Corollary~\ref{ws_decay_2}. 
    We can also get $|x|\Delta \varphi_3\in L^2(\R^d)$ by the same argument for $j=1$. 
    Therefore, we have $\partial_m(x\cdot \nabla)\varphi_j\in L^2(\R^d)$ for $j=1,2,3$. 
\end{proof}
\begin{proof}[Proof of Proposition~\ref{Pohoz_id}]
    Let $\Phi=(\varphi_1,\varphi_2,\varphi_3) \in X_{\omega_1,\omega_2,c}$ be a weak solution. For $\lambda >0$, we put
    \[
    \Phi_{\lambda}(x)=\lambda^{\frac{d}{2}} \Phi (\lambda x).
    \]
    Note that
    \[
    L(\Phi_{\lambda})=\lambda^{2}L(\Phi),\ 
    N(\Phi_{\lambda})=\lambda^{\frac{d}{2}}N(\Phi),\ 
    Q_{1}(\Phi_{\lambda})=Q_1(\Phi),\ Q_2(\Phi_{\lambda})=Q_2(\Phi),\ 
    P(\Phi_{\lambda})=\lambda P(\Phi).
    \]
    Therefore, by the definition of $S_{\omega_1,\omega_2,c}$, we have
    \[
    \begin{split}
        \partial_{\lambda}S_{\omega_1,\omega_2,c}(\Phi_{\lambda})\bigl|_{\lambda=1}
        &=\partial_{\lambda}\left(L(\Phi_{\lambda})+N(\Phi_{\lambda})
        +\omega_1Q_1(\Phi_{\lambda})+\omega_2Q_2(\Phi_{\lambda})+c\cdot P(\Phi_{\lambda})\right)\bigl| _{\lambda=1}\\
        &=2L(\Phi)+\frac{d}{2}N(\Phi)+c\cdot P(\Phi). 
    \end{split}
    \]
    On the other hand, by using the chain rule, we have
    \[
    \partial_{\lambda}S_{\omega_1,\omega_2,c}(\Phi_{\lambda})\bigl|_{\lambda=1}
    =\sum_{j=1}^3\langle D_j S_{\omega_1,\omega_2,c}(\Phi_{\lambda}),\partial_{\lambda }\Phi_{\lambda}\rangle \bigl|_{\lambda=1}
    =\sum_{j=1}^3\langle D_j S_{\omega_1,\omega_2,c}(\Phi),\frac{d}{2}\varphi_j+(x\cdot \nabla)\varphi_j\rangle. 
    \]
    Because $\frac{d}{2}\Phi+(x\cdot \nabla)\Phi\in X_{\omega_1,\omega_2,c}$ by 
    Lemma~\ref{lap_decay_est} and $S_{\omega_1,\omega_2,c}'(\Phi)=0$, we obtain
    \[
    \partial_{\lambda}S_{\omega_1,\omega_2,c}(\Phi_{\lambda})\bigl|_{\lambda=1}=0.
    \]
    
\end{proof}
\section{Existence of ground states}
In this section, we give the proof of Theorem~\ref{ex_gs_1}, 
which is based on \cite{FHI24} and \cite{Lipre} (see, also \cite{HI24}). 
We first prove the positivity 
of the minimal action and the property of Nehari functional. 
\begin{prop}\label{pos_pr_Neh}
Let $1\le d\le 5$, $\sigma_1,\sigma_2,\sigma_3>0$, and $c\in \R^d$. 
Assume one of the conditions {\rm (A)}, {\rm (B)} in Theorem~\ref{ex_gs_1}, and
\[
\begin{split}
{\rm (C)'}\ \ d=4\ {\rm or}\ 5,\ c\ne 0,\ &(\omega_1,\omega_2)\in \partial \Omega_c,\ (\omega_i,\omega_j)=\left(\frac{\sigma_i}{4}|c|^2,\frac{\sigma_j}{4}|c|^2\right)\ {\rm for\ some}\  1\le i<j\le 3,\\
         &\ {\rm and}\ (\omega_1,\omega_2,\omega_3)\ne \left(\frac{\sigma_1}{4}|c|^2,\frac{\sigma_2}{4}|c|^2,\frac{\sigma_3}{4}|c|^2\right) 
\end{split}
\]
holds. Then, the following properties hold:
\begin{itemize}
    \item[{\rm (i)}]\ If $\Psi \in \widetilde{X}_{\omega_1,\omega_2,c}\setminus \{(0,0,0)\}$ 
    satisfies $\widetilde{K}_{\omega_1,\omega_2,c}(\Psi)=0$, 
    then $\|\Psi\|_{\widetilde{X}_{\omega_1,\omega_2,c}}\gtrsim 1$, 
    where the implicit constant does not depend on $\omega_1,\omega_2,c$. 
    In particular, $\widetilde{\mathfrak{m}}_{\omega_1,\omega_2,c}>0$. 
    \item[{\rm (ii)}]\ If $V\in \widetilde{X}_{\omega_1,\omega_2,c}\setminus\{(0,0,0)\}$ satisfies
    $\widetilde{K}_{\omega_1,\omega_2,c}(V)<0$ (resp. $\le 0$), 
    then $\|V\|_{\widetilde{X}_{\omega_1,\omega_2,c}}^2>6\widetilde{\mathfrak{m}}_{\omega_1,\omega_2,c}$ (resp. $\ge 6\widetilde{\mathfrak{m}}_{\omega_1,\omega_2,c}$). 
\end{itemize}
\end{prop}
\begin{proof}
{\rm (i)}\ Assume $\Psi \in \widetilde{X}_{\omega_1,\omega_2,c}\setminus \{(0,0,0)\}$ satisfies $\widetilde{K}_{\omega_1,\omega_2,c}(\Psi)=0$. 
Then, by the definition of the Nehari functional, we obtain
\begin{equation}\label{norm_nonl_int}
\|\Psi\|_{\widetilde{X}_{\omega_1,\omega_2,c}}^2=\widetilde{K}_{\omega_1,\omega_2,c}(\Psi)-3N_c(\Psi)
=-3N_c(\Psi)\le 3\int_{\R^d}|\psi_1\psi_2\psi_3|dx. 
\end{equation}
We use the H\"older inequality and the Sobolev embedding as follows:
\[
\begin{cases}
    H^1(\R^d)\times H^1(\R^d)\times H^1(\R^d)\hookrightarrow L^3(\R^d)\times L^3(\R^d)\times L^3(\R^d)
     \ \ {\rm when\ (A)\ holds}, \\
     \dot{H}^1(\R^d)\times H^1(\R^d)\times H^1(\R^d)\hookrightarrow L^{\frac{2d}{d-2}}(\R^d)\times L^{\frac{4d}{d+2}}(\R^d)\times L^{\frac{4d}{d+2}}(\R^d)
     \ \ {\rm when\ (B)\ holds},\\
     \dot{H}^1(\R^d)\times \dot{H}^1(\R^d)\times H^1(\R^d)\hookrightarrow L^{\frac{2d}{d-2}}(\R^d)\times L^{\frac{2d}{d-2}}(\R^d)\times L^{\frac{d}{2}}(\R^d)
     \ \ {\rm when\ (C)'\ holds}. 
\end{cases}
\]
As a result, we have
\[
\int_{\R^d}|\psi_1\psi_2\psi_3|dx\lesssim \|\Psi\|_{\widetilde{X}_{\omega_1,\omega_2,c}}^3, 
\]
where the implicit constant does not depend on $\omega_1,\omega_2,c$. 
This and {\rm (\ref{norm_nonl_int})} imply $\|\Psi\|_{\widetilde{X}_{\omega_1,\omega_2,c}}\gtrsim 1$. 
Therefore, by {\rm (\ref{S_K_N_rel})}, it holds
\[
\widetilde{S}_{\omega_1,\omega_2,c}(\Psi)=\frac{1}{6}\|\Psi\|_{\widetilde{X}_{\omega_1,\omega_2,c}}^2\gtrsim 1
\]
under the conditions $\Psi \ne 0$ and $\widetilde{K}_{\omega_1,\omega_2,c}(\Psi)=0$. 
Namely $\widetilde{\mathfrak{m}}_{\omega_1,\omega_2,c}>0$. \\

{\rm (ii)} Let $V\in \widetilde{X}_{\omega_1,\omega_2,c}\setminus \{(0,0,0)\}$ 
and assume $\widetilde{K}_{\omega_1,\omega_2,c}(V)\le 0$. 
For $\lambda >0$, we have
\[
\widetilde{K}_{\omega_1,\omega_2,c}(\lambda V)=\|\lambda V\|_{\widetilde{X}_{\omega_1,\omega_2,c}}^2+3N_c(\lambda V)
=\lambda^2\|V\|_{\widetilde{X}_{\omega_1,\omega_2,c}}^2+3\lambda^3N_c(V). 
\]
Note that $N_c(V)\ne 0$ because
\[
3N_c(V)=\widetilde{K}_{\omega_1,\omega_2,c}(V)-\|V\|_{\widetilde{X}_{\omega_1,\omega_2,c}}^2
\le -\|V\|_{\widetilde{X}_{\omega_1,\omega_2,c}}^2<0. 
\]
Therefore, by choosing
\[
\lambda =\frac{\|V\|_{\widetilde{X}_{\omega_1,\omega_2,c}}^2}{-3N_c(V)}, 
\]
we have $\widetilde{K}_{\omega_1,\omega_2,c}(\lambda V)=0$. 
Therefore, by {\rm (\ref{S_K_N_rel})}, we obtain
\[
\widetilde{\mathfrak{m}}_{\omega_1,\omega_2,c}\le 
\widetilde{S}_{\omega_1,\omega_2,c}(\lambda V)=\frac{1}{6}\|\lambda V\|_{\widetilde{X}_{\omega_1,\omega_2,c}}^2
=\frac{\lambda^2}{6}\|V\|_{\widetilde{X}_{\omega_1,\omega_2,c}}^2. 
\]
Since $\lambda \le 1$ holds when $\widetilde{K}_{\omega_1,\omega_2,c}(V)\le 0$ 
and $\lambda < 1$ holds when $\widetilde{K}_{\omega_1,\omega_2,c}(V)< 0$, 
we have the conclusion. 
\end{proof}
The following remark plays an important role to prove Theorem~\ref{stab_thm}. 
\begin{rem}\label{K_m_rel_U}
    When $U=\Lambda_cV$, it holds that 
    $\widetilde{K}_{\omega_1,\omega_2,c}(V)=K_{\omega_1,\omega_2,c}(U)$ and
$\|V\|_{\widetilde{X}_{\omega_1,\omega_2,c}}^2=\|U\|_{X_{\omega_1,\omega_2,c}}^2$.
    Therefore, by {\rm Proposition~\ref{pos_pr_Neh} (i)} and Remark~\ref{X_norm_rep_LQP}, 
    there exists $C>0$ such that
    \[2L(\Phi)+2\omega_1Q_1(\Phi)+2\omega_2Q_2(\Phi)+2c\cdot P(\Phi)\ge C
    \]
    holds 
    for any $\Phi\in X_{\omega_1,\omega_2,c}\setminus \{(0,0,0)\}$ with $K_{\omega_1,\omega_2,c}(\Phi)=0$. 
    Furthermore, {\rm Proposition~\ref{pos_pr_Neh} (ii)} 
    is equivalent to the following:

    If $U\in X_{\omega_1,\omega_2,c}\setminus\{(0,0,0)\}$ satisfies
    $K_{\omega_1,\omega_2,c}(U)<0$ (resp. $\le 0$), 
    then $L(U)+\omega_1Q_1(U)+\omega_2Q_2(U)+c\cdot P(U)>3\mathfrak{m}_{\omega_1,\omega_2,c}$ (resp. $\ge 3\mathfrak{m}_{\omega_1,\omega_2,c}$). 
\end{rem}
\begin{prop}\label{GM_subset_rel}
      Let $1\le d\le 5$, $\sigma_1,\sigma_2,\sigma_3>0$, and $c\in \R^d$. 
    Assume one of the conditions {\rm (A)}, {\rm (B)} in Theorem~\ref{ex_gs_1}, and ${\rm (C)'}$ in Proposition~\ref{pos_pr_Neh} holds. 
    If $\widetilde{\mathcal{M}}_{\omega_1,\omega_2,c}\ne \emptyset$, then we have $\widetilde{\mathcal{M}}_{\omega_1,\omega_2,c}=\widetilde{\mathcal{G}}_{\omega_1,\omega_2,c}$. 
\end{prop}
\begin{proof}
    First, we prove $\widetilde{\mathcal{M}}_{\omega_1,\omega_2,c}\subset \widetilde{\mathcal{G}}_{\omega_1,\omega_2,c}$. 
    By the assumption $\widetilde{\mathcal{M}}_{\omega_1,\omega_2,c}\ \ne \emptyset$, we can take $\Psi \in \widetilde{\mathcal{M}}_{\omega_1,\omega_2,c}\ $. 
    Then, we have $\widetilde{S}_{\omega_1,\omega_2,c}(\Psi)=\widetilde{\mathfrak{m}}_{\omega_1,\omega_2,c}$ and 
    $\widetilde{K}_{\omega_1,\omega_2,c}(\Psi)=0$. By the Lagrange multiplier theorem there exists $\eta \in \R$ such that
    \[
    D_j\widetilde{S}_{\omega_1,\omega_2,c}(\Psi)=\eta D_j \widetilde{K}_{\omega_1,\omega_2,c}(\Psi)\ \ (j=1,2,3). 
    \]
    Therefore, by the definition of $\widetilde{K}_{\omega_1,\omega_2,c}$, we have
    \[
    0=\widetilde{K}_{\omega_1,\omega_2,c}(\Psi)=\sum_{j=1}^3\langle D_j\widetilde{S}_{\omega_1,\omega_2,c}(\Psi),\Psi\rangle 
    =\eta \sum_{j=1}^3\langle D_j\widetilde{K}_{\omega_1,\omega_2,c}(\Psi),\Psi\rangle
    =\eta \partial_{\lambda}\widetilde{K}_{\omega_1,\omega_2,c}(\lambda \Psi)|_{\lambda =1} 
    \]
    and
    \[
    \partial_{\lambda}\widetilde{K}_{\omega_1,\omega_2,c}(\lambda \Psi)|_{\lambda =1}
    =2\|\Psi\|_{\widetilde{X}_{\omega_1,\omega_2,c}}^2+9N_c(\Psi)
    =3\widetilde{K}_{\omega_1,\omega_2,c}(\Psi)-\|\Psi\|_{\widetilde{X}_{\omega_1,\omega_2,c}}^2
    =-\|\Psi\|_{\widetilde{X}_{\omega_1,\omega_2,c}}^2<0. 
    \]
    Hence $\eta =0$, and we obtain $\widetilde{S}_{\omega_1,\omega_2,c}'(\Psi)=0$. 
    Because $\Psi \ne 0$, it implies $\Psi \in \widetilde{\mathcal{E}}_{\omega_1,\omega_2,c}$ 
    and also we have $\Psi \in \widetilde{\mathcal{G}}_{\omega_1,\omega_2,c}$. 
    Indeed, it holds
    \[
    \widetilde{S}_{\omega_1,\omega_2,c}(\Psi)=\widetilde{\mathfrak{m}}_{\omega_1,\omega_2,c}\le \widetilde{S}_{\omega_1,\omega_2,c}(\Theta)
    \]
    for any $\Theta \in \widetilde{\mathcal{E}}_{\omega_1,\omega_2,c}$ 
    because $\widetilde{K}_{\omega_1,\omega_2,c}(\Theta)=0$ is satisfied by Remark~\ref{nehari_weak_sol}.  

    Next, we prove $\widetilde{\mathcal{G}}_{\omega_1,\omega_2,c}\subset \widetilde{\mathcal{M}}_{\omega_1,\omega_2,c}$. 
    Note that $\widetilde{\mathcal{G}}_{\omega_1,\omega_2,c}\ne \emptyset$ holds 
    by $\emptyset \ne \widetilde{\mathcal{M}}_{\omega_1,\omega_2,c}\subset \widetilde{\mathcal{G}}_{\omega_1,\omega_2,c}$. 
    Let $\Psi \in \widetilde{\mathcal{G}}_{\omega_1,\omega_2,c}$. 
    Then $\widetilde{K}_{\omega_1,\omega_2,c}(\Psi)=0$ holds (see, Remark~\ref{nehari_weak_sol}). 
    Therefore, we have 
    \[
    \widetilde{\mathfrak{m}}_{\omega_1,\omega_2,c}\le \widetilde{S}_{\omega_1,\omega_2,c}(\Psi)
    \]
    by the definition of $\widetilde{\mathfrak{m}}_{\omega_1,\omega_2,c}$. 
    On the other hand, there exists $\Psi_0\in \widetilde{\mathcal{M}}_{\omega_1,\omega_2,c}$ 
    because $\widetilde{\mathcal{M}}_{\omega_1,\omega_2,c}\ne \emptyset$,  
    and we obtain $\Psi_0\in \widetilde{\mathcal{G}}_{\omega_1,\omega_2,c}$ 
    by $\widetilde{\mathcal{M}}_{\omega_1,\omega_2,c}\subset \widetilde{\mathcal{G}}_{\omega_1,\omega_2,c}$. Hence, it holds that
    \[
    \widetilde{S}_{\omega_1,\omega_2,c}(\Psi)\le \widetilde{S}_{\omega_1,\omega_2,c}(\Psi_0)=\widetilde{\mathfrak{m}}_{\omega_1,\omega_2,c}. 
    \]
    Thus $\widetilde{S}_{\omega_1,\omega_2,c}(\Psi)=\widetilde{\mathfrak{m}}_{\omega_1,\omega_2,c}$, which implies $\Psi\in \widetilde{\mathcal{M}}_{\omega_1,\omega_2,c}$.    
\end{proof}
\begin{rem}\label{GM_subset_rel_rem}
    We also obtain $\mathcal{M}_{\omega_1,\omega_2,c}=\mathcal{G}_{\omega_1,\omega_2,c}$ 
    because $\Phi\in \mathcal{M}_{\omega_1,\omega_2,c}$ is equivalent to $\Psi\in \widetilde{\mathcal{M}}_{\omega_1,\omega_2,c}$ 
    and $\Phi\in \mathcal{G}_{\omega_1,\omega_2,c}$ is equivalent to $\Psi\in \widetilde{\mathcal{G}}_{\omega_1,\omega_2,c}$ 
    when $\Phi =\Lambda_c\Psi$. 
\end{rem}
To construct a ground state solution, 
we use the following compactness result. 
\begin{lemm}[Lieb's compactness theorem {\rm \cite{Lieb83}}]\label{lieb_cpt_lemm}
Let $\{F_n\}$ be a bounded sequence in $H^1(\R^d)$, 
and suppose that $\limsup_{n\rightarrow \infty}\|F_n\|_{L^p}>0$ 
for some $p\in (2,2^*)$, where $2^*=\infty$ if $d=1$ or $2$, and $2^*=\frac{2d}{d-2}$ if $d\ge 3$. 
Then, there exist $\{x_n\}\subset \R^d$ and $F\in H^1(\R^d)\setminus \{0\}$ 
such that $\{F_n(\cdot -x_n)\}$ has a subsequence that converges to $F$ 
weakly in $H^1(\R^d)$. 
\end{lemm}
Now we prove the convergence of the minimizing sequence. 
For $\tau\in \mathbb{R}$, we define the operator $\Lambda (\tau)$ on $\widetilde{X}_{\omega_1,\omega_2,c}$ by
\begin{equation}\label{phase_tr_def_2}
\Lambda (\tau)V:=(e^{i\frac{\sigma_1}{2}\tau}v_1,e^{i\frac{\sigma_2}{2}\tau}v_2,e^{i\frac{\sigma_3}{2}\tau}v_3)
\end{equation}
for $V=(v_1,v_2,v_3)$. Note that $(\Lambda_cV)(x)=\Lambda (c\cdot x)V(x)$, 
where $\Lambda_c$ is defined in (\ref{op_lambdac_def}). 

Thanks to Proposition~\ref{GM_subset_rel}, 
it suffices to show the following proposition to obtain Theorem~\ref{ex_gs_1}.
\begin{prop}\label{mini_seq_conv}
   Let $1\le d\le 5$, $\sigma_1,\sigma_2,\sigma_3>0$, and $c\in \R^d$. 
    Assume one of the conditions {\rm (A)}, {\rm (B)}, and {\rm (C)} in Theorem~\ref{ex_gs_1} holds. 
    If $\{V_n\}\subset \widetilde{X}_{\omega_1,\omega_2,c}$ satisfies
    \begin{equation}\label{minim_seq_cond}
    \lim_{n\rightarrow \infty}\widetilde{S}_{\omega_1,\omega_2,c}(V_n)=\widetilde{\mathfrak{m}}_{\omega_1,\omega_2,c},\ \ \lim_{n\rightarrow \infty}\widetilde{K}_{\omega_1,\omega_2,c}(V_n)=0,
    \end{equation}
    then there exist $\{x_n\}\subset \R^d$ and $W\in \widetilde{X}_{\omega_1,\omega_2,c}\setminus \{(0,0,0)\}$
    such that $\{\Lambda (-c\cdot x_n)V_n(\cdot -x_n)\}$ 
    has a subsequence that converges to $W$ strongly in $\widetilde{X}_{\omega_1,\omega_2,c}$. 
    Furthermore, $W\in \widetilde{\mathcal{M}}_{\omega_1,\omega_2,c}$ holds. 
\end{prop}
\begin{proof}
Write $\widetilde{X}_{\omega_1,\omega_2,c}=\widetilde{X}_1\times \widetilde{X}_2\times \widetilde{X}_3$. 
 Under the conditions {\rm (A)}, {\rm (B)}, or {\rm (C)}, 
 at least one of $\widetilde{X}_j$ is equal to $H^1(\R^d)$. 
 We only assume the case $\widetilde{X}_{3}=H^1(\R^d)$ because the other cases are the same. 
 Let $V_n=(v_{1n},v_{2n},v_{3n})$ satisfy {\rm (\ref{minim_seq_cond})}. \\

\noindent {\bf Step\ 1}\ (Applying Lieb's compactness theorem)
 
    By {\rm (\ref{S_K_N_rel})}, we have
    \[
    \|V_n\|_{\widetilde{X}_{\omega_1,\omega_2,c}}^2=6\widetilde{S}_{\omega_1,\omega_2,c}(V_n)-2\widetilde{K}_{\omega_1,\omega_2,c}(V_n)\ \rightarrow\ 6\widetilde{\mathfrak{m}}_{\omega_1,\omega_2,c}\ \ (n\rightarrow \infty). 
    \]
    Therefore, $\{V_n\}$ is a bounded sequence in $\widetilde{X}_{\omega_1,\omega_2,c}$. 
    In particular, $\{v_{3n}\}$ is a bounded sequence in $H^1(\R^d)$. 
    Furthermore, by the H\"older inequality, the Sobolev embedding
    \[
    {\rm (SE)}\ 
\begin{cases}
    H^1(\R^d)\times H^1(\R^d)\times H^1(\R^d)\hookrightarrow L^3(\R^d)\times L^3(\R^d)\times L^3(\R^d)
     \ \ {\rm when\ (A)\ holds}, \\
     \dot{H}^1(\R^d)\times H^1(\R^d)\times H^1(\R^d)\hookrightarrow L^{\frac{2d}{d-2}}(\R^d)\times L^{\frac{4d}{d+2}}(\R^d)\times L^{\frac{4d}{d+2}}(\R^d)
     \ \ {\rm when\ (B)\ holds},\\
     \dot{H}^1(\R^d)\times \dot{H}^1(\R^d)\times H^1(\R^d)\hookrightarrow L^{\frac{2d}{d-2}}(\R^d)\times L^{\frac{2d}{d-2}}(\R^d)\times L^{\frac{d}{2}}(\R^d)
     \ \ {\rm when\ (C)\ holds}, 
\end{cases}
    \]
    and the boundedness of $\{V_n\}$, we have
    \[
    -3N_c(V_n)\lesssim \|V_n\|_{\widetilde{X}_{\omega_1,\omega_2,c}}^2\|v_{3n}\|_{L^p}
    \lesssim \|v_{3n}\|_{L^p}
    \]
    for $p=3$ when {\rm (A)} holds, $p=\frac{4d}{d+2}$ when {\rm (B)} holds, 
    and $p=\frac{d}{2}$ when {\rm (C)} holds. 
    Therefore, by the definition of $\widetilde{K}_{\omega_1,\omega_2,c}$, {\rm (\ref{S_K_N_rel})}, and Proposition~\ref{pos_pr_Neh}, we obtain
\[
\|v_{3n}\|_{L^p}\gtrsim 6\widetilde{S}_{\omega_1,\omega_2,c}(V_n)-3\widetilde{K}_{\omega_1,\omega_2,c}(V_n)\ \rightarrow\ 6\widetilde{\mathfrak{m}}_{\omega_1,\omega_2,c}>0\ \ (n\rightarrow \infty). 
\]
Because $p\in (2,2^*)$ holds under the conditions {\rm (A)}, {\rm (B)}, or {\rm (C)}, 
we can apply Lemma~\ref{lieb_cpt_lemm}. 
(Note that if $d=4$, then $\frac{d}{2}=2$. 
Therefore, we cannot apply Lemma~\ref{lieb_cpt_lemm} under the condition ${\rm (C)'}$.)
Namely, there exist $\{x_n\}\subset \R^d$ and $w_3\in H^1(\R^d)\setminus \{0\}$ 
such that after passing to a subsequence, 
\[
w_{3n}(x):=e^{-i\frac{\sigma_3}{2}c\cdot x_n}v_{3n}(x-x_n)\ \rightharpoonup\ w_3(x)\ \ {\rm in}\ \ H^1(\R^d). 
\]
Here ``$\rightharpoonup$'' denotes the weak convergence. 
For the same $x_n$, we define
\[
w_{jn}(x):=e^{-i\frac{\sigma_j}{2}c\cdot x_n}v_{jn}(x-x_n)\ \ \ (j=1,2). 
\]
Since $\|w_{jn}\|_{\widetilde{X}_j}=\|v_{jn}\|_{\widetilde{X}_j}$ holds 
for both cases $\widetilde{X}_j=H^1(\R^d)$ and $\widetilde{X}_j=\dot{H}^1(\R^d)$, 
$\{w_{jn}\}$ is a bounded sequence in $\widetilde{X}_{j}$.  
Therefore, there exists $w_j\in \widetilde{X}_j$ such that up to a subsequence, 
$w_{jn}$ converges to $w_j$ weakly in $\widetilde{X}_j$. 
We put $W:=(w_1,w_2,w_3)$, which is not zero in $\widetilde{X}_{\omega_1,\omega_2,c}$ 
because $w_3\ne 0$. 
Then, after passing to a subsequence, $W_n:=(w_{1n},w_{2n},w_{3n})$ (we also write the subsequence as ``$W_n$'') 
converges to $W$ weakly in $\widetilde{X}_{\omega_1,\omega_2,c}$. \\

\noindent {\bf Step\ 2}\ (Convergence of functionals)

Now, we prove
\begin{equation}\label{nonlin_func_conv}
N_c(W_n)-N_c(W_n-W)-N_c(W)\ \rightarrow\ 0\ (n\rightarrow \infty). 
\end{equation}
We put $r_{jn}:=w_{jn}-w_j$\ $(j=1,2,3)$. 
Then, $r_{jn}$ converges to $0$ weakly in $\widetilde{X}_j$. 
By the definition of $N_c$, we have
\begin{equation}\label{nonlin_expa_calc}
\begin{split}
&N_c(W_n)-N_c(W_n-W)-N_c(W)\\
&=-\left\{{\rm Re}\left(e^{i\frac{\mu}{2}c\cdot x}w_{1n}w_{2n},w_{3n}\right)_{L^2}
-{\rm Re}\left(e^{i\frac{\mu}{2}c\cdot x}r_{1n}r_{2n},r_{3n}\right)_{L^2}
-{\rm Re}\left(e^{i\frac{\mu}{2}c\cdot x}w_{1}w_{2},w_{3}\right)_{L^2}\right\}\\
&= -{\rm Re}\left\{\left(r_{1n}r_{2n},e^{-i\frac{\mu}{2}c\cdot x}w_3\right)_{L^2}+\left(r_{1n}\overline{r_{3n}},e^{-i\frac{\mu}{2}c\cdot x}\overline{w_{2}}\right)_{L^2}+\left(r_{2n}\overline{r_{3n}},e^{-i\frac{\mu}{2}c\cdot x}\overline{w_{1}}\right)_{L^2}\right.\\
&\ \ \ \ \ \ \ \ \ \ \ \ \ \ \ \ \ \ \ \ \ \ \ \ \ \ \ \ \ \ \ \ \left.+\left(r_{1n},e^{-i\frac{\mu}{2}c\cdot x}\overline{w_2}w_3\right)_{L^2}+\left(r_{2n},e^{-i\frac{\mu}{2}c\cdot x}\overline{w_1}w_3\right)_{L^2}+\left(\overline{r_{3n}},e^{-i\frac{\mu}{2}c\cdot x}\overline{w_1}\overline{w_2}\right)_{L^2}\right\}. 
\end{split}
\end{equation}
We recall that $\mu=\sigma_1+\sigma_2-\sigma_3$. 
By the Sobolev embedding {\rm (SE)} as above, we have
the weak convergence
\[
(r_{1n},r_{2n},r_{3n})\ \rightharpoonup\ (0,0,0)\ {\rm in}\ 
\begin{cases}
L^3(\R^d)\times L^3(\R^d)\times L^3(\R^d)
     \ \ {\rm when\ (A)\ holds}, \\
L^{\frac{2d}{d-2}}(\R^d)\times L^{\frac{4d}{d+2}}(\R^d)\times L^{\frac{4d}{d+2}}(\R^d)
     \ \ {\rm when\ (B)\ holds},\\
L^{\frac{2d}{d-2}}(\R^d)\times L^{\frac{2d}{d-2}}(\R^d)\times L^{\frac{d}{2}}(\R^d)
     \ \ {\rm when\ (C)\ holds}
\end{cases}
\]
and the weak convergence
\[
(r_{1n}r_{2n},r_{1n}\overline{r_{3n}},r_{2n}\overline{r_{3n}})\ \rightharpoonup\ (0,0,0)\ {\rm in}\ 
\begin{cases}
L^{\frac{3}{2}}(\R^d)\times L^{\frac{3}{2}}(\R^d)\times L^{\frac{3}{2}}(\R^d)
     \ \ {\rm when\ (A)\ holds}, \\
L^{\frac{4d}{3d-2}}(\R^d)\times L^{\frac{4d}{3d-2}}(\R^d)\times L^{\frac{2d}{d+2}}(\R^d)
     \ \ {\rm when\ (B)\ holds},\\
L^{\frac{d}{d-2}}(\R^d)\times L^{\frac{2d}{d+2}}(\R^d)\times L^{\frac{2d}{d+2}}(\R^d)
     \ \ {\rm when\ (C)\ holds}.  
\end{cases}
\]
Therefore, all the inner products in the right-hand side of {\rm (\ref{nonlin_expa_calc})} 
converge to $0$. 

On the other hand, by the definition of the norm $\|\cdot\|_{\widetilde{X}_{\omega_1,\omega_2,c}}$, 
we obtain
\[
\begin{split}
    &\|W_n\|_{\widetilde{X}_{\omega_1,\omega_2,c}}^2
    -\|W_n-W\|_{\widetilde{X}_{\omega_1,\omega_2,c}}^2
    -\|W\|_{\widetilde{X}_{\omega_1,\omega_2,c}}^2\\
    &=\sum_{j=1}^3\frac{2}{\sigma_j}{\rm Re}\left(\nabla r_{jn},\nabla w_j\right)_{L^2}
    +\sum_{j=1}^3\left(2\omega_j-\frac{\sigma_j}{2}|c|^2\right){\rm Re}(r_{jn},w_j)_{L^2}. 
\end{split}
\]
Therefore, the weak convergence of $r_{jn}$ implies
\begin{equation}\label{norm_func_conv}
\|W_n\|_{\widetilde{X}_{\omega_1,\omega_2,c}}^2
    -\|W_n-W\|_{\widetilde{X}_{\omega_1,\omega_2,c}}^2
    -\|W\|_{\widetilde{X}_{\omega_1,\omega_2,c}}^2\ \rightarrow\ 0\ (n\rightarrow \infty). 
\end{equation}
By {\rm (\ref{nonlin_func_conv})}, {\rm (\ref{norm_func_conv})}, 
and the definition of $\widetilde{K}_{\omega_1,\omega_2,c}$, we also obtain
\begin{equation}\label{nehari_func_conv}
\widetilde{K}_{\omega_1,\omega_2,c}(W_n)-\widetilde{K}_{\omega_1,\omega_2,c}(W_n-W)-\widetilde{K}_{\omega_1,\omega_2,c}(W)\ \rightarrow\ 0\ (n\rightarrow \infty). 
\end{equation}
\\

\noindent {\bf Step\ 3}\ (Minimality of $W$ and strong convergence of $W_n$)

We first note that
\[
\|W_n\|_{\widetilde{X}_{\omega_1,\omega_2,c}}=\|V_n\|_{\widetilde{X}_{\omega_1,\omega_2,c}},\ \ N_c(W_n)=N_c(V_n)
\]
by the definition of $w_{jn}$. Therefore, we also have
\[
\widetilde{K}_{\omega_1,\omega_2,c}(W_n)=\widetilde{K}_{\omega_1,\omega_2,c}(V_n)
\]
by the definition of $\widetilde{K}_{\omega_1,\omega_2,c}$. 
Now, we assume $\widetilde{K}_{\omega_1,\omega_2,c}(W)>0$. 
Then, by {\rm (\ref{nehari_func_conv})} and $\widetilde{K}_{\omega_1,\omega_2,c}(V_n)\rightarrow 0$, 
we obtain
\[
\begin{split}
\widetilde{K}_{\omega_1,\omega_2,c}(W_n-W)
&=\widetilde{K}_{\omega_1,\omega_2,c}(W_n)-\widetilde{K}_{\omega_1,\omega_2,c}(W)
-(\widetilde{K}_{\omega_1,\omega_2,c}(W_n)-\widetilde{K}_{\omega_1,\omega_2,c}(W_n-W)-\widetilde{K}_{\omega_1,\omega_2,c}(W))\\
&\rightarrow\ -\widetilde{K}_{\omega_1,\omega_2,c}(W)<0. 
\end{split}
\]
Therefore, for large $n$, $W_n-W\ne 0$ holds and
Proposition~\ref{pos_pr_Neh} {\rm (ii)} shows that 
$\|W_n-W\|_{\widetilde{X}_{\omega_1,\omega_2,c}}^2>6\widetilde{\mathfrak{m}}_{\omega_1,\omega_2,c}$. 
This and {\rm (\ref{norm_func_conv})} imply
\[
\begin{split}
    \|W\|_{\widetilde{X}_{\omega_1,\omega_2,c}}^2
    &=\|W_n\|_{\widetilde{X}_{\omega_1,\omega_2,c}}^2-\|W_n-W\|_{\widetilde{X}_{\omega_1,\omega_2,c}}^2
    -\left(\|W_n\|_{\widetilde{X}_{\omega_1,\omega_2,c}}^2
    -\|W_n-W\|_{\widetilde{X}_{\omega_1,\omega_2,c}}^2
    -\|W\|_{\widetilde{X}_{\omega_1,\omega_2,c}}^2\right)\\
    &<\|W_n\|_{\widetilde{X}_{\omega_1,\omega_2,c}}^2-6\widetilde{\mathfrak{m}}_{\omega_1,\omega_2,c}
    -\left(\|W_n\|_{\widetilde{X}_{\omega_1,\omega_2,c}}^2
    -\|W_n-W\|_{\widetilde{X}_{\omega_1,\omega_2,c}}^2
    -\|W\|_{\widetilde{X}_{\omega_1,\omega_2,c}}^2\right)\\
    &\rightarrow 0. 
\end{split}
\]
This contradicts $W\ne 0$. 
Thus, we obtain $\widetilde{K}_{\omega_1,\omega_2,c}(W)\le 0$. 
Because $W\ne 0$, we also have $\|W\|_{\widetilde{X}_{\omega_1,\omega_2,c}}^2\ge 6\widetilde{\mathfrak{m}}_{\omega_1,\omega_2,c}$ by Proposition~\ref{pos_pr_Neh} {\rm (ii)}. 
On the other hand, the weak convergence $W_n\ \rightarrow\ W$ in $\widetilde{X}_{\omega_1,\omega_2,c}$ implies 
\[
\|W\|_{\widetilde{X}_{\omega_1,\omega_2,c}}^2\le \liminf_{n\rightarrow \infty}\|W_n\|_{\widetilde{X}_{\omega_1,\omega_2,c}}^2=6\widetilde{\mathfrak{m}}_{\omega_1,\omega_2,c}. 
\]
The contraposition of Proposition~\ref{pos_pr_Neh} {\rm (ii)} shows that $\widetilde{K}_{\omega_1,\omega_2,c}(W)\ge 0$. 
Consequently, we obtain $\|W\|_{\widetilde{X}_{\omega_1,\omega_2,c}}^2=6\widetilde{\mathfrak{m}}_{\omega_1,\omega_2,c}$
, $\widetilde{K}_{\omega_1,\omega_2,c}(W)=0$, 
and also $\widetilde{S}_{\omega_1,\omega_2,c}(W)=\widetilde{\mathfrak{m}}_{\omega_1,\omega_2,c}$ 
by {\rm (\ref{S_K_N_rel})}. 
Finally, by using {\rm (\ref{norm_func_conv})}, 
we obtain the strong convergence $\|W_n-W\|_{\widetilde{X}_{\omega_1,\omega_2,c}}\rightarrow 0$. 
\end{proof}

\section{Global well-posedness}
In this section, we prove the global well-posedness results. 
The proof of Theorem~\ref{GWP_1} is based on \cite{NP21}. 
The proof of Theorem~\ref{GWP_2} is based on \cite{FHI24} and \cite{Lipre}. 
\subsection{Proof of Theorem~\ref{GWP_1}}
Throughout this subsection, we define
\[
\begin{split}
&S_{1}(U):=S_{1,0,0}(U),\ \ S_2(U):=S_{0,1,0}(U),\\
&K_{1}(U):=K_{1,0,0}(U),\ \ K_2(U):=K_{0,1,0}(U). 
\end{split}
\]
Namely, 
\[
S_j(U)=L(U)+N(U)+Q_j(U),\ \ K_j(U)=2L(U)+3N(U)+2Q_j(U). 
\]
We also denote the ground state sets 
$\mathcal{M}_{1,0,0}$ and $\mathcal{M}_{0,1,0}$ 
by $\mathcal{M}_1$ and $\mathcal{M}_2$ respectively. 
Namely, 
\[
\mathcal{M}_j:=\{\Phi \in \mathcal{H}^1_j|\ \Phi \ne 0,\ S_j(\Phi)=\mathfrak{m}_j,\ K_j(\Phi)=0\}, 
\]
where
\[
\begin{split}
&\mathcal{H}^1_1(\R^d):=H^1(\R^d)\times \dot{H}^1(\R^d)\times H^1(\R^d),\ \ 
\mathcal{H}^1_2(\R^d):=\dot{H}^1(\R^d)\times H^1(\R^d)\times H^1(\R^d),\\
&\mathfrak{m}_j:=\inf\{S_j(\Phi)|\ \Phi\in \mathcal{H}_j^1,\ \Phi\ne 0,\ K_j(\Phi)=0\}. 
\end{split}
\]
We first mention about the relation between 
the ground state and the Gagliardo--Nirenberg type inequality. 
Note that
\[
\begin{split}
\int_{\R^d}|u_1u_2u_3|dx
&\le \|u_1\|_{L^{\frac{4d}{d+2}}}\|u_2\|_{L^{\frac{2d}{d-2}}}\|u_3\|_{L^{\frac{4d}{d+2}}}\\
&\le 
C\|u_1\|_{L^2}^{\frac{3}{2}-\frac{d}{4}}\|\nabla u_1\|_{L^2}^{\frac{d}{4}-\frac{1}{2}}\|\nabla u_2\|_{L^2}\|u_3\|_{L^2}^{\frac{3}{2}-\frac{d}{4}}\|\nabla u_3\|_{L^2}^{\frac{d}{4}-\frac{1}{2}}
\end{split}
\]
holds by the H\"older inequality and the Gagliardo--Nirenberg inequality. 
This implies
\[
|N(U)|\le C Q_1(U)^{\frac{3}{2}-\frac{d}{4}}L(U)^{\frac{d}{4}}. 
\]
By the same argument, we also have
\[
|N(U)|\le C Q_2(U)^{\frac{3}{2}-\frac{d}{4}}L(U)^{\frac{d}{4}}. 
\]
We define the spaces $\mathcal{P}_j^{\pm}$ $(j=1,2)$ by
\[
\begin{split}
\mathcal{P}_j^{\pm}&:=\{U \in \mathcal{H}^1_j(\R^d)|\ \pm N(U)<0\},
\end{split}
\]
where the double-signs correspond. 
We also define the functional $J_j$ on $\mathcal{P}_j^{+}\cup \mathcal{P}_j^{-}$ by
\[
J_j(U):=\frac{Q_j(U)^{\frac{3}{2}-\frac{d}{4}}L(U)^{\frac{d}{4}}}{|N(U)|}.
\]
We consider the minimizing problem
\begin{equation}\label{op_gag_ni_eq}
C^{\rm{op}}_j:=\inf_{U\in \mathcal{P}_j^{+}\cup \mathcal{P}_j^{-}}J_j(U). 
\end{equation}
\begin{rem}
    Note that $Q_j(-U)=Q_j(U)$, $L(-U)=L(U)$, $N(-U)=-N(U)$, and $J_j(-U)=J_j(U)$. 
    Therefore, it holds that 
    \[
    C_j^{\rm op}=\displaystyle \inf_{U\in \mathcal{P}_j^+}J_j(U)
    \]
    and 
    we can assume $N(U)<0$ without loss of generality. 
\end{rem}
\begin{rem}
    In \cite{NP21}, the minimizing problem for
    \[
    J(U):=\frac{Q(U)^{\frac{3}{2}-\frac{d}{4}}L(U)^{\frac{d}{4}}}{|N(U)|}
    \]
    is considered, where $Q(U)=Q_1(U)+Q_2(U)$. 
    By considering the minimizing problems for 
    $J_1(U)$ and $J_2(U)$ instead of $J(U)$, 
    we can improve the results {\rm (iii)} and {\rm (iv)} in Theorem~\ref{Theorem_NP}. 
\end{rem}
\begin{prop}\label{gs_op_gag}
    Let $\sigma_1,\sigma_2,\sigma_3>0$ and $3\le d\le 5$. 
    For $j=1,2$, any ground state $\Phi\in \mathcal{M}_{j}$ 
    attains the minimum value of the problem {\rm (\ref{op_gag_ni_eq})}. 
    Namely, $J_j(\Phi)=C_j^{\rm op}$ holds. 
\end{prop}
To prove Proposition~\ref{gs_op_gag}, we give the following lemma. 
\begin{lemm}\label{SJ_eq_lemm}
Let $j\in \{1,2\}$ and $\Phi \in \mathcal{M}_j$. 
If $U\in \mathcal{P}_j^+$ satisfies $L(U)+Q_j(U)=L(\Phi)+Q_j(\Phi)$, 
then we have
\[
|N(U)|\le |N(\Phi)|. 
\]
\end{lemm}
\begin{proof}
Note that 
\[
2L(\Phi)+2Q_j(\Phi)=3|N(\Phi)|
\]
holds by $K_j(\Phi)=0$. 
   For $\lambda >0$, we have
   \[
   K_j(\lambda U)=2\lambda^2L(U)+2\lambda^2Q_j(U)+3\lambda^3N(U). 
   \]
   Therefore, if we put
   \[
   \lambda_0:=\frac{2L(U)+2Q_j(U)}{3|N(U)|}=\frac{|N(\Phi)|}{|N(U)|}, 
   \]
   it holds that $K_j(\lambda_0U)=0$. 
   Because $\Phi$ is a minimizer, we obtain $S_j(\Phi)\le S_j(\lambda_0U)$. 
   On the other hand, by $K_j(\Phi)=0$ and $K_j(\lambda_0U)=0$, we have
   \[
   \begin{split}
   S_j(\Phi)&=\frac{1}{2}K_j(\Phi)-\frac{1}{3}N(\Phi)=\frac{1}{3}|N(\Phi)|,\\
   S_j(\lambda_0 U)&=\frac{1}{2}K_j(\lambda_0 U)-\frac{1}{3}N(\lambda_0 U)=\frac{\lambda_0^3}{3}|N(U)|. 
   \end{split}
   \]
   Consequently, we obtain $|N(\Phi)|\le \lambda_0^3|N(U)|$. 
   This is equivalent to $|N(U)|\le |N(\Phi)|$. 
\end{proof}
The following lemma will be used to treat the case $d=5$.
\begin{lemm}[{\rm Lemma~5.2} in {\rm \cite{Beg02}}, {\rm Lemma~3.1} in {\rm \cite{Pas15}}]\label{poly_cont_prop}
    Let $I\subset \R$ be an open interval containing $0$. 
    Let $q>1$, $b>0$, and $a\in \R$ be constants. 
    Put $\gamma =(bq)^{-\frac{1}{q-1}}$ and define a function $f$ 
    by $f(r)=a-r+br^q$ for $r\ge 0$. 
    Let $G$ be a continuous nonnegative function on $I$ 
    satisfying $f\circ G(t)\ge 0$ for any $t\in I$. 
    Assume $a<(1-\frac{1}{q})\gamma$. 
    If $G(0)<\gamma$, then $G(t)<\gamma$ holds for any $t\in I$. 
\end{lemm}
\begin{proof}[Proof of Proposition~\ref{gs_op_gag}]
Let $\Phi \in \mathcal{M}_j$. 
    For any $U \in \mathcal{P}_j^+$, 
    we put
    \[
    \widetilde{U}(x):=aU(bx),\ \ 
    a:=\left(\frac{Q_j(U)}{Q_j(\Phi)}\right)^{\frac{d-2}{4}}\left(\frac{L(\Phi)}{L(U)}\right)^{\frac{d}{4}},\ \ b:=\left(\frac{Q_j(U)L(\Phi)}{Q_j(\Phi)L(U)}\right)^{\frac{1}{2}}. 
    \]
    Then, we have
    \begin{equation}\label{J_inv_til}
    L(\widetilde{U})=a^2b^{2-d}L(U)=L(\Phi),\ \ Q_j(\widetilde{U})=a^2b^{-d}Q_j(U)=Q_j(\Phi)
    \end{equation}
    and
    \[
    N(\widetilde{U})=a^3b^{-d}N(U)=\left(\frac{Q_j(\Phi)}{Q_j(U)}\right)^{\frac{6-d}{4}}\left(\frac{L(\Phi)}{L(U)}\right)^{\frac{d}{4}}N(U). 
    \]
    Therefore, $L(\widetilde{U})+Q_j(\widetilde{U})=L(\Phi)+Q_j(\Phi)$ holds by {\rm (\ref{J_inv_til})}, and we obtain
    \[
    |N(\widetilde{U})|\le |N(\Phi)|
    \]
    by Lemma~\ref{SJ_eq_lemm}. 
    It is equivalent to
    \[
    J_j(\Phi)\le J_j(U). 
    \]
\end{proof}
\begin{proof}[Proof of {\rm Theorem~\ref{GWP_1}}]
Let $j\in \{1,2\}$ and $\Phi \in \mathcal{M}_j$. 
We first note that
\[
L(\Phi)=\frac{d}{6-d}Q_j(\Phi),\ \ N(\Phi)=-\frac{4}{6-d}Q_j(\Phi). 
\]
These identities follow from Proposition~\ref{Pohoz_id}\ {\rm (ii)} and $K_j(\Phi)=0$. 
Therefore, we get
    \[
    J_j(\Phi)=\frac{1}{4}d^{\frac{d}{4}}(6-d)^{1-\frac{d}{4}}Q_j(\Phi)^{\frac{1}{2}}. 
    \]
Furthermore, by the definition of $C_j^{\rm op}$ and Proposition~\ref{gs_op_gag}, we have
\[
|N(U)|\le \frac{1}{C_j^{\rm op}}Q_j(U)^{\frac{3}{2}-\frac{d}{4}}L(U)^{\frac{d}{4}}
=\frac{4}{d^{\frac{d}{4}}(6-d)^{1-\frac{d}{4}}}\cdot \frac{Q_j(U)^{\frac{3}{2}-\frac{d}{4}}L(U)^{\frac{d}{4}}}{Q_j(\Phi)^{\frac{1}{2}}}. 
\]
Therefore, when $d=4$, we obtain
\[
L(U)=E(U)-N(U)\le E(U_0)+\frac{Q_j(U_0)^{\frac{1}{2}}}{Q_j(\Phi)^{\frac{1}{2}}}L(U).
\]
This implies the a priori bound
\[
L(U)\le \frac{Q_j(\Phi)^{\frac{1}{2}}}{Q_j(\Phi)^{\frac{1}{2}}-Q_j(U_0)^{\frac{1}{2}}}E(U_0)
\]
if $Q_j(U_0)<Q_j(\Phi)$ holds, and local solution 
to {\rm (\ref{dnls2})} can be extended globally in time. 

On the other hand, when $d=5$, we obtain
\[
L(U)\le E(U_0)+\frac{4}{5^{\frac{5}{4}}}\frac{Q_j(U_0)^{\frac{1}{4}}}{Q_j(\Phi)^{\frac{1}{2}}}L(U)^{\frac{5}{4}}. 
\]
We put
\[
a:=E(U_0),\ b:=\frac{4}{5^{\frac{5}{4}}}\frac{Q_j(U_0)^{\frac{1}{4}}}{Q_j(\Phi)^{\frac{1}{2}}},\ 
q:=\frac{5}{4},\ 
f(r):=a-r+br^{q},\ \ G(t):=L(U(t)). 
\]
Then, $f\circ G(t)\ge 0$ and $G(t)\ge 0$ holds for any $t\in I_{\max}$. 
Furthermore, if we put $\gamma :=(bq)^{-\frac{1}{q-1}}$, then
$Q_j(U_0)E(U_0)<Q_j(\Phi)E(\Phi)$ is equivalent to $a<(1-\frac{1}{q})\gamma$, 
and $Q_j(U_0)L(U_0)<Q_j(\Phi)L(\Phi)$ is equivalent to $G(0)<\gamma$ 
because $5E(\Phi)=L(\Phi)$ and $L(\Phi)=5Q_j(\Phi)$ hold when $d=5$. 
Therefore, we obtain the a priori bound $L(U)<\gamma$ by Lemma~\ref{poly_cont_prop}, 
and the local solution 
to {\rm (\ref{dnls2})} can be extended globally in time. 
\end{proof}
\subsection{Proof of Theorem~\ref{GWP_2}}
We define the spaces $\mathcal{A}^{\pm}_{\omega_1,\omega_2,c}$ by
\[
\begin{split}
\mathcal{A}^{+}_{\omega_1,\omega_2,c}
&:=\{U \in \mathcal{H}^1(\R^d)|\ U \ne (0,0,0),\ S_{\omega_1,\omega_2,c}(U)\le \mathfrak{m}_{\omega_1,\omega_2,c},\ K_{\omega_1,\omega_2,c}(U)\ge 0\},\\
\mathcal{A}^{-}_{\omega_1,\omega_2,c}
&:=\{U \in \mathcal{H}^1(\R^d)|\ U \ne (0,0,0),\ S_{\omega_1,\omega_2,c}(U)\le \mathfrak{m}_{\omega_1,\omega_2,c},\ K_{\omega_1,\omega_2,c}(U)<0\}.
\end{split}
\]
\begin{prop}\label{Apm_inv_flow}
    Let $1\le d\le 5$, $\sigma_1,\sigma_2,\sigma_3>0$, $c\in \R^d$, and $T>0$. 
    Assume one of the conditions {\rm (A)}, {\rm (B)} in Theorem~\ref{ex_gs_1}, and ${\rm (C)'}$ in {\rm Proposition~\ref{pos_pr_Neh}} holds. 
    Then, $\mathcal{A}^{\pm}_{\omega_1,\omega_2,c}$ are invariant under the flow of {\rm (\ref{dnls2})}. 
    More precisely, if $U_0\in \mathcal{A}_{\omega_1,\omega_2,c}^{+}$\ (resp. $\in \mathcal{A}_{\omega_1,\omega_2,c}^{-}$) holds, 
    then the solution $U(t)$ to {\rm (\ref{dnls2})} with $U(0)=U_0$ on $[0,T]$, which is constructed in {\rm Theorem~\ref{Theorem_NP}\ (i)}, satisfies 
    $U(t)\in \mathcal{A}_{\omega_1,\omega_2,c}^{+}$\ (resp. $\in \mathcal{A}_{\omega_1,\omega_2,c}^{-}$) for all $t\in [0,T]$.  
\end{prop}
\begin{proof}
    We only prove the invariance for $\mathcal{A}_{\omega_1,\omega_2,c}^{+}$ 
    because we can treat $\mathcal{A}_{\omega_1,\omega_2,c}^{-}$ by the same way. 
    Since $S_{\omega_1,\omega_2,c}$ is a conserved quantity, 
    we have 
    \begin{equation}\label{S_cons_sol_m}
    S_{\omega_1,\omega_2,c}(U(t))=S_{\omega_1,\omega_2,c}(U_0)\le \mathfrak{m}_{\omega_1,\omega_2,c}
    \end{equation}
    for any $t\in [0,T]$. 
    Therefore, it is enough to show that $K_{\omega_1,\omega_2,c}(U(t))\ge 0$ 
    for any $t\in [0,T]$. 
    Assume that $K_{\omega_1,\omega_2,c}(U(\tau))< 0$ for some $\tau \in [0,T]$. 
    Then, there exists $\tau_0\in [0,\tau]$ such that $K_{\omega_1,\omega_2,c}(U(\tau_0))=0$ by the continuity of $K\circ U$ with respect to $t$. 
    This implies $\mathfrak{m}_{\omega_1,\omega_2,c}\le S_{\omega_1,\omega_2,c}(U(\tau_0))$ by the definition of $\mathfrak{m}_{\omega_1,\omega_2,c}$. 
    Since {\rm (\ref{S_cons_sol_m})} holds for $t=\tau_0$, 
    we obtain $S_{\omega_1,\omega_2,c}(U(\tau_0))=\mathfrak{m}_{\omega_1,\omega_2,c}$. 
    Namely, $U(\tau_0)\in \mathcal{M}_{\omega_1,\omega_2,c}=\mathcal{G}_{\omega_1,\omega_2,c}$ 
    by Proposition~\ref{GM_subset_rel} (see, also Remark~\ref{GM_subset_rel_rem}). 
    Write  $U(\tau_0)=(u_1,u_2,u_3)$ and define
    \[
    W(t,x):=(e^{i\omega_1(t-\tau_0)}u_1,e^{i\omega_2(t-\tau_0)}u_2,e^{i(\omega_1+\omega_2)(t-\tau_0)}u_3)
    (t,x-c(t-\tau_0)).
    \]
    Since $U(\tau_0)$ is a weak solution to the stationary problem, then $W$ is a solution to {\rm (\ref{dnls2})} with $W(\tau_0)=U(\tau_0)$. 
    By the uniqueness of the solution, 
    we have $W(t)=U(t)$ for any $t\in [\tau_0,T]$. 
    In particular, it holds that $U(\tau)=W(\tau)\in \mathcal{M}_{\omega_1,\omega_2,c}$. 
    This contradicts $K_{\omega_1,\omega_2,c}(U(\tau))<0$. 
\end{proof}
\begin{prop}\label{Ap_Gwp}
    Let $1\le d\le 5$, $\sigma_1,\sigma_2,\sigma_3>0$, and $c\in \R^d$. 
    Assume one of the conditions {\rm (A)}, {\rm (B)} in {\rm Theorem~\ref{ex_gs_1}}, and ${\rm (C)'}$ in {\rm Proposition~\ref{pos_pr_Neh}} holds. 
    If $U_0\in \mathcal{A}_{\omega_1,\omega_2,c}^{+}$ holds, 
    then the local solution $U\in C([0,T_{\max});\mathcal{H}^1(\R^d))$ to {\rm (\ref{dnls2})} with $U(0)=U_0$, 
    which is constructed in {\rm Theorem~\ref{Theorem_NP}\ (i)}, 
    can be extended globally in time, 
    where $T_{\max}$ is the maximal existence time of local solution.  
\end{prop}
\begin{proof}
    Let $U_0\in \mathcal{A}^+_{\omega_1,\omega_2,c}$. 
    Then, the solution $U(t)$ to {\rm (\ref{dnls2})} with $U(0)=U_0$ satisfies $U(t)\in \mathcal{A}^+_{\omega_1,\omega_2,c}$ for any $t\in [0,T_{\max})$ 
    by Proposition~\ref{Apm_inv_flow}. 
    Therefore, we have $S_{\omega_1,\omega_2,c}(U(t))\le \mathfrak{m}_{\omega_1,\omega_2,c}$,  
    $K_{\omega_1,\omega_2,c}(U(t))\ge 0$. This implies
    \[
    \mathfrak{m}_{\omega_1,\omega_2,c}\ge S_{\omega_1,\omega_2,c}(U(t))
    =\widetilde{S}_{\omega_1,\omega_2,c}(\Lambda_{-c}U(t))
    =\frac{1}{3}\widetilde{K}_{\omega_1,\omega_2,c}(\Lambda_{-c}U(t))+\frac{1}{6}\|\Lambda_{-c}U(t)\|_{\widetilde{X}_{\omega_1,\omega_2,c}}^2
    \]
    Because $\widetilde{K}_{\omega_1,\omega_2,c}(\Lambda_{-c}U(t))=K_{\omega_1,\omega_2,c}(U(t))\ge 0$ 
    and
    \[
    \begin{split}
    \|\Lambda_{-c}U(t)\|_{\widetilde{X}_{\omega_1,\omega_2,c}}^2
    &=\sum_{j=1}^3\frac{1}{\sigma_j}\left\|\nabla u_j(t)-i\frac{\sigma_j}{2}cu_j(t)\right\|^2_{L^2}
    +\sum_{j=1}^3\left(\omega_j-\frac{\sigma_j}{4}|c|^2\right)\|u_j(t)\|_{L^2}^2\\
    &\ge \sum_{j=1}^3\frac{1}{2\sigma_j}\|\nabla u_j(t)\|_{L^2}^2+\sum_{j=1}^3\left(\omega_j-\frac{\sigma_j}{2}|c|^2\right)\|u_j(t)\|_{L^2}^2
    \end{split}
    \]
    hold, we obtain the a priori bound
    \[
    \sum_{j=1}^3\frac{1}{\sigma_j}\|\nabla u_j(t)\|_{L^2}^2
    \lesssim \mathfrak{m}_{\omega_1,\omega_2,c}+Q_1(U_0)+Q_2(U_0)
    \]
    by the conservation law for $Q_1$ and $Q_2$. 
\end{proof}
\begin{lemm}\label{S_sca_rot}
Let $1\le d\le 5$, $\sigma_1,\sigma_2,\sigma_3>0$, and $c\in \R^d\setminus \{0\}$. 
    Assume one of the conditions {\rm (A)}, {\rm (B)} in {\rm Theorem~\ref{ex_gs_1}}, and ${\rm (C)'}$ in {\rm Proposition~\ref{pos_pr_Neh}} holds. 
    For any $e\in \R^d$ with $|e|=1$, we have
    \begin{equation}\label{act_sc_rot}
    \mathfrak{m}_{\omega_1,\omega_2,c}=|c|^{6-d}\mathfrak{m}_{\frac{\omega_1}{|c|^2},\frac{\omega_2}{|c|^2},\frac{c}{|c|}}=|c|^{6-d}\mathfrak{m}_{\frac{\omega_1}{|c|^2},\frac{\omega_2}{|c|^2},e}. 
    \end{equation}
\end{lemm}
\begin{proof}
    Let $\{\Phi_n\}\subset X_{\frac{\omega_1}{|c|^2},\frac{\omega_2}{|c|^2},\frac{c}{|c|}}$ be a minimizing sequence of $\mathfrak{m}_{\frac{\omega_1}{|c|^2},\frac{\omega_2}{|c|^2},\frac{c}{|c|}}$ with $K_{\frac{\omega_1}{|c|^2},\frac{\omega_2}{|c|^2},\frac{c}{|c|}}(\Phi_n)=0$. 
    
    To prove the first equality in {\rm (\ref{act_sc_rot})}, we put
    \[
    \Phi_{n,c}(x):=|c|^2\Phi_n(|c|x). 
    \]
    Then, we have 
    \begin{equation}\label{phi_c_sc}
    K_{\omega_1,\omega_2,c}(\Phi_{n,c})=|c|^{6-d}K_{\frac{\omega_1}{|c|^2},\frac{\omega_2}{|c|^2},\frac{c}{|c|}}(\Phi_n)=0,\ \ 
    S_{\omega_1,\omega_2,c}(\Phi_{n,c})=|c|^{6-d}S_{\frac{\omega_1}{|c|^2},\frac{\omega_2}{|c|^2},\frac{c}{|c|}}(\Phi_n). 
    \end{equation}
    Therefore, we obtain
    \[
    \mathfrak{m}_{\omega_1,\omega_2,c}\le |c|^{6-d}S_{\frac{\omega_1}{|c|^2},\frac{\omega_2}{|c|^2},\frac{c}{|c|}}(\Phi_n). 
    \]
    Because $\{\Phi_n\}$ is a minimizing sequence, 
    as $n\rightarrow \infty$, we get
    \[
    \mathfrak{m}_{\omega_1,\omega_2,c}
    \le |c|^{6-d}\mathfrak{m}_{\frac{\omega_1}{|c|^2},\frac{\omega_2}{|c|^2},\frac{c}{|c|}}. 
    \]
    By the similar argument, we also obtain
    \[
    |c|^{6-d}\mathfrak{m}_{\frac{\omega_1}{|c|^2},\frac{\omega_2}{|c|^2},\frac{c}{|c|}}
    \le \mathfrak{m}_{\omega_1,\omega_2,c}. 
    \]

    Next, we put
    \[
    \Phi_{n}^T(x):=\Phi_n(Tx), 
    \]
    where $T$ is a $d\times d$ orthogonal matrix such that $Te=\frac{c}{|c|}$.  
    Then, we have
    \[
    E(\Phi_n^T)=E(\Phi_n),\ \ Q_1(\Phi_n^T)=Q_1(\Phi_n),\ \ Q_2(\Phi_n^T)=Q_2(\Phi_n),\ \ 
    e\cdot P(\Phi_n^T)=\frac{c}{|c|}\cdot P(\Phi_n). 
    \]
    This implies
    \[
    K_{\frac{\omega_1}{|c|^2},\frac{\omega_2}{|c|^2},e}(\Phi_n^T)=K_{\frac{\omega_1}{|c|^2},\frac{\omega_2}{|c|^2},\frac{c}{|c|}}(\Phi_n)=0,\ \ 
    S_{\frac{\omega_1}{|c|^2},\frac{\omega_2}{|c|^2},e}(\Phi_n^T)=S_{\frac{\omega_1}{|c|^2},\frac{\omega_2}{|c|^2},\frac{c}{|c|}}(\Phi_n). 
    \]
    Therefore, we obtain the second equality in {\rm (\ref{act_sc_rot})} 
    by the same argument as above. 
\end{proof}
\begin{proof}[Proof of {\rm Theorem~\ref{GWP_2}}]
Let $d=4$, $\sigma_1,\sigma_2,\sigma_3>0$, $V_0\in \mathcal{H}^1(\R^4)$, 
and $U_0=\Lambda_cV_0$. 
We first consider the case $\mu=\sigma_1+\sigma_2-\sigma_3>0$. 
Then, $(\omega_1,\omega_2)=(\frac{\sigma_1}{4}|c|^2,\frac{\sigma_2}{4}|c|^2)$ 
satisfies the condition ${\rm (C)'}$ in Proposition~\ref{pos_pr_Neh}. 
For $c\in \R^4\setminus \{0\}$, we have
\[
\begin{split}
K_{\frac{\sigma_1}{4}|c|^2,\frac{\sigma_2}{4}|c|^2,c}(U_0)
&=\widetilde{K}_{\frac{\sigma_1}{4}|c|^2,\frac{\sigma_2}{4}|c|^2,c}(V_0)
=\sum_{j=1}^3\frac{1}{\sigma_j}\|\nabla v_{j,0}\|_{L^2}^2+\frac{\mu}{4}|c|^2\|v_{3,0}\|_{L^2}^2+3N_c(V_0),\\
S_{\frac{\sigma_1}{4}|c|^2,\frac{\sigma_2}{4}|c|^2,c}(U_0)
&=\widetilde{S}_{\frac{\sigma_1}{4}|c|^2,\frac{\sigma_2}{4}|c|^2,c}(V_0)
=\frac{1}{2}\sum_{j=1}^3\frac{1}{\sigma_j}\|\nabla v_{j,0}\|_{L^2}^2+\frac{\mu}{8}|c|^2\|v_{3,0}\|_{L^2}^2+N_c(V_0)
\end{split}
\]
by substituting $(\omega_1,\omega_2)=(\frac{\sigma_1}{4}|c|^2,\frac{\sigma_2}{4}|c|^2)$ 
in the definition of $\widetilde{K}_{\omega_1,\omega_2,c}$ and $\widetilde{S}_{\omega_1,\omega_2,c}$\ $(${\rm (\ref{S_phi_psi})} and {\rm (\ref{def_nehari})}$)$. 
By Lemma~\ref{S_sca_rot}, it holds that
\[
\mathfrak{m}_{\frac{\sigma_1}{4}|c|^2,\frac{\sigma_2}{4}|c|^2,c}-S_{\frac{\sigma_1}{4}|c|^2,\frac{\sigma_2}{4}|c|^2,c}(U_0)
=|c|^2\left(\mathfrak{m}_{\frac{\sigma_1}{4},\frac{\sigma_2}{4},e}-\frac{\mu}{8}\|v_{3,0}\|_{L^2}^2\right)
-\left(\frac{1}{2}\sum_{j=1}^3\frac{1}{\sigma_j}\|\nabla v_{j,0}\|_{L^2}^2+N_c(V_0)\right). 
\]
Because $\displaystyle \lim_{|c|\rightarrow \infty}N_c(V_0)=0$ 
by the Riemann--Lebesgue theorem, if 
\[
\|v_{3,0}\|_{L^2}^2<\frac{8}{\mu}\mathfrak{m}_{\frac{\sigma_1}{4},\frac{\sigma_2}{4},e}
\]
is satisfied, then we obtain
\[
K_{\frac{\sigma_1}{4}|c|^2,\frac{\sigma_2}{4}|c|^2,c}(U_0)\ge 0,\ \ 
S_{\frac{\sigma_1}{4}|c|^2,\frac{\sigma_2}{4}|c|^2,c}(U_0)\le \mathfrak{m}_{\frac{\sigma_1}{4}|c|^2,\frac{\sigma_2}{4}|c|^2,c}
\]
for sufficiently large $|c|$. 
Therefore by Proposition~\ref{Ap_Gwp}, 
the solution $U(t)$ to {\rm (\ref{dnls})} with $U(0)=U_0$ can be extended globally in time. 

Next, we consider the case $\mu=\sigma_1+\sigma_2-\sigma_3<0$. 
Then, $(\omega_1,\omega_2)=(\frac{\sigma_1}{4}|c|^2,\frac{\sigma_3-\sigma_1}{4}|c|^2)$ 
and $(\omega_1,\omega_2)=(\frac{\sigma_3-\sigma_2}{4}|c|^2,\frac{\sigma_2}{4}|c|^2)$
satisfy the condition ${\rm (C)'}$ in Proposition~\ref{pos_pr_Neh}. 
For $c\in \R^4\setminus \{0\}$, we have
\[
\begin{split}
&K_{\frac{\sigma_1}{4}|c|^2,\frac{\sigma_3-\sigma_1}{4}|c|^2,c}(U_0)
=\sum_{j=1}^3\frac{1}{\sigma_j}\|\nabla v_{j,0}\|_{L^2}^2+\frac{|\mu|}{4}|c|^2\|v_{2,0}\|_{L^2}^2+3N_c(V_0),\\
&\mathfrak{m}_{\frac{\sigma_1}{4}|c|^2,\frac{\sigma_3-\sigma_1}{4}|c|^2,c}-S_{\frac{\sigma_1}{4}|c|^2,\frac{\sigma_3-\sigma_1}{4}|c|^2,c}(U_0)\\
&=|c|^2\left(\mathfrak{m}_{\frac{\sigma_1}{4},\frac{\sigma_3-\sigma_1}{4},e}-\frac{|\mu|}{8}\|v_{2,0}\|_{L^2}^2\right)
-\left(\frac{1}{2}\sum_{j=1}^3\frac{1}{\sigma_j}\|\nabla v_{j,0}\|_{L^2}^2+N_c(V_0)\right)
\end{split}
\]
and
\[
\begin{split}
&K_{\frac{\sigma_3-\sigma_2}{4}|c|^2,\frac{\sigma_2}{4}|c|^2,c}(U_0)
=\sum_{j=1}^3\frac{1}{\sigma_j}\|\nabla v_{j,0}\|_{L^2}^2+\frac{|\mu|}{4}|c|^2\|v_{1,0}\|_{L^2}^2+3N_c(V_0),\\
&\mathfrak{m}_{\frac{\sigma_3-\sigma_2}{4}|c|^2,\frac{\sigma_2}{4}|c|^2,c}-S_{\frac{\sigma_3-\sigma_2}{4}|c|^2,\frac{\sigma_2}{4}|c|^2,c}(U_0)\\
&=|c|^2\left(\mathfrak{m}_{\frac{\sigma_3-\sigma_2}{4},\frac{\sigma_2}{4},e}-\frac{|\mu|}{8}\|v_{1,0}\|_{L^2}^2\right)
-\left(\frac{1}{2}\sum_{j=1}^3\frac{1}{\sigma_j}\|\nabla v_{j,0}\|_{L^2}^2+N_c(V_0)\right)
\end{split}
\]
by the same calculation as above. 
Therefore, if either 
\[
\|v_{1,0}\|_{L^2}^2<\frac{8}{|\mu|}\mathfrak{m}_{\frac{\sigma_3-\sigma_2}{4},\frac{\sigma_2}{4},e}\ \ {\rm or}\ \ \|v_{2,0}\|_{L^2}^2<\frac{8}{|\mu|}\mathfrak{m}_{\frac{\sigma_1}{4},\frac{\sigma_3-\sigma_1}{4},e}
\]
is satisfied, then
the solution $U(t)$ to {\rm (\ref{dnls})} with $U(0)=U_0$ 
for sufficiently large $|c|$ can be extended globally in time. 
\end{proof}

Finally in this subsection, 
we define the spaces $\mathcal{B}^{\pm}_{\omega_1,\omega_2,c}$ by
\[
\begin{split}
\mathcal{B}^{+}_{\omega_1,\omega_2,c}
&:=\{U \in \mathcal{H}^1(\R^d)|\ U \ne (0,0,0),\ S_{\omega_1,\omega_2,c}(U)\le \mathfrak{m}_{\omega_1,\omega_2,c},\ N(U)\ge -2\mathfrak{m}_{\omega_1,\omega_2,c}\},\\
\mathcal{B}^{-}_{\omega_1,\omega_2,c}
&:=\{U \in \mathcal{H}^1(\R^d)|\ U \ne (0,0,0),\ S_{\omega_1,\omega_2,c}(U)\le \mathfrak{m}_{\omega_1,\omega_2,c},\ N(U)<-2\mathfrak{m}_{\omega_1,\omega_2,c}\}.
\end{split}
\]
We recall that 
    \[
    \begin{split}
    S_{\omega_1,\omega_2,c}(U)&=L(U)+N(U)+\omega_1Q_1(U)+\omega_2Q_2(U)+c\cdot P(U),\\
    K_{\omega_1,\omega_2,c}(U)&=2L(U)+3N(U)+2\omega_1Q_1(U)+2\omega_2Q_2(U)+2c\cdot P(U)
    \end{split}
    \]
    and note that
    \begin{equation}\label{rel_N_K_S_eq}
    N(U)=K_{\omega_1,\omega_2,c}(U)-2S_{\omega_1,\omega_2,c}(U). 
    \end{equation}
The following property will be used in next section. 
\begin{prop}\label{rel_A_B_set}
     Let $1\le d\le 5$, $\sigma_1,\sigma_2,\sigma_3>0$, and $c\in \R^d$. 
    Assume one of the conditions {\rm (A)}, {\rm (B)} in {\rm Theorem~\ref{ex_gs_1}}, and ${\rm (C)'}$ in {\rm Proposition~\ref{pos_pr_Neh}} holds. 
    Then, we have $\mathcal{B}^{+}_{\omega_1,\omega_2,c}=\mathcal{A}^{+}_{\omega_1,\omega_2,c}$ and $\mathcal{B}^{-}_{\omega_1,\omega_2,c}=\mathcal{A}^{-}_{\omega_1,\omega_2,c}$. 
\end{prop} 
\begin{proof}
    
    If $U\in \mathcal{A}^{+}_{\omega_1,\omega_2,c}$, 
    then we have
    \[
    N(U)\ge -2\mathfrak{m}_{\omega_1,\omega_2,c}
    \]
    by {\rm (\ref{rel_N_K_S_eq})}. 
    This says that $U\in \mathcal{B}^{+}_{\omega_1,\omega_2,c}$. 
    Conversely, if $U\in \mathcal{B}^{+}_{\omega_1,\omega_2,c}$, 
    then we obtain $K_{\omega_1,\omega_2,c}(U)\ge 0$. 
     Indeed, if we assume $K_{\omega_1,\omega_2,c}(U)<0$, 
    then Remark~\ref{K_m_rel_U} shows that
    \[
    \begin{split}
    S_{\omega_1,\omega_2,c}(U)&=(L(U)+\omega_1Q_1(U)+\omega_2Q_2(U)+c\cdot P(U))+N(U)\\
    &>3\mathfrak{m}_{\omega_1,\omega_2,c}-2\mathfrak{m}_{\omega_1,\omega_2,c}=\mathfrak{m}_{\omega_1,\omega_2,c},
    \end{split}
    \]
which is a contradiction. 
    Therefore, it holds that $U\in \mathcal{A}^{+}_{\omega_1,\omega_2,c}$ 
    and we get $\mathcal{B}^{+}_{\omega_1,\omega_2,c}=\mathcal{A}^{+}_{\omega_1,\omega_2,c}$. 
    By the similar argument, we also obtain $\mathcal{B}^{-}_{\omega_1,\omega_2,c}=\mathcal{A}^{-}_{\omega_1,\omega_2,c}$.
\end{proof}
\section{Stability of ground state sets}
In this section, we prove Theorem~\ref{stab_thm} based on \cite{HI24}. 
Let $c\in \R^d$. 
If $c\ne 0$, we fix $(\omega_1,\omega_2)\in \Omega_c$. 
If $c=0$, we fix $(\omega_1,\omega_2)\in (\Omega_0\cup \partial\Omega_0)\setminus \{(0,0)\}$. 
Note that at least one of $\omega_1>0$ and $\omega_2>0$ holds. 
We assume $\omega_1>0$ throughout in this section. 
The case $\omega_1=0$, $\omega_2>0$, $c=0$ can be 
treated by the same way to the case $\omega_1>0$, $\omega_2=0$, $c=0$.  
By Pohozaev's identity (Proposition~\ref{Pohoz_id}) 
and the Nehari condition $K_{\omega_1,\omega_2,c}(\Phi)=0$, 
we obtain
\begin{equation}\label{m_identi}
\mathfrak{m}_{\omega_1,\omega_2,c}=\frac{1}{6-d}\left\{2\omega_1Q_1(\Phi)+2\omega_2Q_2(\Phi)+c\cdot P(\Phi)\right\}
\end{equation}
for any $\Phi\in \mathcal{M}_{\omega_1,\omega_2,c}$. 
We write $\mathfrak{m}_{\omega_1,\omega_2,c}=\mathfrak{m}(\omega_1,\omega_2,c)$ 
and define the function $h$ on $(-\infty,\sqrt{\omega_1}]$ by
\[
h(\tau):=\mathfrak{m}\left((\sqrt{\omega_1}-\tau)^2,\frac{\omega_2}{\omega_1}(\sqrt{\omega_1}-\tau)^2,\frac{c}{\sqrt{\omega_1}}(\sqrt{\omega_1}-\tau)\right). 
\]
Then, by Lemma~\ref{S_sca_rot}, it holds that
\begin{equation}\label{h_del_oiden}
    \begin{split}
        &h(0)=\mathfrak{m}(\omega_1,\omega_2,c),\\
        &h'(0)=-(6-d)\omega_1^{\frac{5-d}{2}}\mathfrak{m}\left(1,\frac{\omega_2}{\omega_1},\frac{c}{\sqrt{\omega_1}}\right)=-\frac{6-d}{\sqrt{\omega_1}}\mathfrak{m}(\omega_1,\omega_2,c),\\
        &h''(0)=(6-d)(5-d)\omega_1^{\frac{4-d}{2}}\mathfrak{m}\left(1,\frac{\omega_2}{\omega_1},\frac{c}{\sqrt{\omega_1}}\right)=\frac{(6-d)(5-d)}{\omega_1}\mathfrak{m}(\omega_1,\omega_2,c). 
    \end{split}
\end{equation}
For $3\le d\le 5$, the identities (\ref{m_identi}), (\ref{h_del_oiden}) are also true when $(\omega_2,c)=(0,0)$. 

For $\eta>0$, we define the set $\mathcal{M}_{\omega_1,\omega_2,c}^*(\eta)$ and the function $F_{\eta}$ by
\[
\begin{split}
    &\mathcal{M}_{\omega_1,\omega_2,c}^*(\eta):=\{\Phi\in \mathcal{M}_{\omega_1,\omega_2,c}|\ 
    (8-2d)\omega_1Q_1(\Phi)+(8-2d)\omega_2Q_2(\Phi)+(5-d)c\cdot P(\Phi)\ge \eta\},\\
    &F_{\eta}(\tau):=\inf_{\Phi\in \mathcal{M}_{\omega_1,\omega_2,c}^*(\eta)}\left\{\frac{h''(\tau)}{2}-Q_1(\Phi)-\frac{\omega_2}{\omega_1}Q_2(\Phi)\right\}
    =\frac{h''(\tau)}{2}-\sup_{\Phi\in \mathcal{M}_{\omega_1,\omega_2,c}^*(\eta)}\left\{Q_1(\Phi)+\frac{\omega_2}{\omega_1}Q_2(\Phi)\right\}.  
\end{split}
\]
\begin{lemm}\label{F_low_b}
    Let $1\le d\le 3$, $\sigma_1,\sigma_2,\sigma_3>0$, and $c\in \R^d$. Assume 
    \[
    \begin{cases}
        (\omega_1,\omega_2)\in \Omega_c& {\rm if}\ c\ne 0,\\
        \omega_1>0,\ \omega_2\ge 0& {\rm if}\ c=0.
        \end{cases}
        \]
    Here, $\omega_2=0$ is allowed only for $d=3$. 
    For any $\eta>0$, there exists $\tau_*=\tau_*(\omega_1,\omega_2,c,\eta)>0$ such that 
    $F_{\eta}(\tau)\ge \frac{\eta}{4\omega_1}$ holds for any $\tau \in (-\tau_*,\tau_*)$. 
\end{lemm}
\begin{proof}
    Because $h$ is a polynomial function, $F_{\eta}$ is continuous. 
    Furthermore, by (\ref{m_identi}) and (\ref{h_del_oiden}), we have
    \[
    \begin{split}
    &\frac{h''(0)}{2}-\left\{Q_1(\Phi)+\frac{\omega_2}{\omega_1}Q_2(\Phi)\right\}\\
    &=\frac{1}{2\omega_1}\left\{(8-2d)\omega_1Q_1(\Phi)+(8-2d)\omega_2Q_2(\Phi)+(5-d)c\cdot P(\Phi)\right\}
    \end{split}
    \]
    holds for any $\Phi\in \mathcal{M}_{\omega_1,\omega_2,c}$. 
    Therefore, by the definition of $\mathcal{M}_{\omega_1,\omega_2,c}^*(\eta)$ and $F_{\eta}$, 
    it holds that
    \[
    F_{\eta}(0)\ge \frac{\eta}{2\omega_1}
    \]
    and we obtain the conclusion from the continuity of $F_{\eta}$. 
\end{proof}

For $\tau_0>0$, we put
\[
\begin{split}
&\omega_{1\pm}=\omega_{1\pm}(\tau_0):=(\sqrt{\omega_1}\pm \tau_0)^2,\ \ 
\omega_{2\pm}=\omega_{2\pm}(\tau_0):=\frac{\omega_2}{\omega_1}(\sqrt{\omega_1}\pm \tau_0)^2,\\
&c_{\pm}=c_{\pm}(\tau_0):=\frac{c}{\sqrt{\omega_1}}(\sqrt{\omega_1}\pm \tau_0), 
\end{split}
\]
where the double-signs correspond. 
\begin{prop}\label{ini_set_B}
    Let $1\le d\le 3$, $\sigma_1,\sigma_2,\sigma_3>0$, and $c\in \R^d$. Assume 
    \[
    \begin{cases}
        (\omega_1,\omega_2)\in \Omega_c& {\rm if}\ c\ne 0,\\
        \omega_1>0,\ \omega_2\ge 0& {\rm if}\ c=0.
        \end{cases}
        \]
    Here, $\omega_2=0$ is allowed only for $d=3$. 
    Let $\eta >0$ and we also assume $\mathcal{M}_{\omega_1,\omega_2,c}^*(\eta)\ne \emptyset$. 
    Then for any $\tau_0\in (0,\tau_*)$, there exists $\delta =\delta(\tau_0,\omega_1,\omega_2,c,\eta)>0$ 
    such that if $U_0\in \mathcal{H}^1(\R^d)$ satisfies
    \begin{equation}\label{ini_gs_app}
    \inf_{\Phi\in \mathcal{M}_{\omega_1,\omega_2,c}^*(\eta)}\|U_0-\Phi\|_{X_{\omega_1,\omega_2,c}}<\delta,
    \end{equation}
    then it holds that $U_0\in \mathcal{B}_{\omega_{1+},\omega_{2+},c_{+}}^+\cap \mathcal{B}_{\omega_{1-},\omega_{2-},c_{-}}^-$, 
    where $\tau_*>0$ is given in Lemma~\ref{F_low_b}. 
    Furthermore, $\delta$ tends to $0$ as $\tau_0\rightarrow 0$. 
\end{prop}
\begin{rem}
    By Proposition~\ref{norm_X_H1_equi}, 
    we can assume $X_{\omega_1,\omega_2,c}=\mathcal{H}^1(\R^d)$ when $(\omega_1,\omega_2)\in \Omega_c$. 
    On the other hand, if $d=3$, $\omega_1>0$, $\omega_2=0$, and $c=0$, then $X_{\omega_1,0,0}=H^1(\R^3)\times \dot{H}^1(\R^3)\times H^1(\R^3)$.     
\end{rem}
\begin{proof}[Proof of Proposition~\ref{ini_set_B}]
    Assume $U_0\in \mathcal{H}^1(\R^d)$ satisfies (\ref{ini_gs_app}). 
    Then, there exists $\Phi_{\omega_1,\omega_2,c}\in \mathcal{M}_{\omega_1,\omega_2,c}^*(\eta)$ such that
    \begin{equation}\label{ini_gs_ap2}
    \|U_0-\Phi_{\omega_1,\omega_2,c}\|_{X_{\omega_1,\omega_2,c}}<2\delta. 
    \end{equation}
    This implies
    \begin{equation}\label{Qj_diff_est}
       |\omega_jQ_j(U_0)-\omega_jQ_j(\Phi_{\omega_1,\omega_2,c})|\
        \lesssim (\|U_0\|_{\mathcal{H}^1}+\|\Phi_{\omega_1,\omega_2,c}\|_{X_{\omega_1,\omega_2,c}})\|U_0-\Phi_{\omega_1,\omega_2,c}\|_{X_{\omega_1,\omega_2,c}}\lesssim \delta
    \end{equation}
    for $j=1, 2$ and
    \begin{equation}\label{P_diff_est}
        |c\cdot P(U_0)-c\cdot P(\Phi_{\omega_1,\omega_2,c})|\
        \lesssim (\|U_0\|_{\mathcal{H}^1}+\|\Phi_{\omega_1,\omega_2,c}\|_{X_{\omega_1,\omega_2,c}})\|U_0-\Phi_{\omega_1,\omega_2,c}\|_{X_{\omega_1,\omega_2,c}}\lesssim \delta,
    \end{equation}
    where the implicit constants depend only on 
    $\omega_1,\omega_2$, and $c$. 

    Now, we prove 
    \begin{equation}\label{S_m_ineq}
        S_{\omega_{1\pm},\omega_{2\pm},c_{\pm}}(U_0)<\mathfrak{m}(\omega_{1\pm},\omega_{2\pm},c_{\pm}). 
    \end{equation}
    Direct calculation shows that
    \[
    \begin{split}
        S_{\omega_{1\pm},\omega_{2\pm},c_{\pm}}(U_0)
    &=S_{\omega_1,\omega_2,c}(U_0)
    \pm \frac{\tau_0}{\sqrt{\omega_1}}\left(2\omega_1Q_1(U_0)+2\omega_2Q_2(U_0)+c\cdot P(U_0)\right)\\
    &\ \ \ \ \ \ \ \ \ \ \ \ \ \ \ \ \ \ \ \ \ \ \ \ \ +\frac{\tau_0^2}{\omega_1}\left(\omega_1Q_1(U_0)+\omega_2Q_2(U_0)\right). 
    \end{split}
    \]
    Therefore, by (\ref{Qj_diff_est}), (\ref{P_diff_est}), and (\ref{h_del_oiden}), we obtain
    \[
    \begin{split}
        S_{\omega_{1\pm},\omega_{2\pm},c_{\pm}}(U_0)
    &=S_{\omega_1,\omega_2,c}(\Phi_{\omega_1,\omega_2,c})
    \pm \frac{\tau_0}{\sqrt{\omega_1}}\left(2\omega_1Q_1(\Phi_{\omega_1,\omega_2,c})+2\omega_2Q_2(\Phi_{\omega_1,\omega_2,c})+c\cdot P(\Phi_{\omega_1,\omega_2,c})\right)\\
    &\ \ \ \ \ \ \ \ \ \ \ \ \ \ \ \ \ \ \ \ \ \ \ \ \ +\frac{\tau_0^2}{\omega_1}\left(\omega_1Q_1(\Phi_{\omega_1,\omega_2,c})+\omega_2Q_2(\Phi_{\omega_1,\omega_2,c})\right)+O(\delta)\\
    &=h(0)\mp \tau_0h'(0)+\frac{\tau_0^2}{\omega_1}\left(\omega_1Q_1(\Phi_{\omega_1,\omega_2,c})+\omega_2Q_2(\Phi_{\omega_1,\omega_2,c})\right)+O(\delta).
    \end{split}
    \]
    Because 
    \[
    h(\mp \tau_0)=h(0)\mp \tau_0h'(0)+\frac{\tau_0^2}{2}h''(\theta)
    \]
    holds for some $\theta\in (-\tau_0,\tau_0)\ (\subset (-\tau_*,\tau_*))$ by Taylor's theorem, 
    we have
    \[
    \begin{split}
    S_{\omega_{1\pm},\omega_{2\pm},c_{\pm}}(U_0)
    &=h(\mp \tau_0)-\tau_0^2\left\{\frac{h''(\theta)}{2}-Q_1(\Phi_{\omega_1,\omega_2,c})-\frac{\omega_2}{\omega_1}Q_2(\Phi_{\omega_1,\omega_2,c})\right\}+O(\delta). 
    \end{split}
    \]
    This implies
    \[
    S_{\omega_{1\pm},\omega_{2\pm},c_{\pm}}(U_0)\le h(\mp \tau_0)-\tau_0^2F_{\eta}(\theta)+O(\delta)
    \]
    since $\Phi_{\omega_1,\omega_2,c}\in \mathcal{M}_{\omega_1,\omega_2, c}^*(\eta)$, and we get
    \[
    S_{\omega_{1\pm},\omega_{2\pm},c_{\pm}}(U_0)\le h(\mp \tau_0)-\frac{\eta}{4\omega_1}\tau_0^2+O(\delta)
    \]
    by Lemma~\ref{F_low_b}. 
    Therefore, if we choose $\delta>0$ small enough as
    \[
    \delta \ll \frac{\eta}{4\omega_1}\tau_0^2, 
    \]
    it holds that
    \[
    S_{\omega_{1\pm},\omega_{2\pm},c_{\pm}}(U_0)<h(\mp \tau_0).
    \]
     This is equivalent to the desired estimate (\ref{S_m_ineq})
      because $h(\mp \tau_0)=\mathfrak{m}(\omega_{1\pm},\omega_{2\pm},c_{\pm})$. 

      Next, we prove
    \begin{equation}\label{N_m_ineq}
    -2\mathfrak{m}(\omega_{1+},\omega_{2+},c_{+})<N(U_0)<-2\mathfrak{m}(\omega_{1-},\omega_{2-},c_{-}).
    \end{equation}
    Because
    \[
    h(\tau)=(\sqrt{\omega_1}-\tau)^{6-d}\mathfrak{m}\left(1,\frac{\omega_2}{\omega_1},\frac{c}{\sqrt{\omega_1}}\right)
    \]
    is decreasing on $(-\infty,\sqrt{\omega_1})$, we have
    \[
    h(\tau)<h(0)<h(-\tau)
    \]
    for any $\tau \in (0,\tau_0)$. 
    Because
    \[
    h(0)=\mathfrak{m}_{\omega_1,\omega_2,c}
    =S_{\omega_1,\omega_2,c}(\Phi_{\omega_1,\omega_2,c})
    =\frac{1}{2}\left(K_{\omega_1,\omega_2,c}(\Phi_{\omega_1,\omega_2,c})-N(\Phi_{\omega_1,\omega_2,c})\right)
    =-\frac{1}{2}N(\Phi_{\omega_1,\omega_2,c})
    \]
    holds by $\Phi_{\omega_1,\omega_2,c}\in \mathcal{M}_{\omega_1,\omega_2,c}$ and the definition of $S_{\omega_1,\omega_2,c},K_{\omega_1,\omega_2,c},N$, 
    we obtain
    \[
    -2h(-\tau)<N(\Phi_{\omega_1,\omega_2,c})<-2h(\tau). 
    \]
    Note that
    \begin{equation}\label{N_diff_est}
    |N(U_0)-N(\Phi_{\omega_1,\omega_2,c})|\
        \lesssim (\|U_0\|_{\mathcal{H}^1}^2+\|\Phi_{\omega_1,\omega_2,c}\|_{X_{\omega_1,\omega_2,c}}^2)\|U_0-\Phi_{\omega_1,\omega_2,c}\|_{X_{\omega_1,\omega_2,c}}\lesssim \delta
    \end{equation}
    holds by (\ref{ini_gs_ap2}), 
    where the implicit constant depends only on $\omega_1,\omega_2$, and $c$. 
    Therefore, there exists $C=C(\omega_1,\omega_2,c)>0$ such that 
    \[
    -C\delta -2h(-\tau)<N(U_0)<C\delta-2h(\tau). 
    \]
    If we choose $\delta>0$ small enough as
    \[
    \delta<\frac{2}{C}\min\{h(\tau)-h(\tau_0),h(-\tau_0)-h(-\tau)\},
    \]
    then it holds that
    \[
    -2h(-\tau_0)<N(U_0)<-2h(\tau_0).
    \]
    This is equivalent to the desired estimate (\ref{N_m_ineq})
      because $h(\mp \tau_0)=\mathfrak{m}(\omega_{1\pm},\omega_{2\pm},c_{\pm})$. 
\end{proof}
\begin{rem}
When $\omega_2=0$, 
    we cannot obtain the estimate for $Q_2(U_0)$ from (\ref{Qj_diff_est})
    in the proof of {\rm Proposition~\ref{ini_set_B}}. 
    But $Q_2(U_0)$ does not appear in $S_{\omega_{1\pm},\omega_{2\pm},0}(U_0)$ when $\omega_2=0$, 
    and we do not need the estimate for $Q_2(U_0)$. 
    Therefore, the argument in the proof of {\rm Proposition~\ref{ini_set_B}} is also true 
    even if we assume the case $(d,\omega_2,c)=(3,0,0)$. 
    Note that (\ref{N_diff_est}) for $\omega_2=0$ can be obtained by using the estimate
    \[
    \int_{\R^d}fghdx\le \|f\|_{L^{\frac{4d}{d+2}}}\|g\|_{L^{\frac{2d}{d-2}}}\|h\|_{L^{\frac{4d}{d+2}}}
    \lesssim \|f\|_{H^1}\|g\|_{\dot{H}^1}\|h\|_{H^1}. 
    \]
\end{rem}
The following lemma is useful to prove 
the stability of $\mathcal{M}_{\omega_1,\omega_2,c}^*(\eta)$. 
\begin{lemm}\label{lim_func_low}
Let $X$ be a normed space and 
$G$ be a map from $X$ to $\R$ satisfying 
$|G(U)-G(U')|\lesssim \|U-U'\|_{X}$ for any $U,U'\in X$. 
Assume $\{U_n\}\subset X$ 
satisfies $\displaystyle \limsup_{n\rightarrow \infty}G(U_n)\ge \eta$ for $\eta>0$. 
If $U_n$ converges to $U\in X$ strongly in $X$ 
as $n\rightarrow \infty$, then it holds that $G(U)\ge \eta$. 
\end{lemm}
\begin{proof}
    Because
    \[
    |G(U_n)-G(U)|\lesssim \|U_n-U\|_{X}\ \rightarrow\ 0, 
    \]
    we have
    \[
    G(U)=\limsup_{n\rightarrow \infty}\left\{G(U_n)-(G(U_n)-G(U))\right\}
    \ge \limsup_{n\rightarrow \infty}G(U_n)\ge \eta. 
    \]
\end{proof}
\begin{prop}\label{stab_M_star}
Let $1\le d\le 3$, $\sigma_1,\sigma_2,\sigma_3>0$, and $c\in \R^d$. Assume 
    \[
    \begin{cases}
        (\omega_1,\omega_2)\in \Omega_c& {\rm if}\ c\ne 0,\\
        \omega_1>0,\ \omega_2\ge 0& {\rm if}\ c=0.
        \end{cases}
        \]
    Here, $\omega_2=0$ is allowed only for $d=3$. 
For $\eta >0$,  if $\mathcal{M}_{\omega_1,\omega_2,c}^*(\eta)\ne \emptyset$, 
then $\mathcal{M}_{\omega_1,\omega_2,c}^*(\eta)$ is orbitally stable. 
\end{prop}
\begin{proof}
For $n\in \N$, we put
\[
\omega_{1n\pm}:=\left(\sqrt{\omega_1}\pm\frac{1}{n}\right)^2,\ \ 
\omega_{2n\pm}:=\frac{\omega_2}{\omega_1}\left(\sqrt{\omega_1}\pm\frac{1}{n}\right)^2,\ \ 
c_{n\pm}:=\frac{c}{\sqrt{\omega_1}}\left(\sqrt{\omega_1}\pm\frac{1}{n}\right). 
\] 
    Let $\tau_*$ be given in Lemma~\ref{F_low_b}. 
    For any $n\in \N$ with $n>\frac{1}{\tau_*}$ (namely $\frac{1}{n}\in (0,\tau_*)$), 
    let $\delta_n:=\delta(\frac{1}{n},\omega_1,\omega_2,c,\eta)$ 
    satisfying $\delta_n\rightarrow 0$ as $n\rightarrow \infty$
    be given in Proposition~\ref{ini_set_B} as $\tau_0=\frac{1}{n}$.     
    We prove the orbital stability of $\mathcal{M}_{\omega_1,\omega_2,c}^*(\eta)$ by contradiction. 
   Note that if $U_0\in \mathcal{H}^1(\R^d)$ satisfies
\[
\inf_{\Phi\in \mathcal{M}_{\omega_1,\omega_2,c}^*(\eta)}\|U_0-\Phi\|_{X_{\omega_1,\omega_2,c}}<\delta_n, 
\]
then it holds that $U_0\in \mathcal{B}_{\omega_{1n+},\omega_{2n+},c_{n+}}^+\cap \mathcal{B}_{\omega_{1n-},\omega_{2n-},c_{n-}}^-$ by Proposition~\ref{ini_set_B}. 
Furthermore by Propositions~\ref{Ap_Gwp} 
and ~\ref{rel_A_B_set}, 
the solution $U(t)$ to {\rm (\ref{dnls2})} with $U(0)=U_0$ can be extended globally in time. 

Now, we assume that $\mathcal{M}_{\omega_1,\omega_2,c}^*(\eta)$ is unstable. 
Then, there exists $\epsilon>0$ such that the following property holds:

    For each $n\in \N$ with $n>\frac{1}{\tau_*}$, 
    there exists $U_{n,0}\in \mathcal{H}^1(\R^d)$ 
    such that
    \begin{equation}\label{U0n_phi_est}
\inf_{\Phi\in \mathcal{M}_{\omega_1,\omega_2,c}^*(\eta)}\|U_{n,0}-\Phi\|_{X_{\omega_1,\omega_2,c}}<\delta_n
\end{equation}
and the global solution $U_n(t)$ to {\rm (\ref{dnls2})} 
with $U_n(0)=U_{n,0}$ satisfies
    \begin{equation}\label{Un_phi_est}
    \inf_{\Phi\in \mathcal{M}_{\omega_1,\omega_2,c}^*(\eta)}\|U_n(t_n)-\Phi\|_{X_{\omega_1,\omega_2, c}}\ge \epsilon
    \end{equation}
    for some $t_n>0$. 
    
    Because $U_{n,0}\in \mathcal{B}_{\omega_{1n+},\omega_{2n+},c_{n+}}^+\cap \mathcal{B}_{\omega_{1n-},\omega_{2n-},c_{n-}}^-$ 
    and the sets $\mathcal{B}_{\omega_{1n\pm},\omega_{2n\pm},c_{n\pm}}^{\pm}$ 
    are invariant under the flow of {\rm (\ref{dnls2})}, 
    we obtain $U_n(t_n)\in \mathcal{B}_{\omega_{1n+},\omega_{2n+},c_{n+}}^+\cap \mathcal{B}_{\omega_{1n-},\omega_{2n-},c_{n-}}^-$. 
    Namely, it holds that
    \[
    -2h\left(-\frac{1}{n}\right)\le N(U_n(t_n))<-2h\left(\frac{1}{n}\right)
    \]
    and we obtain
    \[
    \lim_{n\rightarrow \infty}N(U_n(t_n))=-2h(0)=-2\mathfrak{m}_{\omega_1,\omega_2,c}
    \]
    by the continuity of $h$. 
    On the other hand, by (\ref{U0n_phi_est}), 
    there exists $\Phi_{n}\in \mathcal{M}_{\omega_1,\omega_2,c}^*(\eta)$
    such that
    \begin{equation}\label{uno_phin_app}
    \|U_{n,0}-\Phi_{n}\|_{X_{\omega_1,\omega_2,c}}<2\delta_n.
    \end{equation}
    Because $S_{\omega_1,\omega_2,c}$ is a conserved quantity 
    and $\Phi_{n}$ is a minimizer of $S_{\omega_1,\omega_2,c}$, 
    we obtain
    \[
    |S_{\omega_1,\omega_2,c}(U_n(t_n))-\mathfrak{m}_{\omega_1,\omega_2,c}|
    =|S_{\omega_1,\omega_2,c}(U_{n,0})-S_{\omega_1,\omega_2,c}(\Phi_{n})|
    \lesssim \delta_n\ \rightarrow\ 0\ (n\rightarrow \infty), 
    \]
    where we used the estimate such as 
    (\ref{Qj_diff_est}), (\ref{P_diff_est}), and (\ref{N_diff_est}).
    Therefore, we also have
    \[
    K_{\omega_1,\omega_2,c}(U_n(t_n))
    =2S_{\omega_1,\omega_2,c}(U_n(t_n))+N(U_n(t_n))
    \ \rightarrow\ 0\ (n\rightarrow \infty)
    \]
    and $\{U_n(t_n)\}$ becomes a minimizing sequence 
    for $\mathfrak{m}_{\omega_1,\omega_2,c}$. 
    This says that the sequence 
    $\{V_n(t_n)\}$ defined by $U_n(t_n)=\Lambda_cV_n(t_n)$ 
    becomes a minimizing sequence for $\widetilde{\mathfrak{m}}_{\omega_1,\omega_2,c}$, 
    where $\Lambda_c$ is defined in (\ref{op_lambdac_def}).
    Therefore by Proposition~\ref{mini_seq_conv}, 
    there exists $\{x_n\}\subset \R^d$ and $\widetilde{W}\in \widetilde{\mathcal{M}}_{\omega_1,\omega_2,c}$
    such that $\{\Lambda(-c\cdot x_n)V_n(t_n,\cdot -x_n)\}$ 
    has a subsequence that converges to $\widetilde{W}$ strongly in $\widetilde{X}_{\omega_1,\omega_2,c}$, 
    where, $\Lambda (-c\cdot x_n)$ is defined in (\ref{phase_tr_def_2}). 
    We put $W:=\Lambda_c\widetilde{W}$. Then $W\in X_{\omega_1,\omega_2,c}$, 
    and up to subsequence, it holds that
    \[
    \lim_{n\rightarrow \infty}\|U_n(t_n,\cdot -x_n)-W\|_{X_{\omega_1,\omega_2,c}}=0. 
    \]
    It is equivalent to 
    \begin{equation}\label{un_w_c_conv2}
    \lim_{n\rightarrow \infty}\|U_n(t_n)-W(\cdot +x_n)\|_{X_{\omega_1,\omega_2,c}}=0. 
    \end{equation}
    If $W(\cdot +x_n)\in \mathcal{M}_{\omega_1,\omega_2,c}^*(\eta)$, 
    then (\ref{un_w_c_conv2}) contradicts to (\ref{Un_phi_est}). 

    To prove $W(\cdot +x_n)\in \mathcal{M}_{\omega_1,\omega_2,c}^*(\eta)$, we define the functional $G$ by
    \[
    G(U):=(8-2d)\omega_1Q_1(U)+(8-2d)\omega_2Q_2(U)+(5-d)c\cdot P(U). 
    \]
    Clearly, $|G(U)-G(U')|\lesssim \|U-U'\|_{X_{\omega_1,\omega_2,c}}$ holds for any $U,U'\in X_{\omega_1,\omega_2,c}$. 
    Because $Q_1$, $Q_2$, and $P$ are conserved quantities, and
    invariant under the translation, we obtain
    \[
    G(U_n(t_n,\cdot -x_n))
    =G(U_n(t_n))=G(U_{n,0}).
    \]
    Therefore by (\ref{uno_phin_app}) and $\Phi_n\in \mathcal{M}_{\omega_1,\omega_2,c}^*(\eta)$, 
    we have
    \[
    \limsup_{n\rightarrow \infty} G(U_n(t_n,\cdot -x_n))
    =\limsup_{n\rightarrow \infty}\left(G(\Phi_n)+O(\delta_n)\right)\ge \eta. 
    \]
    By applying Lemma~\ref{lim_func_low}, $W\in \mathcal{M}_{\omega_1,\omega_2,c}^*(\eta)$ and
    also $W(\cdot +x_n)\in \mathcal{M}_{\omega_1,\omega_2,c}^*(\eta)$ are obtained
    since
    \[
    G(W)=G(W(\cdot +x_n)). 
    \]
\end{proof}
\begin{proof}[Proof of {\rm Theorem~\ref{stab_thm}}]
By Proposition~\ref{stab_M_star}, it suffices to show that 
there exists $c_0=c_0(\omega_1,\omega_2)>0$ such that
if $|c|\le c_0$, then 
$\mathcal{M}_{\omega_1,\omega_2,c}^*(\eta)=\mathcal{M}_{\omega_1,\omega_2,c}$ 
holds for some $\eta >0$. 
It is achieved by proving the following {\rm (i)} and {\rm (ii)}
\begin{itemize}
    \item[{\rm (i)}] For any $\omega_1,\omega_2>0$, there exists $c_1=c_1(\omega_1,\omega_2)$ such that 
     \[
     A_{\omega_1,\omega_2}:=\sup_{0<|c|\le c_1}\sup_{\Phi \in \mathcal{M}_{\omega_1,\omega_2,c}}|P(\Phi)|<\infty.
     \]
    \item[{\rm (ii)}] For any $\omega_1>0$ and $\omega_2\ge 0$, there exists $c_2=c_2(\omega_1,\omega_2)$ such that 
     \[
     B_{\omega_1,\omega_2}:=\inf_{0<|c|\le c_2}\inf_{\Phi \in \mathcal{M}_{\omega_1,\omega_2,c}}(\omega_1Q_1(\Phi)+\omega_2 Q_2(\Phi))>0.
     \]
\end{itemize}
Indeed, by putting 
\[
c_3=c_3(\omega_1,\omega_2):=\frac{(4-d)B_{\omega_1,\omega_2}}{(5-d)A_{\omega_1,\omega_2}}\in (0,\infty],\ \ 
\eta=\eta(\omega_1,\omega_2):=(4-d)B_{\omega_1,\omega_2}\in (0,\infty), 
\]
we have
\[
\begin{split}
&(8-2d)\omega_1Q_1(\Phi)+(8-2d)\omega_2Q_2(\Phi)+(5-d)c\cdot P(\Phi)\\
&\ge (8-2d)B_{\omega_1,\omega_2}-(5-d)|c|A_{\omega_1,\omega_2}\\
&\ge \eta
\end{split}
\]
for any $\Phi\in \mathcal{M}_{\omega_1,\omega_2,c}$ with $|c|\le \min\{c_1,c_2,c_3\}$. 
Note that $c_3, \eta>0$ when $1\le d\le 3$. 
As seen below, 
we can choose $C>0$ independently of $\omega_1,\omega_2$ 
such that $B_{\omega_1,\omega_2}\ge C$. 

Now, we prove {\rm (i)}. 
Let $\Phi_{\omega_1,\omega_2}\in \mathcal{M}_{\omega_1,\omega_2,0}$, 
which exists by Theorem~\ref{ex_gs_1} even if $\omega_2=0$. 
Because $K_{\omega_1,\omega_2,0}(\Phi_{\omega_1,\omega_2})=0$ 
and $S_{\omega_1,\omega_2,0}(\Phi_{\omega_1,\omega_2})=\mathfrak{m}_{\omega_1,\omega_2,0}$, 
we have $N(\Phi_{\omega_1,\omega_2})=-2\mathfrak{m}_{\omega_1,\omega_2,0}\ (<0)$ 
by (\ref{rel_N_K_S_eq}) with $c=0$. 
Therefore, by the definition of $K_{\omega_1,\omega_2,c}$, for $\lambda\in \R$, 
it holds that
\[
\begin{split}
    &K_{\omega_1,\omega_2,c}(\lambda \Phi_{\omega_1,\omega_2})\\
&=2\lambda^2L(\Phi_{\omega_1,\omega_2})+3\lambda^3N(\Phi_{\omega_1,\omega_2})+2\lambda^2\omega_1Q_1(\Phi_{\omega_1,\omega_2})+2\lambda^2\omega_2Q_2(\Phi_{\omega_1,\omega_2})+2\lambda^2c\cdot P(\Phi_{\omega_1,\omega_2})\\
    &=\lambda^2K_{\omega_1,\omega_2,0}(\Phi_{\omega_1,\omega_2})+3\lambda^2(\lambda-1)N(\Phi_{\omega_1,\omega_2})+2\lambda^2c\cdot P(\Phi_{\omega_1,\omega_2})\\
    &=2\lambda^2\{-3(\lambda -1)\mathfrak{m}_{\omega_1,\omega_2,0}+c\cdot P(\Phi_{\omega_1,\omega_2})\}. 
\end{split}
\]
We set
\[
\lambda_0:=1+\frac{1}{3\mathfrak{m}_{\omega_1,\omega_2,0}}c\cdot P(\Phi_{\omega_1,\omega_2}). 
\]
Because $|P(\Phi_{\omega_1,\omega_2})|\le \frac{3d}{2}\|\Phi_{\omega_1,\omega_2}\|_{\mathcal{H}^1}^2$ holds by the Cauchy-Schwarz inequality, 
we have $0<\lambda_0<2$ if $|c|<\frac{2\mathfrak{m}_{\omega_1,\omega_2,0}}{d}\|\Phi_{\omega_1,\omega_2}\|_{\mathcal{H}^1}^{-2}$. 
Furthermore, $K_{\omega_1,\omega_2,c}(\lambda_0\Phi_{\omega_1,\omega_2})=0$ holds and it implies
\[
\begin{split}
\mathfrak{m}_{\omega_1,\omega_2,c}
&\le S_{\omega_1,\omega_2,c}(\lambda_0\Phi_{\omega_1,\omega_2})\\
&=\lambda_0^2L(\Phi_{\omega_1,\omega_2})+\lambda_0^3N(\Phi_{\omega_1,\omega_2})+\lambda_0^2\omega_1Q_1(\Phi_{\omega_1,\omega_2})+\lambda_0^2\omega_2Q_2(\Phi_{\omega_1,\omega_2})+\lambda_0^2c\cdot P(\Phi_{\omega_1,\omega_2})\\
&=\lambda_0^3\mathfrak{m}_{\omega_1,\omega_2,0}<8\mathfrak{m}_{\omega_1,\omega_2,0}. 
\end{split}
\]
Therefore, by Remark~\ref{rel_m_tilm} and Proposition~\ref{norm_X_H1_equi}, we obtain
\[
|P(\Phi)|\lesssim \|\Phi\|_{\mathcal{H}^1}^2
\lesssim \left(\min_{j=1,2,3}\frac{\omega_j}{2}\right)^{-1}
\|\Phi\|_{X_{\omega_1,\omega_2,c}}^2\sim \mathfrak{m}_{\omega_1,\omega_2,c}
\lesssim \mathfrak{m}_{\omega_1,\omega_2,0}<\infty
\]
for any $\Phi\in \mathcal{M}_{\omega_1,\omega_2,c}$
if $0<|c|^2<{\displaystyle\min_{j=1,2,3}}\frac{2\omega_j}{\sigma_j}$, 
where the implicit constants depend only on $\omega_1,\omega_2$. 

Finally, we  prove {\rm (ii)}. 
We first consider the case $d=1$ or $2$. 
By Remark~\ref{K_m_rel_U}, at least one of
\[
2\omega_1Q_1(\Phi)+2\omega_2Q_2(\Phi)\ge C/2\ \ {\rm or}\ \ 2L(\Phi)+2c\cdot P(\Phi)\ge C/2
\]
holds for any $\Phi\in \mathcal{M}_{\omega_1,\omega_2,c}$. 
If the former estimate holds, then the proof is finished. 
Therefore, we assume the latter estimate holds. 
Then, by $K_{\omega_1,\omega_2,c}(\Phi)=0$ and Proposition~\ref{Pohoz_id}, we obtain
\[
\begin{split}
2\omega_1Q_1(\Phi)+2\omega_2Q_2(\Phi)&=\frac{2(6-d)}{d}L(\Phi)+\frac{2(3-d)}{d}c\cdot P(\Phi)\\
&\ge \frac{3-d}{d}(2L(\Phi)+2c\cdot P(\Phi))\ge \frac{3-d}{2d}C.
\end{split}
\]
Next, we consider the case $d=3$. 
By $K_{\omega_1,\omega_2,c}(\Phi)=0$ and Proposition~\ref{Pohoz_id}, we obtain
\[
2L(\Phi)=2\omega_1Q_1(\Phi)+2\omega_2Q_2(\Phi)
\]
for any $\Phi\in \mathcal{M}_{\omega_1,\omega_2,c}$. 
Therefore, we have
\[
\begin{split}
2\omega_1Q_1(\Phi)+2\omega_2Q_2(\Phi)
&=L(\Phi)+\omega_1Q_1(\Phi)+\omega_2Q_2(\Phi)\\
&=\frac{1}{2}\sum_{j=1}^3\frac{1}{\sigma_j}\|\nabla \varphi_j\|_{L^2}^2+\frac{1}{2}\sum_{j=1}^3\omega_j\|\varphi_j\|_{L^2}^2. 
\end{split}
\]
Furthermore, it holds that
\[
\begin{split}
    &\frac{1}{\sigma_j}\|\nabla \varphi_j\|_{L^2}^2+\omega_j\|\varphi_j\|_{L^2}^2\\
    &\gtrsim \frac{1}{\sigma_j}\left\|\nabla \varphi_j-i\frac{\sigma_j}{2}c\varphi_j\right\|_{L^2}^2+\left(\omega_j+\frac{\sigma_j}{4}|c|^2\right)\|\varphi_j\|_{L^2}^2
    -|c|\left\|\nabla\varphi_j-i\frac{\sigma_j}{2}c\varphi_j\right\|_{L^2}\|\varphi_j\|_{L^2}
\end{split}
\]
by the triangle inequality, and
\[
|c|\left\|\nabla\varphi_j-i\frac{\sigma_j}{2}c\varphi_j\right\|_{L^2}\|\varphi_j\|_{L^2}
\le \frac{1}{2\sigma_j}\left\|\nabla\varphi_j-i\frac{\sigma_j}{2}c\varphi_j\right\|_{L^2}^2+\frac{\sigma_j}{2}|c|^2\|\varphi_j\|_{L^2}^2
\]
by the Young inequality. 
As a result, we obtain
\[
\begin{split}
&2\omega_1Q_1(\Phi)+2\omega_2Q_2(\Phi)\ge \frac{1}{4}\|\Phi\|_{X_{\omega_1,\omega_2,c}}^2\ge C. 
\end{split}
\]
by Proposition~\ref{pos_pr_Neh} (see, also Remark~\ref{K_m_rel_U}).
\end{proof}
\appendix
\section{Regularity of solutions to the stationary problem under the condition {\rm $(C)'$}}
\begin{prop}
    Let $d=4$ or $5$, $\sigma_1,\sigma_2,\sigma_3>0$, and $c\in \R^d\setminus \{0\}$. 
    If $\omega_1=\frac{\sigma_1}{4}|c|^2$, $\omega_2=\frac{\sigma_2}{4}|c|^2$, 
    and $\omega_3> \frac{\sigma_3}{4}|c|^2$ hold, 
    then any weak solution $\Psi =(\psi_1,\psi_2,\psi_3) \in \dot{H}^1(\R^d)\times \dot{H}^1(\R^d)\times H^1(\R^d)$ to {\rm (\ref{ellip_sys2})}
    satisfies $\nabla\psi_1, \nabla \psi_2\in \bigcap_{m=0}^{\infty}H^m(\R^d)$ and $\psi_3\in \bigcap_{m=0}^{\infty}H^m(\R^d)$.  
\end{prop}
\begin{proof}
    Let $p_0\in \left[\frac{d}{d-1},\frac{d}{d-2}\right]$ and put
    \[
    q_0:=\frac{2dp_0}{2d-p_0(d-2)}. 
    \]
    Then, we have $q_0\ge 2$ and
    \[
    1\ge \frac{d}{2}-\frac{d}{q_0}. 
    \]
    Therefore, by the H\"older inequality and the Sobolev embeddings $H^1(\R^d)\hookrightarrow L^{q_0}(\R^d)$, 
    $\dot{H}^1(\R^d)\hookrightarrow L^{\frac{2d}{d-2}}(\R^d)$, it holds that
    \[
    \|\psi_1\|_{\dot{W}^{2,p_0}}=\left\|\left(-\frac{1}{\sigma_1}\Delta\right)^{-1}(e^{-i\frac{\mu}{2}c\cdot x}\overline{\psi_2}\psi_3)\right\|_{\dot{W}^{2,p_0}}
    \sim \|\psi_2\psi_3\|_{L^{p_0}}\le \|\psi_2\|_{L^{\frac{2d}{d-2}}}\|\psi_3\|_{L^{q_0}}
    \lesssim \|\psi_2\|_{\dot{H}^1}\|\psi_3\|_{H^1}. 
    \]
    It says that $\psi_1\in \dot{W}^{2,p_0}$ holds for any $p_0\in \left[\frac{d}{d-1},\frac{d}{d-2}\right]$. 
    Similarly, $\psi_2\in \dot{W}^{2,p_0}$ holds for any $p_0\in \left[\frac{d}{d-1},\frac{d}{d-2}\right]$. 
    In particular, we obtain $\psi_1,\psi_2\in \dot{H}^2(\R^4)$ 
    because $2\in \left[\frac{d}{d-1},\frac{d}{d-2}\right]$ holds when $d=4$. 

    Next, we put $r_0:=\frac{4d}{8+d}$. Note that
    \[
    2=\frac{d}{r_0}-\frac{d}{4}
    \]
    holds. By the H\"older inequality and the Sobolev embedding $\dot{W}^{2,r_0}(\R^d)\hookrightarrow L^{4}(\R^d)$, 
    we obtain
    \[
    \begin{split}
    \|\psi_3\|_{H^2}&=\left\|\left(-\frac{1}{\sigma_3}\Delta +\omega_3-\frac{\sigma_3}{4}|c|^2\right)^{-1}(e^{i\frac{\mu}{2}c\cdot x}\psi_1\psi_2)\right\|_{H^2}\\
    &\sim \|\psi_1\psi_2\|_{L^{2}}\le \|\psi_1\|_{L^{4}}\|\psi_2\|_{L^{4}}
    \lesssim \|\psi_1\|_{\dot{W}^{2,r_0}}\|\psi_2\|_{\dot{W}^{2,r_0}}. 
    \end{split}
    \]
    Because $r_0\in \left[\frac{d}{d-1},\frac{d}{d-2}\right]$ holds, 
    we have $\psi_3\in H^2(\R^d)$. 
    Furthermore, when $d=5$, 
    by the H\"older inequality and the Sobolev embeddings $H^2(\R^5)\hookrightarrow L^{10}(\R^5)$, 
    $\dot{W}^{2,\frac{5}{4}}(\R^5)\hookrightarrow L^{\frac{5}{2}}(\R^5)$, it holds that
    \[
    \begin{split}
    \|\psi_1\|_{\dot{H}^{2}}&=\left\|\left(-\frac{1}{\sigma_1}\Delta\right)^{-1}(e^{-i\frac{\mu}{2}c\cdot x}\overline{\psi_2}\psi_3)\right\|_{\dot{H}^{2}}\\
    &\sim \|\psi_2\psi_3\|_{L^{2}}\le \|\psi_2\|_{L^{\frac{5}{2}}}\|\psi_3\|_{L^{10}}
    \lesssim \|\psi_2\|_{\dot{W}^{2,\frac{5}{4}}}\|\psi_3\|_{H^2}. 
    \end{split}
    \]
    It says that $\psi_1\in \dot{H}^2(\R^5)$ because $\frac{5}{4}\in \left[\frac{d}{d-1},\frac{d}{d-2}\right]$ 
    holds when $d=5$. Similarly, $\psi_2\in \dot{H}^2(\R^5)$ holds. 

    From the above argument, we obtain $(\psi_1,\psi_2,\psi_3)\in \dot{H}^2(\R^d)\times \dot{H}^2(\R^d)\times H^2(\R^d)$
    for $d=4$ or $5$. By the induction such as in the proof of Proposition~\ref{regu_prop_ws}, 
    we have the conclusion. 
    Indeed, for $m\ge 2$, we can get $\psi_1,\psi_2\in \dot{H}^{m+1}(\R^d)$ 
    by the same argument in the proof of Proposition~\ref{regu_prop_ws}. 
    On the other hand, we have
    \[
    \begin{split}
    \|\psi_3\|_{\dot{H}^{m+1}}
    &=\left\|\left(-\frac{1}{\sigma_3}\Delta +\omega_3-\frac{\sigma_3}{4}|c|^2\right)^{-1}(e^{i\frac{\mu}{2}c\cdot x}\psi_1\psi_2)\right\|_{\dot{H}^{m+1}}\\
    &\lesssim \|\psi_1\psi_2\|_{H^{m-1}}
    \lesssim \|\psi_1\psi_2\|_{L^2}+\|(|\nabla|^{m-1}\psi_1)\psi_2\|_{L^{2}}
    +\|\psi_1|\nabla |^{m-1}\psi_2\|_{L^{2}}. 
    \end{split}
    \]
    The first term of the right-hand-side can be treated by the same way as above. 
    To control the second and third terms of the right-hand-side, 
    we use the H\"older inequality and the Sobolev embeddings $\dot{W}^{2,\frac{d}{3}}(\R^d)\hookrightarrow L^{d}(\R^d)$, $H^1(\R^d)\hookrightarrow L^{\frac{2d}{d-2}}(\R^d)$, 
    and we obtain
    \[
    \|\psi_k|\nabla |^{m-1}\psi_l\|_{L^2}\le \|\psi_k\|_{L^{d}}\||\nabla|^{m-1}\psi_l\|_{L^{\frac{2d}{d-2}}}
    \lesssim \|\psi_k\|_{\dot{W}^{2,\frac{d}{3}}}\||\nabla|^m\psi_l\|_{L^2}
    \]
    for $(k,l)=(1,2),(2,1)$. 
    It says $\psi_3\in \dot{H}^{m+1}(\R^d)$ because $\frac{d}{3}\in \left[\frac{d}{d-1},\frac{d}{d-2}\right]$ 
    holds when $d=4$ or $5$. 
\end{proof}

 \subsection*{Acknowledgement}
The first author is supported by JSPS KAKENHI Grant Number JP25K07071. 
The second author is supported by JSPS KAKENHI Grant Numbers JP25K24910 and JP23K03174. 
 
 \subsection*{Data availability}
Data sharing not applicable to this article as no datasets were generated or analyzed during the current study.

\subsection*{Conflict of interest}
The authors declare that they have no conflict of interest.



\begin{thebibliography}{30}
\bibitem{Ardila18}
Ardila. A. H, 
    {\it Orbital stability of standing waves for a system of nonlinear Schr\"odinger equations with
three wave interaction}, 
    Nonlinear Anal. {\bf 167} (2018), 1--20.

\bibitem{Beg02}
    P. B\'egout,
    {\it Necessary conditions and suﬃcient conditions for global existence in the nonlinear Schr\"odinger equation}, Adv. Math. Sci. Appl. {\bf 12} (2002), 817--827.

\bibitem{CC04}
    M. Colin and T. Colin,
    {\it On a quasilinear Zakharov system describing laser-plasma interactions}, Diff. Int. Equas., {\bf 17} (2004), 297--330.

\bibitem{CCO09}
    M. Colin, T. Colin and M. Ohta,
    {\it Stability of solitary waves for a system of nonlinear Schr\"odinger equations with three wave interaction}, Ann. Inst. H. Poincar\'e C Anal. Non Lin\'eaire., {\bf 26} (2009), 2211--2226.


\bibitem{FLYY26}
L. Forcella, X. Luo, T. Yang, and X. Yang, 
    {\it Standing waves for a Schr\"odinger system with three waves interaction}, 
    Math.Ann. {\bf 394} (2026), no.2, Paper No. 30, 43pp.

\bibitem{FHI24}
N. Fukaya, M. Hayashi and T. Inui, {\it Traveling waves for a nonlinear Schr\"odinger system with quadratic interaction},  Math. Ann. {\bf 388} (2024), no. 2, 1357–1378. 

\bibitem{Hapre}
M. Hamano, {\it Global dynamics below the ground state for the quadratic Schr\"odinger system in 5d}, 
arXiv:1805.12245. 

\bibitem{HIN21}
M. Hamano, T. Inui, and K. Nishimura, 
{\it Scattering for the quadratic nonlinear Schr\"odinger system in $\R^5$ without mass-resonance condition}, 
Funkcial. Ekvac. {\bf 64} (2021), no. 3, 261–291.

\bibitem{HM21}
M. Hamano and S. Masaki, 
{\it A sharp scattering threshold level for mass-subcritical nonlinear Schr\"odinger system}, 
Discrete Contin. Dyn. Syst. {\bf 41} (2021), no. 3, 1415–1447. 

\bibitem{HOT13}
N. Hayashi, T. Ozawa, and K. Tanaka, {\it On a system of nonlinear Schr\"odinger equations with quadratic interaction}, Ann. Inst. H. Poincar\'e Anal. Non Lin\'eaire {\bf 30} (2013), 661--690. 

\bibitem{H14}H. Hirayama, {\it Well-posedness  and scattering for a system of quadratic derivative
nonlinear Schr\"odinger equations with low regularity initial data}, Commun. Pure Appl. Anal. {\bf 13} (2014), no. 4, 1563--1591. 

\bibitem{HI24}
H. Hirayama and M. Ikeda, 
    {\it Variational problems for the system of nonlinear Schr\"odinger equations with derivative nonlinearities}, Calc. Var. Partial Differential Equations {\bf 63} (2024), no. 7, Paper No. 170, 31 pp. 

\bibitem{HK19}H. Hirayama and S. Kinoshita, 
{\itshape Sharp bilinear estimates and its application to 
a system of quadratic derivative nonlinear Schr\"odinger equations}, 
Nonlinear Anal. {\bf 178} (2019), 205--226. 

\bibitem{HKO21}
H. Hirayama, S. Kinoshita, and M. Okamoto,
{\it Well-posedness for a system of quadratic derivative nonlinear Schr\"odinger equations in almost critical spaces},
J. Math. Anal. Appl. {\bf 499} (2021), no. 2, Paper No. 125028, 29 pp.

\bibitem{IKN19}
T. Inui, N. Kishimoto, and K. Nishimura,
{\it Scattering for a mass critical NLS system below the ground state with and without mass-resonance condition},
Discrete Contin. Dyn. Syst. {\bf 39} (2019), no. 11, 6299–6353.

\bibitem{KM26}
H. Kokufukata and H. Matsuzawa, 
    {\it Ground state for a system of nonlinear Schr\"odinger equations with three waves interaction and critical nonlinearities}, 
    J. Math. Anal. Appl. {\bf 553} (2026), no.1, Paper No. 129854, 20pp.

\bibitem{KO22}
K. Kurata and Y. Osada, 
    {\it Variational problems associated with a system of nonlinear Schr\"odinger equa-
tions with three wave interaction}, 
    Discrete Contin. Dyn. syst. Ser. B {\bf 27} (2022), no.3, 1511--1547.

\bibitem{Lieb83}
E. H. Lieb, 
    {\it On the lowest eigenvalue of the Laplacian for the intersection of two domains}, 
    Invent. Math. {\bf 74} (1983), 441--448.

\bibitem{NP21}
N. Noguera and A. Pastor, 
    {\it A system of Schr\"odinger equations with general quadratic-type nonlinearities}, Commun. Contemp. Math. {\bf 23} (2021), no. 4, Paper No. 2050023, 66 pp.

\bibitem{NP21_2}
N. Noguera and A. Pastor, 
    {\it Scattering of radial solutions for quadratic-type Schr\"odinger systems in dimension five}, Discrete Contin. Dyn. Syst. {\bf 41} (2021), no. 8, 3817–3836. 

\bibitem{NP21_3}
N. Noguera and A. Pastor, 
    {\it Scattering for quadratic-type Schr\"odinger systems in dimension five without mass-resonance}, Partial Differ. Equ. Appl. {\bf 2} (2021), no. 4, Paper No. 60, 30 pp. 

\bibitem{NP21_4}
N. Noguera and A. Pastor, 
    {\it On the dynamics of a quadratic Schr\"odinger system in dimension $n=5$}, Dyn. Partial Differ. Equ. {\bf 17} (2020), no. 1, 1–17.

\bibitem{NP22}
N. Noguera and A. Pastor, 
    {\it Blow-up solutions for a system of Schr\"odinger equations with general quadratic-type nonlinearities in dimensions five and six},  Calc. Var. Partial Differential Equations {\bf 61} (2022), no. 3, Paper No. 111, 35 pp.

\bibitem{Lipre}
Y. Li,
{\it On traveling waves and global existence for a nonlinear
Schr\"odinger system with three waves interaction}, J. Math. Phys. {\bf 66}(2025), no.9, Paper No. 091511, 19pp.

\bibitem{Pas15}
A. Pastor,
{\it Weak concentration and wave operator for a 3D coupled nonlinear Schr\"odinger system}, J. Math.
Phys. {\bf 56} (2015), 021507. 

\bibitem{Pomponio10}
A. Pomponio,
{\it Ground states for a system of nonlinear Schr\"odinger equations with three wave interaction}, 
J. Math. Phys. {\bf 51} (2010), 093513. 

\bibitem{Wang17}
J. Wang,
{\it Solitary waves for coupled nonlinear elliptic system with nonhomogeneous non-linearities}, 
Calc. Var. Partial. Differ. Equ. {\bf 56} (2017), no.2, Paper No.38, 38pp. 

\end{thebibliography}

\end{document}